\documentclass[12pt]{article}
\usepackage[kerning, tracking, spacing]{microtype}
\microtypecontext{spacing=nonfrench}

\usepackage{amsmath, amsthm}
\usepackage{amssymb}
\usepackage{enumitem}
\usepackage{graphicx}
\usepackage{mathdots}
\usepackage{color}
\usepackage{hyperref,url}
\usepackage{tikz, tikz-cd}
\usetikzlibrary{matrix}
\usetikzlibrary{shapes}
\usetikzlibrary{arrows,decorations.markings, tikzmark}
\usepackage{mathtools}
\usepackage{ stmaryrd }
\usepackage[normalem]{ulem}
\usepackage[utf8]{inputenc}

 \usepackage{xcolor}
\hypersetup{
    colorlinks,
    linkcolor={red!50!black},
    citecolor={blue!50!black},
    urlcolor={blue!80!black}
}

\pgfarrowsdeclare{bad to}{bad to}
{
  \pgfarrowsleftextend{-2\pgflinewidth}
  \pgfarrowsrightextend{\pgflinewidth}
}
{
  \pgfsetlinewidth{0.8\pgflinewidth}
  \pgfsetdash{}{0pt}
  \pgfsetroundcap
  \pgfsetroundjoin
  \pgfpathmoveto{\pgfpoint{-3\pgflinewidth}{4\pgflinewidth}}
  \pgfpathcurveto
  {\pgfpoint{-2.75\pgflinewidth}{2.5\pgflinewidth}}
  {\pgfpoint{0pt}{0.25\pgflinewidth}}
  {\pgfpoint{0.75\pgflinewidth}{0pt}}
  \pgfpathcurveto
  {\pgfpoint{0pt}{-0.25\pgflinewidth}}
  {\pgfpoint{-2.75\pgflinewidth}{-2.5\pgflinewidth}}
  {\pgfpoint{-3\pgflinewidth}{-4\pgflinewidth}}
  \pgfusepathqstroke
}

\DeclareMathAlphabet{\mymathbb}{U}{BOONDOX-ds}{m}{n}

\def\Mbar{\overline{\mathcal{M}}}
\def\M{\mathfrak{M}}

\def\CH{\mathrm{CH}}
\def\CHOP{\mathrm{CH}_{\mathrm{OP}}}
\def\RCH{\overline{\mathrm{CH}}}
\def\Ms{\mathsf{M}}
\def\R{\mathrm{R}}
\def\st{\mathrm{st}}
\def\aut{\mathrm{Aut}}
\def\zero{\mymathbb{0}}
\def\one{\mymathbb{1}}
\def\CO{\mathcal{O}}
\def\A{\mathcal{A}}

\def\AA{\mathbb{A}}

\def\GG{\mathbb{G}}
\def\PP{\mathbb{P}}
\def\QQ{\mathbb{Q}}
\def\UU{\mathbb{U}}
\def\ZZ{\mathbb{Z}}
\def\TT{\mathcal{T}}

\def\StratAlg{\mathcal{S}}
\def\Rels{\mathcal{R}}
\def\WDVV{\mathrm{WDVV}}
\def\DM{\mathrm{DM}}
\def\Hom{\mathrm{Hom}}

\def\PGL{\mathrm{PGL}}

\usepackage{amsthm}
\theoremstyle{definition}
\newtheorem{definition}{Definition}[section]
\newtheorem{theorem}[definition]{Theorem}
\newtheorem{example}[definition]{Example}
\newtheorem{proposition}[definition]{Proposition}
\newtheorem{corollary}[definition]{Corollary}
\newtheorem{lemma}[definition]{Lemma}
\newtheorem{conjecture}[definition]{Conjecture}
\newtheorem{question}[definition]{Question}
\newtheorem{remark}[definition]{Remark}

\newtheorem*{conjecture*}{Conjecture}
\newtheorem*{question*}{Question}

\title{Chow rings of stacks of prestable curves II}
\author{Younghan Bae\footnote{Department of Mathematics, ETH Zurich} , Johannes Schmitt\footnote{Mathematical Institute, University of Bonn}}
\date{July 2021}

\begin{document}
\maketitle

\begin{abstract}
    We continue the study of the Chow ring of the moduli stack $\M_{g,n}$ of prestable curves begun in \cite{BS1}. In genus $0$, we show that the Chow ring of $\M_{0,n}$ coincides with the tautological ring and give a complete description in terms of (additive) generators and relations. This generalizes earlier results by Keel and Kontsevich-Manin for the spaces of stable curves.
    Our argument uses the boundary stratification of the moduli stack together with the study of the first higher Chow groups of the strata, in particular providing a new proof of the results of Kontsevich and Manin. 
    \\~\\
    \noindent 2020 Mathematics Subject Classification:\\
    \noindent 14H10, 14C15 (Primary) 14C17, 14F42 (Secondary)
\end{abstract}

\newpage

\tableofcontents

\section{Introduction}
\subsection*{The tautological ring of the moduli stack of prestable curves}
Let $\M_{g,n}$ be the moduli stack of prestable curves of genus $g$ with $n$ markings. It is a natural extention of the Deligne-Mumford space $\Mbar_{g,n}$ of stable curves. 
In the paper \cite{BS1}, we studied the rational Chow ring\footnote{We assume $(g,n)$ is different from $(1,0)$.} $\CH^*(\M_{g,n})$ and its subring
\[\R^*(\M_{g,n})\subseteq \CH^*(\M_{g,n})\]
of \emph{tautological classes}, naturally extending the corresponding notion on $\Mbar_{g,n}$. 

To describe the elements of $\R^*(\M_{g,n})$, let 
\[\pi\colon \mathfrak{C}_{g,n}\to \M_{g,n}\] be the universal curve and let $\omega_\pi$ be the relative dualizing sheaf. Let \[\sigma_i\colon\M_{g,n}\to \mathfrak{C}_{g,n}\] be the $i$-th universal section and $\mathfrak{S}_i\subset \mathfrak{C}_{g,n}$ be the corresponding divisor.  We define $\psi$ and $\kappa$-classes: given $1 \leq i \leq n$ we set
\begin{equation}\label{eqn:defpsi}
    \psi_i = c_1(\sigma_i^* \omega_\pi) \in \CH^1(\M_{g,n})\,,    
\end{equation}
and for given $m\geq 0$ we set
\begin{equation}\label{eqn:defkappa}
    \kappa_m = \pi_*\left(c_1\left(\omega_\pi\left(\sum_{i=1}^n\mathfrak{S}_i\right)\right)^{m+1} \right) \in \CH^m(\M_{g,n})\,.
\end{equation}
Let $\Gamma$ be a prestable\footnote{A prestable graph is given by the same data as a stable graph, except that one removes the condition that every vertex $v$ should be stable, i.e. satisfy $2g(v)-2+n(v)>0$.} graph in genus $g$ with $n$ markings. Each prestable graph defines a gluing map
\[\xi_\Gamma : \M_\Gamma = \prod_{v \in V(\Gamma)} \M_{g(v),n(v)} \to \M_{g,n}\]
Given any prestable graph $\Gamma$, consider the products
\begin{equation}\label{eqn:alphadecoration}
\alpha = \prod_{v\in V} \left( \prod_{i \in H(v)} \psi_{v,i}^{a_i} \prod_{a=1}^{m_v} \kappa_{v,a}^{b_{v,a}} \right)\in \CH^*(\M_\Gamma).\end{equation}
of $\psi$ and $\kappa$-classes on the space $\M_\Gamma$ above. Then we define the \emph{decorated stratum class} $[\Gamma,\alpha]$ as the pushforward
\[[\Gamma,\alpha] = (\xi_\Gamma)_* \alpha \in \R^*(\M_{g,n}).\]
\begin{definition}
The tautological ring $\R^*(\M_{g,n})$ is the $\QQ$-subspace of $\CH^*(\M_{g,n})$ additively generated by decorated strata classes.\footnote{In \cite[Definition 1.3]{BS1} the tautological ring of $\M_{g,n}$ is defined in a much more conceptual way, but we show that it is equivalent to the above presentation (\cite[Theorem 1.4]{BS1}).}
\end{definition}
The paper \cite{BS1} then develops a calculus of decorated stratum classes. Below, such results from \cite{BS1} are frequently referred.

In full generality, a description of the tautological ring $\R^*(\M_{g,n})$ is hard to approach. In this paper we specialize our attention to the moduli space of genus zero prestable curves.

\subsection*{The tautological ring in genus zero}
In Section \ref{sec:genus_zero}, we give a complete description of the Chow groups of $\M_{0,n}$ in terms of explicit generators and relations. 

For the moduli spaces  $\Mbar_{0,n}$ of stable curves, Keel \cite{keel} proved 
that the tautological ring of $\Mbar_{0,n}$ coincides with the Chow ring.
Moreover, he showed that this ring is generated as an \emph{algebra} by the boundary divisors of $\Mbar_{0,n}$ and that the \emph{ideal} of relations is generated by  the \emph{WDVV relations}, the pullbacks of the relations
\begin{equation} \label{eqn:wdvveq} 
    \begin{tikzpicture}[baseline=-3pt,label distance=0.5cm,thick,
    virtnode/.style={circle,draw,scale=0.5}, 
    nonvirt node/.style={circle,draw,fill=black,scale=0.5} ]
    \node [nonvirt node] (A) {};
    \node at (1,0) [nonvirt node] (B) {};
    \draw [-] (A) to (B);
    \node at (-.7,.5) (n1) {$1$};
    \node at (-.7,-.5) (n2) {$2$};
    \draw [-] (A) to (n1);
    \draw [-] (A) to (n2);
    
    \node at (1.7,.5) (m1) {$3$};
    \node at (1.7,-.5) (m2) {$4$};
    \draw [-] (B) to (m1);
    \draw [-] (B) to (m2);    
    \end{tikzpicture}
 = 
    \begin{tikzpicture}[baseline=-3pt,label distance=0.5cm,thick,
    virtnode/.style={circle,draw,scale=0.5}, 
    nonvirt node/.style={circle,draw,fill=black,scale=0.5} ]
    \node at (0,0) [nonvirt node] (A) {};
    \node at (1,0) [nonvirt node] (B) {};
    \draw [-] (A) to (B);
    \node at (-.7,.5) (n1) {$1$};
    \node at (-.7,-.5) (n2) {$3$};
    \draw [-] (A) to (n1);
    \draw [-] (A) to (n2);
    
    \node at (1.7,.5) (m1) {$2$};
    \node at (1.7,-.5) (m2) {$4$};
    \draw [-] (B) to (m1);
    \draw [-] (B) to (m2);    
    \end{tikzpicture}
=
    \begin{tikzpicture}[baseline=-3pt,label distance=0.5cm,thick,
    virtnode/.style={circle,draw,scale=0.5}, 
    nonvirt node/.style={circle,draw,fill=black,scale=0.5} ]
    \node at (0,0) [nonvirt node] (A) {};
    \node at (1,0) [nonvirt node] (B) {};
    \draw [-] (A) to (B);
    \node at (-.7,.5) (n1) {$1$};
    \node at (-.7,-.5) (n2) {$4$};
    \draw [-] (A) to (n1);
    \draw [-] (A) to (n2);
    
    \node at (1.7,.5) (m1) {$2$};
    \node at (1.7,-.5) (m2) {$3$};
    \draw [-] (B) to (m1);
    \draw [-] (B) to (m2);    
    \end{tikzpicture}\,
\end{equation}
in $\CH^1(\Mbar_{0,4})$ under the forgetful maps $\Mbar_{0,n} \to \Mbar_{0,4}$, together with the relations $D_1 \cdot D_2 = 0$ for $D_1, D_2$ disjoint boundary divisors.

Later, Kontsevich and Manin \cite{KontsevichManinGW, KMproduct} showed that the Chow groups of $\Mbar_{0,n}$ are generated as a \emph{$\mathbb{Q}$-vector space} by the classes of the closures of boundary strata of $\Mbar_{0,n}$. Moreover, the set of \emph{linear} relations between such strata classes are generated by the pushforwards of WDVV relations under boundary gluing maps. Our treatment of the Chow groups of $\M_{0,n}$ will be closer in spirit to the one by Kontsevich and Manin, since we provide \emph{additive} generators and relations.

\subsubsection*{Generators}
A first new phenomenon we see for $\M_{0,n}$ is that its Chow group is no longer generated by boundary strata. This comes from the fact that for $n=0,1,2$, the loci $\M_{0,n}^\textup{sm} \subset \M_{0,n}$ of smooth curves already have non-trivial Chow groups. They are given by polynomial algebras
\begin{equation} \label{eqn:nontrivChowsmooth}
    \CH^*(\M_{0,0}^\textup{sm}) = \mathbb{Q}[\kappa_2],\ \ \CH^*(\M_{0,1}^\textup{sm}) = \mathbb{Q}[\psi_1],\ \  \CH^*(\M_{0,2}^\textup{sm}) = \mathbb{Q}[\psi_1].
\end{equation}
generated by the class $\kappa_2$ on $\M_{0,0}$ and the classes $\psi_1$ on $\M_{0,1}$ and $\M_{0,2}$.\footnote{These $\kappa$ and $\psi$-classes are defined similarly to the corresponding classes on the moduli space of stable curves, see \cite[Definition 3.2]{BS1}.}
So we see that the Chow group can no longer be generated by boundary strata because all strata contained in the boundary restrict to zero on the locus $\M_{0,n}^\textup{sm}$ of smooth curves.

Instead, we prove that $\CH^*(\M_{0,n})$ is generated by strata of $\M_{0,n}$ decorated by $\kappa$ and $\psi$-classes. More precisely, the generators are indexed by the data $[\Gamma, \alpha]$, where $\Gamma$ is a prestable graph (describing the shape of the generic curve inside the boundary stratum) and $\alpha$ is a product of $\psi$-classes at vertices of $\Gamma$ with $1$ or $2$ outgoing half-edges (or $\Gamma$ is the trivial graph for $\M_{0,0}$ and $\alpha = \kappa_2^a$). We call such a class $[\Gamma, \alpha]$ a \emph{decorated stratum class in normal form}.
The allowed decorations $\alpha$ precisely reflect the non-trivial Chow groups \eqref{eqn:nontrivChowsmooth} above. We illustrate some of the generators that appear in Figure \ref{fig:tautgeneratorsg0}, for the precise construction of the corresponding classes $[\Gamma, \alpha] \in \CH^*(\M_{0,n})$ see \cite[Definition 3.3]{BS1}.

\begin{figure}[htb]
    \centering
    \[
    \begin{tikzpicture}[baseline=-3pt,label distance=0.5cm,thick,
    virtnode/.style={circle,draw,scale=0.5}, 
    nonvirt node/.style={circle,draw,fill=black,scale=0.5} ]
    \node [nonvirt node] (A) {};
    \node at (2,0) [nonvirt node] (B) {};
    \node at (4,0) [nonvirt node] (C) {};
    \draw [-] (A) to (B);
    \draw [-] (B) to (C);
    \node at (0.5, 0.4) {$\psi^3$};
    \node at (1.5, 0.4) {$\psi$};
    \node at (2.5, 0.4) {$\psi^4$};
    \end{tikzpicture}
    \quad \quad 
    \begin{tikzpicture}[baseline=-3pt,label distance=0.5cm,thick,
    virtnode/.style={circle,draw,scale=0.5}, 
    nonvirt node/.style={circle,draw,fill=black,scale=0.5} ]
    \node [nonvirt node] (A) {};
    \node at (0, 0.4) {$\kappa_2^5$};
    \end{tikzpicture}
    \]
    \[
\begin{tikzpicture}[scale=0.7, baseline=-3pt,label distance=0.3cm,thick,
    virtnode/.style={circle,draw,scale=0.5}, 
    nonvirt node/.style={circle,draw,fill=black,scale=0.5} ]
    \node at (0,0) [nonvirt node] (A) {};
    \node at (2,0) [nonvirt node] (B) {};
    \node at (-1,{sqrt(3)}) [nonvirt node] (C) {};
    \node at (-1,{-sqrt(3)}) [nonvirt node] (D) {};
    \draw [-] (A) to (B);
    \draw [-] (A) to (C);
    \draw [-] (A) to (D);
    \node at (1.5,0.4) {$\psi^2$};
    \node at (-0.3,1.4) {$\psi^5$};
    \end{tikzpicture}
    \]
    \caption{Some decorated strata classes $[\Gamma, \alpha]$ in normal form, giving generators  of $\CH^{10}(\M_{0,0})$}
    \label{fig:tautgeneratorsg0}
\end{figure}
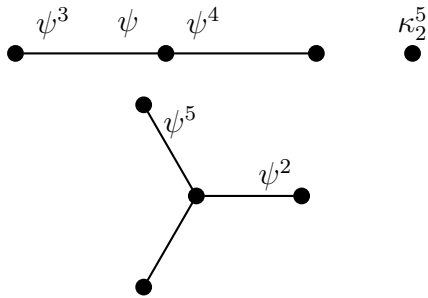

In particular, since all such classes are contained in the tautological ring, we generalize Keel's result that all Chow classes on $\Mbar_{0,n}$ are tautological.
\begin{theorem} \label{Thm:M0ngenerators}
For $n \geq 0$ we have the equality $\CH^*(\M_{0,n}) = \R^*(\M_{0,n})$.
\end{theorem}
The idea of proof for this first theorem is easy to describe: consider the excision sequence of Chow groups for the open substack $\M_{0,n}^\textup{sm} \subset \M_{0,n}$ with complement $\partial\, \M_{0,n}$:
\begin{equation} \label{eqn:excisionintro1}
    \CH^{*-1}(\partial\, \M_{0,n}) \to \CH^*(\M_{0,n}) \to \CH^*(\M_{0,n}^\textup{sm}) \to 0.
\end{equation}
From \eqref{eqn:nontrivChowsmooth} for $n=0,1,2$ and the classical statement
\begin{equation} \label{eqn:trivialChowsmooth}
\CH^*(\M_{0,n}^\textup{sm}) = \CH^*(\mathcal{M}_{0,n}) = \mathbb{Q} \cdot [\mathcal{M}_{0,n}] \text{ for }n \geq 3\end{equation}
we see that all classes in $\CH^*(\M_{0,n}^\textup{sm})$ have tautological representatives. It follows that it suffices to prove that all classes supported on $\partial\, \M_{0,n}$ are tautological. But $\partial\, \M_{0,n}$ is parameterized (via the union of finitely many gluing morphisms) by products of spaces $\M_{0,n_i}$. This allows us to set up a recursive proof.

One thing to verify in this last part of the argument is that the Chow group of a product of spaces $\M_{0,n_i}$ is generated by cycles coming from the factors $\M_{0,n_i}$. In fact, we can show more, namely that the stacks of prestable curves in genus $0$ satisfy a certain \emph{Chow-K\"unneth property}. To formulate it, we need to introduce two technical properties of locally finite type stacks $Y$: we say that $Y$ has a \emph{good filtration by finite type stacks} if $Y$ is the union of an increasing sequence $(\mathcal{U}_k)_k$ of finite type open substacks such that the codimension of the complement of $\mathcal{U}_k$ becomes arbitrarily large as $k$ increases. We say that $Y$ has a \emph{stratification by quotient stacks} if there exists a stratification of $Y$ by locally closed substacks which are each isomorphic to a global quotient of an algebraic space by a linear algebraic group. All stacks $\M_{g,n}$ for $(g,n) \neq (1,0)$ satisfy both of these properties.
\begin{proposition}[Proposition \ref{pro:M0nCKgP}, Corollary \ref{Cor:CKPM0n}] \label{Prop:M0nCKsummary}
Consider the stack $\M_{0,n}$ (for $n \geq 0$) and let $Y$ be a locally finite type stack. Then the map\footnote{Below, the notation $\alpha \times \beta$ denotes the exterior product of cycles constructed in \cite[Proposition 3.2.1]{kreschartin}.}
\[
\CH_*(\M_{0,n}) \otimes_{\mathbb{Q}} \CH_*(Y) \to \CH_*(\M_{0,n} \times Y), \alpha \otimes \beta \mapsto \alpha \times \beta
\]
is surjective if $Y$ has a good filtration by finite type stacks and a stratification by quotient stacks. 
The map is an isomorphism if $Y$ is a quotient stack.
\end{proposition}


In the proposition above, the technical conditions (like $Y$ being a quotient stack or having a stratification by quotient stacks) are currently needed since some of the results we cite in our proof have them as assumptions. We expect that these conditions can be relaxed, but do not pursue this since Proposition \ref{Prop:M0nCKsummary} is sufficient for the purpose of our paper.



\subsubsection*{Relations}
Returning to the stacks $\M_{0,n}$ themselves, we also give a full description of the set of linear relations between the generators $[\Gamma, \alpha]$ above. 
An important example is the degree one relation  
\begin{equation} \label{eqn:psi12}
   \psi_{1} + \psi_{2} = \begin{tikzpicture}[scale=0.7, baseline=-3pt,label distance=0.3cm,thick,
    virtnode/.style={circle,draw,scale=0.5}, 
    nonvirt node/.style={circle,draw,fill=black,scale=0.5} ]
    \node at (0,-.2) [nonvirt node] (A) {};
    \node at (2,-.2) [nonvirt node] (B) {};
    \draw [-] (A) to (B);
    \node at (-.7,.5) (n1) {$1$};
    \draw [-] (A) to (n1);
    \node at (2.7,.5) (m1) {$2$};
    \draw [-] (B) to (m1);
    \end{tikzpicture} \in \CH^1(\M_{0,2})
\end{equation}
on $\M_{0,2}$. 
What we can show is that \emph{all} tautological relations in genus $0$ are implied by the relation \eqref{eqn:psi12} together with the natural extension of the relation \eqref{eqn:wdvveq} to $\CH^1(\M_{0,4})$. 
\begin{theorem}[informal version, see Theorems \ref{Thm:fullgenus0rels} and \ref{Thm:WDVVnfrels}]\label{Thm:wdvv}
 For $n \geq 0$, the system of all linear relations in $\CH^*(\M_{0,n})$ between the decorated strata classes $[\Gamma, \alpha]$ of normal form  is generated by the WDVV relation (\ref{eqn:wdvveq}) on $\M_{0,4}$ and the relation (\ref{eqn:psi12}) on $\M_{0,2}$.
\end{theorem}
We give a precise description what we mean by the system of relations ``generated'' by (\ref{eqn:wdvveq}) and (\ref{eqn:psi12}) in Definition \ref{Def:relsgeneratedby}, but roughly the allowed operations are as follows:
\begin{itemize}
    \item For $n \geq 4$ we can pull back the WDVV relation \eqref{eqn:wdvveq} under the morphism $\M_{0,n} \to \M_{0,4}$ forgetting $n-4$ of the marked points.
    \item We can multiply the relation \eqref{eqn:psi12} by an arbitrary polynomial in $\psi_1, \psi_2$.
    \item Given a decorated stratum $[\Gamma_0, \alpha_0]$ in normal form, a vertex $v \in V(\Gamma_0)$ at which $\alpha_0$ is the trivial decoration and a known relation $R_0$ in the Chow group $\CH^*(\M_{0,n(v)})$ associated to the vertex $v$, we can create a new relation by gluing $R_0$ into the vertex $v$ of $[\Gamma, \alpha]$. See Example \ref{exa:gluingrelationinvertex} for an illustration.
\end{itemize}
Again, our proof strategy for Theorem \ref{Thm:wdvv} begins by looking at the excision sequence \eqref{eqn:excisionintro1}, but now extended on the left using the first higher Chow group of $\M_{0,n}^\textup{sm}$, again defined by the work of \cite{kreschartin}
\begin{equation} \label{eqn:excisionintro2}
    \CH^*(\M_{0,n}^\textup{sm}, 1) \xrightarrow{\partial} \CH^{*-1}(\partial\, \M_{0,n}) \to \CH^*(\M_{0,n}) \to \CH^*(\M_{0,n}^\textup{sm}) \to 0\,.
\end{equation}
To illustrate how we can compute tautological relations using this sequence, consider the set of prestable graphs $\Gamma_i$ with exactly one edge. The associated decorated strata $[\Gamma_i]$ are supported on $\partial \M_{0,n}$ and in fact form a basis of $\CH^0(\partial \M_{0,n})$. Then from \eqref{eqn:excisionintro2} we see that the linear relations between the classes $[\Gamma_i] \in \CH^1(\M_{0,n})$ are exactly determined by the image of $\partial$. 

For this purpose, we compute the first higher Chow groups of
the space $\M_{0,0}^\textup{sm}$ and finite products of spaces $\M_{0,n_i}^\textup{sm}$ ($n_i \geq 1$), which parametrize strata in the boundary of $\M_{0,n}$. The corresponding results are given in Propositions \ref{Pro:M0nsmCKP_higherChow} and \ref{Pro:M0nsmCKP_stable_higherChow}. The proof of Theorem \ref{Thm:wdvv} then proceeds by an inductive argument using, again, the stratification of $\M_{0,n}$ according to dual graphs.

Restricting our argument to the moduli spaces $\Mbar_{0,n}$ of stable curves, our approach to tautological relations via higher Chow groups gives a new proof that the relations between classes of strata are additively generated by boundary pushforwards of WDVV relations. As mentioned before, this result was originally stated by Kontsevich and Manin in \cite[Theorem 7.3]{KontsevichManinGW} together with a sketch of proof which was expanded in \cite{KMproduct}. 

The proof relied on Keel's result \cite{keel} that the WDVV relations generate the ideal of relations multiplicatively and thus required an explicit combinatorial analysis of the product structure of $\CH^*(\Mbar_{0,n})$. In turn, the original proof by Keel proceeded by constructing $\Mbar_{0,n}$ as an iterated blowup of $(\PP^1)^{n-3}$, carefully keeping track how the Chow group changes in each step. 

In comparison, our proof is more conceptual, since we can trace each WDVV-relation on $\Mbar_{0,n}$ to a generator of a higher Chow group $\CH^*(\mathcal{M}^\Gamma,1)$ of some stratum $\mathcal{M}^\Gamma \subseteq \Mbar_{0,n}$ of the moduli space. 
A very similar approach appears in \cite{Petersen1}, where Petersen used the mixed Hodge structure of $\mathcal{M}_{0,n}$ and the spectral sequence associated to stratification of $\Mbar_{0,n}$ to reproduce \cite{keel, KontsevichManinGW, KMproduct}.



A non-trivial consequence of our proof is the following result, stating that in codimension at least two the Chow groups of $\Mbar_{0,n}$ agree with the Chow groups of its boundary $\partial \Mbar_{0,n}$ (up to a degree shift).
\begin{corollary}[see Corollary \ref{Cor:WDVV_connecting_homomorphism}]
Let $n \geq 4$, then the inclusion $$\iota : \partial \Mbar_{0,n} \to \Mbar_{0,n}$$ of the boundary of $\Mbar_{0,n}$ induces an isomorphism
\[
\iota_* : \CH^\ell(\partial \Mbar_{0,n}) \to \CH^{\ell+1}(\Mbar_{0,n})
\]
for $\ell>0$.
\end{corollary}
This result follows easily using higher Chow groups: we have the exact sequence
\begin{equation*}
  \CH^{\ell+1}(\mathcal{M}_{0,n},1) \xrightarrow{\partial} \CH^\ell(\partial \Mbar_{0,n}) \xrightarrow{\iota_*} \CH^{\ell+1}(\Mbar_{0,n}) \to 0 \,.
\end{equation*}
Using that $\mathcal{M}_{0,n}$ can be seen as a hyperplane complement in $\AA^{n-3}$, it is easy to show that the group $\CH^{\ell+1}(\mathcal{M}_{0,n},1)$ vanishes for $\ell>0$. Thus for $\ell > 0$ the map $\iota_*$ is an isomorphism by the exact sequence. In Remark \ref{Rmk:ChowisomboundaryMbar0} we explain how, alternatively, the corollary follows from the results \cite{KontsevichManinGW, KMproduct} of Kontsevich and Manin.



\subsection*{Relation to other work}
\subsubsection*{Gromov-Witten theory}
Gromov-Witten theory studies intersection numbers on the moduli spaces $\Mbar_{g,n}(X,\beta)$ of stable maps to a nonsingular projective variety $X$. Since the spaces of stable maps admit forgetful morphisms
\begin{equation} \label{eqn:GWforgetfulmorphism}
\Mbar_{g,n}(X,\beta) \to \M_{g,n},\ (f: (C, p_1, \ldots, p_n) \to X) \mapsto (C, p_1, \ldots, p_n),
\end{equation}
results about the Chow groups of $\M_{g,n}$ can often be translated to results about Gromov-Witten invariants of \emph{arbitrary} target varieties $X$.\footnote{Note that a priori it is not possible to directly pull back classes in $\CH^*(\M_{g,n})$ under the map $\Mbar_{g,n}(X,\beta) \to \M_{g,n}$, since this map is in general neither flat nor lci. However, there exists an isomorphism $$\CH^*(\M_{g,n}) \to \CHOP^*(\M_{g,n})$$from the Chow group of $\M_{g,n}$ to its operational Chow group, and operational Chow classes are functorial under arbitrary morphisms. Then, any operational Chow class acts on the Chow group of $\Mbar_{g,n}(X,\beta)$, see Section \cite[Appendix C.]{BS1}.}

As an example, in \cite{gathmann} Gathmann used the pullback formula of $\psi$-classes along the stabilization morphism $\st\colon \M_{g,1}\to \Mbar_{g,1}$ to prove certain properties of the Gromov-Witten potential. 
Similarly, the paper \cite{LeePandharipande04} proved degree one relations on the  moduli space $\Mbar_{0,n}(\PP^N,d)$ of stable maps to a projective space and used them to reduce two pointed genus $0$ potentials to one pointed genus $0$ potentials. 
As we explain in Example \ref{Exa:LeePandharipande}, the relations used in \cite{LeePandharipande04} are the pullback of the tautological relation \eqref{eqn:psi12} on $\M_{0,2}$ and a similar relation on $\M_{0,3}$ under forgetful morphisms \eqref{eqn:GWforgetfulmorphism}.

\subsubsection*{Chow rings of open substacks of $\M_{0,n}$}
Several people have studied Chow rings with rational coefficients of open substacks of $\M_{0,n}$, and we explain how their results relate to ours.

In \cite{oesinghaus}, Oesinghaus computed the Chow rings of the loci $\M_{0,2}^\textup{ss}$ and $\M_{0,3}^\textup{ss}$ of semistable curves in $\M_{0,2}$ and $\M_{0,3}$. His proof identified the rings in terms of the known algebra of quasi-symmetric functions $\mathrm{QSym}$ (see \cite{quasisymmetric} for an overview). However, for many generators of $\mathrm{QSym}$ it remained unclear which (geometric) cycle classes on $\M_{0,2}^\textup{ss}$ and $\M_{0,3}^\textup{ss}$ they corresponded to. 
In \cite{BS1} we answered this question, identifying an additive basis of $\mathrm{QSym}$ with explicit decorated strata classes in the tautological rings of $\M_{0,2}^\textup{ss}$ and $\M_{0,3}^\textup{ss}$. 
In Example \ref{Exa:Oesinghaus} below we continue this argument by showing how Theorem \ref{Thm:M0ngenerators} and \ref{Thm:wdvv} can be used to give a new proof of Oesinghaus' results, showing that the decorated strata classes above are indeed linearly independent generators of the Chow group.

On the other hand, in \cite{fulghesurational, fulghesutaut, fulghesu3nodes} Fulghesu gave a computation of the Chow ring $\CH^*(\M_{0,3}^{\leq 3})$ of the locus $\M_{0,0}^{\leq 3}$ of curves with at most three nodes inside $\M_{0,0}$. Using a computer program, we compare his results to ours and find that our results \emph{almost} agree, except for the fact that in \cite{fulghesu3nodes} there is a missing tautological relation in the final step of the proof. This is explained in detail in Example \ref{Exa:fulghesu}.

\subsection*{Outlook and open questions}
We want to finish the introduction with a discussion of some conjectures and questions about the Chow groups of $\M_{g,n}$. 

The first concerns the relation to the Chow groups of the moduli spaces $\Mbar_{g,n}$ of stable curves. Since $\Mbar_{g,n}$ is an open substack of $\M_{g,n}$, the Chow groups of $\M_{g,n}$ determine those of $\Mbar_{g,n}$. The following conjecture would imply that the converse holds as well.
\begin{conjecture*}[Conjecture \ref{conj:pullbackinjective}] 
Let $(g,n) \neq (1,0)$, then for a fixed $d \geq 0$ there exists $m_0 \geq 0$ such that for any $m \geq m_0$, the forgetful morphism\footnote{Note that, importantly, the morphism $F_m$ does not stabilize the curve $C$, it simply forgets the last $m$ markings and returns the corresponding prestable curve.}
\[F_m : \Mbar_{g,n+m} \to \M_{g,n}, (C,p_1, \ldots, p_n, p_{n+1}, \ldots, p_{n+m}) \mapsto (C,p_1, \ldots, p_n)\]
satisfies that the pullback
\[F_m^* : \CH^d(\M_{g,n}) \to \CH^d(\Mbar_{g,n+m})\]
is injective.
\end{conjecture*}
It is easy to see that the system of morphisms $(F_m)_{m \geq 0}$ forms an atlas of $\M_{g,n}$ and that the complement of the image of $F_m$ has arbitrarily large codimension  as $m$ increases. Thus for a fixed degree $d$, the Chow groups $\CH^d(F_m(\Mbar_{g,n+m}))$ converge to $\CH^d(\M_{g,n})$, but it remains to verify that the pullback by $F_m$ indeed becomes injective. In Section \ref{Sec:forgetfulpullback} we provide some additional motivation and a number of cases $(n,d)$ in genus zero where the conjecture holds. 

Since the map $F_m^*$ sends tautological classes on $\M_{g,n}$ to tautological classes in $\Mbar_{g,n+m}$, the conjecture would also imply that knowing all tautological rings of moduli spaces of stable curves would uniquely determine the tautological rings of the stacks of prestable curves. In \cite{Pixtonconj}, Pixton proposed a set of relations between tautological classes on the moduli spaces of stable curves, proven to hold in cohomology \cite{ppz} and in Chow \cite{jandaEquivP1}, and he conjectured that these are \emph{all} tautological relations. Combined with the conjecture above, this would then determine all tautological rings of the stacks $\M_{g,n}$. It is an interesting question if Pixton's set of relations can also be generalized directly to the stacks of prestable curves to give a conjecturally complete set of relations.

Finally, recall that Theorems \ref{Thm:M0ngenerators} and \ref{Thm:wdvv} completely determine the Chow rings of $\M_{0,n}$. Given an open substack $U \subseteq \M_{0,n}$ which is a union of strata, it is easy to see that $\CH^*(U)$ is the quotient of $\CH^*(\M_{0,n})$ by the span of all tautological classes supported on the complement of $U$, so the Chow rings of such $U$ are likewise determined.

For such open substacks $U$ we can ask some more refined questions. The first concerns the structure of $\CH^*(U)$ as an algebra.

\vspace{0.4 cm}
\noindent \textbf{Question 1} (see Question \ref{quest:finitelygeneratedAlgebra})\label{quest:1}
Is it true that for $U \subset \M_{0,n}$ an open substack of finite type which is a union of strata, the Chow ring $\CH^*(U)$ is a finitely generated $\QQ$-algebra?
\vspace{0.4 cm}

Supporting evidence for this question is that it has an affirmative answer for all stacks $\M_{0,n}^\textup{sm}$ by \eqref{eqn:nontrivChowsmooth} and \eqref{eqn:trivialChowsmooth}, and 
by the computations in \cite{fulghesu3nodes}  also for the substacks $U = \M_{0,0}^{\leq e}$, $e=0,1,2,3$, of unmarked rational curves with at most $e$ nodes. Similar to the proof technique in \cite{fulghesu3nodes}, a possible approach to Question \ref{quest:1} for arbitary $U$ is to gradually enlarge $U$, adding one stratum of the moduli stack $\M_{0,n}$ at a time and showing in each step that only finitely many additional generators are necessary. 

Note that for $U$ not of finite type, Question \ref{quest:1} will have a negative answer in general: from \cite{oesinghaus} it is easy to see that the Chow ring $\CH^*(\M_{0,2}^\textup{ss})$ of the semistable locus in $\M_{0,2}$ is not finitely generated as an algebra.

Our second question concerns the Hilbert series
\[H_U = \sum_{d \geq 0} \dim_{\mathbb{Q}} \CH^d(U) t^d\]
of the Chow ring of $U$.

\vspace{0.4 cm}
\noindent \textbf{Question 2} (see Question \ref{quest:rationalHilbertSeries})\label{quest:2}
Is it true that for $U \subset \M_{0,n}$ any open substack which is a union of strata, the Hilbert series $H_U$ is the expansion of a rational function at $t=0$?
\vspace{0.4 cm}

First note that a positive answer to Question \ref{quest:1} would imply Question \ref{quest:2} for all finite type substacks $U \subset \M_{0,n}$, since the Hilbert series of a finitely generated graded algebra is a rational function, all of whose poles are at roots of unity (\cite[Theorem 13.2]{matsumura}). 
However, Question \ref{quest:2} \emph{also} has a positive answer for the non-finite type stacks $U = \M_{0,2}^\textup{ss}$ and $\M_{0,3}^\textup{ss}$ studied in \cite{oesinghaus}. 
In Figure \ref{fig:hilbertseriesoverview} we collect some examples of Hilbert series for different $U$, computed in Example \ref{Exa:fulghesu} and Section \ref{Sect:opensubstack}. Note how for $U = \M_{0,2}^\textup{ss}$ or $\M_{0,3}^\textup{ss}$ the rational function $H_U$ has poles at $1/2$, which is not a root of unity (thus giving one way to see that the Chow rings are not finitely generated).
\begin{figure}[htb]
    \begin{center}
    {\renewcommand{\arraystretch}{2}%
    \begin{tabular}{cc}
      $U$  & $H_U$ \\
      \hline$\M_0^{\leq 0}$  & $\displaystyle \frac{1}{1-t^2}$\\
      $\M_0^{\leq 1}$  & $\displaystyle \frac{1}{(1-t^2)(1-t)}$ \\
      $\M_0^{\leq 2}$  & $\displaystyle \frac{t^4+1}{(1-t^2)^2(1-t)}$ \\
      $\M_0^{\leq 3}$ & $\displaystyle \frac{t^6 + t^5 + 2 t^4 + t^3 +1}{(1-t^2)^2(1-t)(1-t^3)}$\\
      $\M_{0,2}^\textup{ss}$ & $\displaystyle \frac{1}{1-2t}$\\
      $\M_{0,3}^\textup{ss}$ & $\displaystyle \frac{(1-t)^3}{(1-2t)^3}$
    \end{tabular} }
    \end{center}
    \caption{The Hilbert series of the Chow rings of open substacks $U$ of $\M_{0,n}$}
    \label{fig:hilbertseriesoverview}
\end{figure}    


\subsection*{Structure of the paper}
In Section \ref{sec:genus_zero} we treat the Chow groups of the stacks $\M_{0,n}$ of prestable curves of genus zero. We start in Section \ref{Sect:psiandkappa} by computing the Chow groups of the loci $\M_{0,n}^\textup{sm}$ of smooth curves and explaining how (most) $\kappa$ and $\psi$-classes on $\M_{0,n}$ can be expressed in terms of cycles supported on the boundary. In Section \ref{Sect:genuszerogenerators} we show that every class in the Chow ring of $\M_{0,n}$ is tautological.  In Section \ref{Sect:HigherChowK} we compute the first higher Chow groups of the strata of $\M_{0,n}$ and use this in Section \ref{sec:tautrelations} to classify the tautological relations on $\M_{0,n}$. We finish this part of the paper by discussing the relation to earlier work in Section \ref{Sect:Relpreviouswork} and including some observations and questions about Chow groups of open substacks of $\M_{0,n}$ in Section \ref{Sect:opensubstack}.

In Section \ref{sec:comparison} we compare the Chow rings of the stacks $\M_{g,n}$ of prestable curves and the stacks $\Mbar_{g,n}$ of stable curves. We present a conjectural relation between these in Section \ref{Sec:forgetfulpullback}. We extend the known results about divisor classes on $\Mbar_{g,n}$ to $\M_{g,n}$ in Section \ref{Sect:divisors} and discuss how the study of zero cycles extends in Section \ref{Sect:zerocycles}.

Finally, Appendix \ref{Sect:Gysinpullbackhigher} summarizes a construction of a Gysin pullback for higher Chow groups following \cite{DegliseJinKhan2018, khan2019virtual}.

\subsection*{Notation}
For the convenience of the reader, we provide an overview of notations used in the paper in Table \ref{Tab:Notations}.
{\renewcommand{\arraystretch}{1.5}
\begin{table}[htb]
\begin{center}
\begin{tabular}{|c|l|}
\hline
$\M_{g,n}$ & moduli space of prestable curves   \\ 
$\M_{g,n,a}$ & moduli space of prestable curves with values in a semigroup \\
$\M_\Gamma$ & $\prod_{v\in V(\Gamma)}\M_{g(v),n(v)}$, where $\Gamma$ is a prestable graph    \\
$\M^\Gamma$ & moduli space of curves with dual graph precisely $\Gamma$ \\
$\Rels_\WDVV$ & set of WDVV relations\\
$\Rels_{\kappa, \psi}$ & set of $\psi$ and $\kappa$ relations\\
\hline
\end{tabular}
\end{center}
\caption{Notations}
\label{Tab:Notations}
\end{table}
}


\section*{Acknowledgements}
We are also grateful to Andrew Kresch, Jakob Oesinghaus and Rahul Pandharipande for many interesting conversations.
We thank Lorenzo Mantovani, Alberto Merici, Hyeonjun Park and Maria Yakerson for helpful explanations concerning higher Chow groups. The first author would like to thank Adeel Khan who helped to understand the construction of Gysin pullback for higher Chow groups.

Y.~B. was supported by ERC Grant ERC-2017-AdG-786580-MACI and Korea Foundation for Advanced Studies (KFAS).
J.~S. was supported by the SNF Early Postdoc.Mobility grant 184245 and thanks
the Max Planck Institute for Mathematics in Bonn for its hospitality.

The project has received funding from the European Research Council (ERC) under the European Union Horizon 2020 research and innovation program (grant agreement No. 786580).

\section{The Chow ring in genus \texorpdfstring{$0$}{0}}\label{sec:genus_zero}
In this section we prove Theorem \ref{Thm:M0ngenerators} and Theorem \ref{Thm:wdvv}. These results completely describe the rational Chow group of $\M_{0,n}$.
\subsection{\texorpdfstring{$\psi$}{psi} and \texorpdfstring{$\kappa$}{kappa} classes in genus \texorpdfstring{$0$}{0}} \label{Sect:psiandkappa}
In \cite{KontsevichManinGW, KMproduct}, Kontsevich and Manin described the Chow groups of $\Mbar_{0,n}$ via generators given by boundary strata and additive relations, called the {\em WDVV relations}. Their approach relies on the fact that every class on $\Mbar_{0,n}$ can be represented by boundary classes without $\psi$ or $\kappa$ classes. This is because the locus of smooth $n$ pointed rational curves $\mathcal{M}_{0,n}$ has a trivial Chow group for $n\geq 3$. 

However the Chow group of the locus of smooth curves $\M_{0,n}^\textup{sm}$ is no longer trivial when $n=0,1,2$ and hence not all tautological classes on $\M_{0,n}$ can be represented by boundary classes. We first summarize what is known about the Chow groups of $\M_{0,n}^\textup{sm}$.
\begin{lemma} \label{Lem:Chowsmooth}
For the moduli spaces of prestable curves in genus $0$ we have
\begin{enumerate}[label=(\alph*)]
\item $\mathfrak{M}_{0,0}^{\textup{sm}} = B \mathrm{PGL}_{2} $ and $\CH^*(\mathfrak{M}_{0,0}^{\textup{sm}}) = \mathbb{Q}[\kappa_{2}]$,

\item $\mathfrak{M}_{0,1}^{\textup{sm}} = B \UU$ for 
\[\UU = \left \{ \begin{bmatrix} a & b \\ 0 & d \end{bmatrix} \in \mathrm{PGL}_{2} \right\} \cong \mathbb{G}_a \rtimes \mathbb{G}_m\]
and $\CH^*(\mathfrak{M}_{0,1}^{\textup{sm}}) = \mathbb{Q}[\psi_{1}]$,

\item $\mathfrak{M}_{0,2}^{\textup{sm}} \cong B \mathbb{G}_m$ and  $\CH^*(\mathfrak{M}_{0,2}^{\textup{sm}}) = \mathbb{Q}[\psi_{1}]$,

\item $\mathfrak{M}_{0,n}^{\textup{sm}} = \mathcal{M}_{0,n}$ and $\CH^*(\mathfrak{M}_{0,n}^{\textup{sm}}) = \mathbb{Q} \cdot [\mathfrak{M}_{0,n}^{\textup{sm}}]$ for $n \geq 3$.
\end{enumerate}
\end{lemma}
\begin{proof}
The first three statements are proved in \cite{fulghesutaut}. The last statement comes from the fact that $\mathcal{M}_{0,n}$ is an open subscheme of $\mathbb{A}^{n-3}$.
\end{proof}
Note that for part (a) of the Lemma above, it is important that we work with $\mathbb{Q}$-coefficients. Indeed, the Chow groups with integral coefficients of $B \PGL_2 \cong B \mathrm{SO}(3)$ have been computed in \cite{PandharipandeEquivariant} as
\[\CH^*(B \mathrm{PGL}_{2})_{\mathbb{Z}} = \mathbb{Z}[c_1, c_2, c_3]/(c_1, 2 c_3)\,,\]
so we see that there exists a non-trivial $2$-torsion element in codimension $3$.

By Lemma \ref{Lem:Chowsmooth} we know that any monomial in $\kappa$ and $\psi$-classes on $\M_{0,n}$ can be written as a multiple of our preferred generators above (a power of $\kappa_2$ for $n=0$ or a power of $\psi_1$ for $n=1,2$) plus a contribution from the boundary.
Next we give explicit formulas how to do this. 

We start with the $\psi$-classes. For $n=0$ there is no marking and for $n=1$ the class $\psi_1$ is our preferred generator. For $n=2$ we have the following useful tautological relation.
\begin{lemma} There is a codimension one relation
\begin{equation*} \label{Lem:psi12}
   \psi_{1} + \psi_{2} =      \begin{tikzpicture}[scale=0.7, baseline=-3pt,label distance=0.3cm,thick,
    virtnode/.style={circle,draw,scale=0.5}, 
    nonvirt node/.style={circle,draw,fill=black,scale=0.5} ]
    \node at (0,-.2) [nonvirt node] (A) {};
    \node at (2,-.2) [nonvirt node] (B) {};
    \draw [-] (A) to (B);
    \node at (-.7,.7) (n1) {$1$};
    \draw [-] (A) to (n1);
    
    \node at (2.7,.7) (m1) {$2$};
    \draw [-] (B) to (m1);
    \end{tikzpicture}
\end{equation*}
in $\CH^{1}(\mathfrak{M}_{0,2})$.
\end{lemma}
\begin{proof}
For the identification $\M_{0,2}^\textup{sm} \cong B \mathbb{G}_m = [\mathrm{Spec}\, k/\mathbb{G}_m]$, the universal family over $\M_{0,2}^\textup{sm}$ is given by
\begin{equation*}
\begin{tikzcd}
 {[\mathbb{P}^1 / \mathbb{G}_m]} \arrow[d] \\
 {[\mathrm{Spec}\, k/\mathbb{G}_m]} \arrow[u, bend right, "p_2 = \infty", swap] \arrow[u, bend left, "p_1 = 0"]\,.
\end{tikzcd}    
\end{equation*}
We have that $- \psi_1, -\psi_2$ are the first Chern classes of the normal bundles of $p_1, p_2$. We have $\psi_{1} + \psi_{2} = 0$ in $\CH^1(\M_{0,2}^\textup{sm})$ because the $\mathbb{G}_m$-action on $\mathbb{P}^1$ has opposite weights at $0, \infty$. Thus, from the 
excision sequence
\[\CH^0(\partial\, \M_{0,2}) \to \CH^1(\M_{0,2}) \to \CH^1(\M_{0,2}^\textup{sm}) \to 0\,,\]
it follows that $\psi_{1} + \psi_{2}$ can be written as a linear combination of fundamental class of two boundary strata
\[\psi_1 + \psi_2 = a\, \begin{tikzpicture}[scale=0.7, baseline=-3pt,label distance=0.3cm,thick,
    virtnode/.style={circle,draw,scale=0.5}, 
    nonvirt node/.style={circle,draw,fill=black,scale=0.5} ]
    \node at (0,-.2) [nonvirt node] (A) {};
    \node at (2,-.2) [nonvirt node] (B) {};
    \draw [-] (A) to (B);
    \node at (-.5,.7) (n1) {$1$};
    \draw [-] (A) to (n1);
    
    \node at (2.5,.7) (m1) {$2$};
    \draw [-] (B) to (m1);
    \end{tikzpicture}
+ b  \,\, \begin{tikzpicture}[scale=0.7, baseline=-3pt,label distance=0.3cm,thick,
    virtnode/.style={circle,draw,scale=0.5}, 
    nonvirt node/.style={circle,draw,fill=black,scale=0.5} ]
    \node at (0,-.2) [nonvirt node] (A) {};
    \node at (2,-.2) [nonvirt node] (B) {};
    \draw [-] (A) to (B);
    \node at (1.5,.7) (n1) {$1$};
    \draw [-] (B) to (n1);
    \node at (2.5,.7) (m1) {$2$};
    \draw [-] (B) to (m1);
    \end{tikzpicture}\,.\] 
The coefficients can be determined by pulling back under the forgetful morphism $F_3 : \Mbar_{0,5} \to \M_{0,2}$. 
\end{proof}
So for $n=2$ we can express $\psi_2$ as the multiple $- \psi_1$ of our preferred generator plus a term supported in the boundary.

Now let $n\geq 3$. For $\{1,\ldots,n\}=I_1\sqcup I_2$, we denote by
\begin{equation*}
    D(I_1\mid I_2)=
\begin{tikzpicture}[baseline=-3pt,label distance=0.5cm,thick,
    virtnode/.style={circle,draw,scale=0.5}, 
    nonvirt node/.style={circle,draw,fill=black,scale=0.5} ]
    \node [nonvirt node] (A) {};
    \node at (2,0) [nonvirt node] (B) {};
    \draw [-] (A) to (B);
    \node at (-.8,0) (n0) {$I_1$};
    \node at (-.7,.5) (n1) {};
    \node at (-.7,-.5) (n2) {};
    \draw [-] (A) to (n1);
    \draw [-] (A) to (n2);
    
    \node at (2.7,.5) (m1) {};
    \node at (2.8,0) (m0) {$I_2$};
    \node at (2.7,-.5) (m2) {};
    \draw [-] (B) to (m1);
    \draw [-] (B) to (m2);    
    \end{tikzpicture} 
    \hspace{5mm} 
\end{equation*}
the class of the boundary divisor in $\CH^1(\M_{0,n})$ associated to the splitting $I_1\sqcup I_2$ of the marked points.
The following lemma shows how to write $\psi$-classes on $\M_{0,n}$ via boundary strata.
\begin{lemma}\label{lem:removePsi}
For $n\geq 3$ and $1 \leq i \leq n$, we have 
\begin{equation*}
    \psi_i= \sum_{\substack{I_1\sqcup I_2=\{1,\ldots,n\}\\i\in I_1; j,\ell\in I_2}} D(I_1\mid I_2)
\end{equation*}
in $\CH^1(\M_{0,n})$ for any choice of $1\leq j,\ell \neq i\leq n$.
\end{lemma}
\begin{proof}
When $n=3$, we have $\CH^1(\M_{0,3}^\textup{sm})=0$, so $\psi_i$ can be written as a linear combination of the four boundary divisors of $\M_{0,3}$. Again, the coefficients can be determined via the pullback under $F_2 : \Mbar_{0,5} \to \M_{0,3}$.
The relation for $n\geq 4$ follows by pulling back the relation on $\M_{0,3}$ via the morphism $F \colon \M_{0,n}\to \M_{0,3}$ forgetting all markings except $\{i,j,\ell\}$. This pullback can be computed via \cite[Corollary 3.9]{BS1} .
\end{proof}

For the $\kappa$-classes on $\M_{0,n}$ we have the following boundary expressions.
\begin{lemma}\label{Lem:remove_kappa}
Let $a$ be a nonnegative integer and consider $\kappa_a \in \CH^a(\M_{0,n})$.
\begin{enumerate}[label=(\alph*)] 
    \item When $n \geq 1$, the class $\kappa_a$ can be written as a linear combination of  monomials in $\psi$-classes  and boundary classes $[\Gamma_i, \alpha_i]$ for nontrivial prestable graphs $\Gamma_i$.
    \item When $n=0$, the class $\kappa_a$ can be written as a linear combination of  monomials in $\kappa_2$ and $\psi$-classes and boundary classes $[\Gamma_i, \alpha_i]$ for nontrivial prestable graphs $\Gamma_i$.
\end{enumerate}
\end{lemma}
\begin{proof}
It is enough to prove the corresponding statement on $\M_{0,n,\one}$ because the restriction to the open substack $\M_{0,n} \subset \M_{0,n,\one}$ does not create additional $\kappa$ classes.

(a) In the calculation below, we use the notion of tautological classes on the moduli stack of $\A$-valued prestable curves when $\A$ is a semigroup with two elements $\{\zero,\one\}$. This stack parametrizes prestable curves with additional decoration of $a_v\in \A$ at each component $v$. This technique is useful for computing tautological classes on $\M_{g,n}$ involving $\kappa$-classes. We refer \cite[Section 2.2]{BS1} for details. 

Consider the universal curve
\begin{equation*}
    \pi\colon \M_{0,n+1,\one}\to \M_{0,n,\one} 
\end{equation*}
so that $\kappa_a = \pi_* (\psi_{n+1}^{a+1})$. We prove the claim by induction on $a$, where the induction start $a=0$ is trivial since $\kappa_0=n-2$. 

If $n \geq 2$ and $a \geq 1$, we claim that $\psi_{n+1} \in \CH^1(\M_{0,n+1,\one})$ can be written as a sum of boundary divisors. Indeed, by Lemma \ref{lem:removePsi} this is true for $\psi_{n+1} \in \CH^1(\Mbar_{0,n})$ and so the statement on $\Mbar_{0,n+1,\one}$ follows by pulling back under the forgetful map $F_\A : \Mbar_{0,n+1,\one} \to \Mbar_{0,n+1}$ of $\A$-values using \cite[Proposition 3.12]{BS1}. Thus replacing one of the factors $\psi_{n+1}$ in $\psi_{n+1}^{a+1}$ with this boundary expression, we get a sum of boundary divisors in $\Mbar_{0,n+1,\one}$ decorated with $\psi_{n+1}^{a}$. 
After pushing forward to $\M_{0,n,1}$, this class can be written as a tautological class without $\kappa$ class by the induction hypothesis.

When $n=1$, we also conclude by induction on $a$. By Lemma \ref{Lem:psi12}, we have
\begin{equation*}
    \kappa_{a} = \pi_*(\psi_{2}^{a+1}) = \pi_*\Big(-\psi_2^{a}\psi_1 +
    \begin{tikzpicture}[scale=0.7, baseline=-3pt,label distance=0.3cm,thick,
    virtnode/.style={circle,draw,scale=0.5}, 
    nonvirt node/.style={circle,draw,fill=black,scale=0.5} ]
    \node at (0,-.5) [nonvirt node] (A) {};
    \node at (2,-.5) [nonvirt node] (B) {};
    \draw [-] (A) to (B);
    \node at (-.7,.5) (n1) {$1$};
    \draw [-] (A) to (n1);
    \node at (2.7,.5) (m1) {$\psi_2^{a}$};
    \draw [-] (B) to (m1);
    \end{tikzpicture}
    \Big)\,,
\end{equation*}
where implicitly we sum over all $\mathcal{A}$-valued graph where the sum of degrees is equal to $\one$. By \cite[Proposition 3.10]{BS1} we have
\[\psi_1 = \pi^*\psi_1 + 
 \begin{tikzpicture}[scale=0.7, baseline=-3pt,label distance=0.3cm,thick,
    virtnode/.style={circle,draw,scale=0.5}, 
    nonvirt node/.style={circle,draw,fill=black,scale=0.5} ]
    \node [nonvirt node,label=below:$(0{,}\one)$] (A) {};
    \node at (2,0) [virtnode,label=below:$(0{,}\zero)$] (B) {};
    \draw [-] (A) to (B);
    \node at (3,.5) (m1) {$1$};
    \node at (3,-.5) (m2) {$2$};
    \draw [-] (B) to (m1);
    \draw [-] (B) to (m2);    
    \end{tikzpicture}\,.\]
Using the projection formula and the induction hypothesis, we get the result.

(b) When $a$ is an odd number, this statement follows from the Grothendieck-Riemann-Roch computation in \cite[Proposition 1]{faberpandharipande00}. Namely,
\begin{equation*}
    0 = \mathrm{ch}_{2a-1}(\pi_*\omega_{\pi}) = \frac{B_{2a}}{(2a)!}\Big( \kappa_{2a-1} + \frac{1}{2}\sum_{\Gamma}\sum_{i=0}^{2a-2}(-1)^i\psi_h^i\psi_{h'}^{2a-a-i}[\Gamma]\Big)
\end{equation*}
where the sum is over $\A$-valued graphs $\Gamma$ with one edge $e = (h,h')$ and degree $\one$. Here $B_{2a}$ is the $2a$-th Bernoulli number.

We prove the statement for $\kappa_{2a}$ by the induction on $a$. Consider the forgetful morphism
\[ \pi_2\colon \M_{0,2,\one}\to \M_{0,0,\one} \,. \]
By the projection formula and \cite[Proposition 3.10]{BS1}, one computes
\[\pi_{2*}(\psi_1^{3}\psi_2^{2a+1}) = \kappa_{2a+2} + \kappa_2\kappa_{2a}\,.\]
On the other hand, 
\begin{align*}
  \pi_{2*}(\psi_1^{3}\psi_2^{2a+1}) &=  \pi_{2*}\Big(-\psi_1^2\psi_2^{2a+2} + 
\begin{tikzpicture}[scale=0.7, baseline=-3pt,label distance=0.3cm,thick,
    virtnode/.style={circle,draw,scale=0.5}, 
    nonvirt node/.style={circle,draw,fill=black,scale=0.5} ]
    \node at (0,-.5) [nonvirt node] (A) {};
    \node at (2,-.5) [nonvirt node] (B) {};
    \draw [-] (A) to (B);
    \node at (-.7,.5) (n1) {$\psi_1^2$};
    \draw [-] (A) to (n1);
    \node at (2.7,.5) (m1) {$\psi_2^{2a+1}$};
    \draw [-] (B) to (m1);
    \end{tikzpicture}
    \Big) \\
    &= -\kappa_{2a+2} - \kappa_1\kappa_{2a+1}+ \pi_{2*}\Big(\begin{tikzpicture}[scale=0.7, baseline=-3pt,label distance=0.3cm,thick,
    virtnode/.style={circle,draw,scale=0.5}, 
    nonvirt node/.style={circle,draw,fill=black,scale=0.5} ]
    \node at (0,-.5) [nonvirt node] (A) {};
    \node at (2,-.5) [nonvirt node] (B) {};
    \draw [-] (A) to (B);
    \node at (-.7,.5) (n1) {$\psi_1^2$};
    \draw [-] (A) to (n1);
    \node at (2.7,.5) (m1) {$\psi_2^{2a+1}$};
    \draw [-] (B) to (m1);
    \end{tikzpicture}
    \Big)\,.
\end{align*}
by Lemma \ref{Lem:psi12}. By the induction hypothesis, comparing the two equalities ends the proof.
\end{proof}
\subsection{Generators of \texorpdfstring{$\CH^*(\M_{0,n})$}{CH*(M0n)}} \label{Sect:genuszerogenerators}
The goal of this subsection is to prove Theorem \ref{Thm:M0ngenerators}.
The basic idea is simple: by Lemma \ref{Lem:Chowsmooth} we know that classes on the smooth locus of $\M_{0,n}$ have tautological representatives. By an excision argument, it suffices to show that classes supported on the boundary are tautological. But the boundary is parametrized under the gluing maps by products of $\M_{0,n_i}$. Then we want to conclude using an inductive argument.

The two main technical steps to complete are as follows:
\begin{itemize}
    \item The boundary of $\M_{0,n}$ is covered by a finite union of boundary gluing maps, which are proper and representable. We want to show that the direct sum of pushforwards by the gluing maps is surjective on the Chow group of the boundary.
    \item Knowing that classes on $\M_{0,n_1}$ and $\M_{0,n_2}$ are tautological up to a certain degree $d$, we want to conclude that classes of degree at most $d$ on the product $\M_{0,n_1} \times \M_{0,n_2}$ are tensor products of tautological classes.
\end{itemize}
The first issue is resolved by the fact that the pushforward along a proper surjective morphism of relative Deligne-Mumford type is surjective on the rational Chow group. We prove this statement in \cite[Appendix B.4]{BS1}.

We now turn to the second issue, understanding the Chow group of products of spaces $\M_{0,n_i}$. We make the following general definition, extended from \cite[Definition 6]{oesinghaus}.
\begin{definition}
An algebraic stack $X$ locally of finite type over $k$ 
is said to have the \emph{Chow K\"unneth generation property} (CKgP) if for all algebraic stacks $Y$ of finite type over $k$ the natural morphism
\begin{equation} \label{eqn:tensorsurject} \CH_*(X) \otimes \CH_*(Y) \to \CH_*(X \times Y)\end{equation}
is surjective. It is said to have the \emph{Chow K\"unneth property} (CKP) if this map is an isomorphism.
\end{definition}
It is immediate that if $X$ has the CKgP (or the CKP) and in addition has a good filtration, then the map  \eqref{eqn:tensorsurject} is surjective (or an isomorphism) for all $Y$ locally finite type over $k$ admitting a good filtration. The additional assumption of the good filtration is added since in general tensor products and right exact sequences are not compatible with inverse limits.

We now turn to showing the following result, resolving the second issue mentioned at the beginning of the section.
\begin{proposition} \label{pro:M0nCKgP}
For all $n \geq 0$, the stacks $\M_{0,n}$ have the CKgP  for finite type stacks $Y$ having a stratification by quotient stacks.
\end{proposition}
For the proof we start with the smooth part of $\M_{0,n}$.
\begin{proposition} \label{Pro:M0nsmCKP}
For all $n\geq 0$, the stacks $\M_{0,n}^{\textup{sm}}$ have the CKP.
\end{proposition}
\begin{proof}
Starting with the easy cases, for $n=2$ we have $\mathfrak{M}_{0,2}^{\textup{sm}} \cong B \mathbb{G}_m$ by Lemma \ref{Lem:Chowsmooth} and it was shown in \cite[Lemma 2]{oesinghaus} that this satisfies the CKP. On the other hand, for $n \geq 3$ we have $\M_{0,n}^{\textup{sm}} = \mathcal{M}_{0,n}$, which is an open subset of $\mathbb A^{n-3}$. Then for any finite type stack $Y$ we have $\CH_*(\mathcal{M}_{0,n}) \otimes \CH_*(Y)   \cong \CH_*(Y)$ and the map (\ref{eqn:tensorsurject}) is just the pullback under the projection $\mathcal{M}_{0,n} \times Y \to Y$. Combining \cite[Corollary 2.5.7]{kreschartin} and the excision sequence, we see that this pullback is surjective. On the other hand, composing it with the Gysin pullback by an inclusion 
\[Y \cong \{C_0\} \times Y \subset \mathcal{M}_{0,n} \times Y\]
for some $C_0 \in \mathcal{M}_{0,n}$ we obtain the identity on $\CH_*(Y)$, so it is also injective.

Next we consider the case $n=1$. By Lemma \ref{Lem:Chowsmooth} we have $\mathfrak{M}_{0,1}^{\textup{sm}} \cong B \UU$ for $\UU=\mathbb{G}_a \rtimes \mathbb{G}_m$. The group $\UU$ contains $\mathbb{G}_m$ as a subgroup and we claim that the natural map $B \mathbb{G}_m \to B \UU$ is an affine bundle with fibre $\mathbb{A}^1$. Indeed, the fibres are $\mathbb{U}/\mathbb{G}_m \cong \mathbb{A}^1$ and the structure group is $\UU=\mathrm{Aff}(1)$ acting by affine transformations on $\mathbb{A}^1$. Of course also for any finite type stack $Y$ it is still true that $Y \times B \mathbb{G}_m \to Y \times B \UU$ is an affine bundle. Then by \cite[Corollary 2.5.7]{kreschartin} we have that the two vertical maps in the diagram
\begin{equation*}
\begin{tikzcd}
\CH_*(Y) \otimes \CH_*(B \mathbb{G}_m) \arrow[r] & \CH_*(Y \times B \mathbb{G}_m)\\
\CH_*(Y) \otimes \CH_*(B \UU) \arrow[r] \arrow[u] & \CH_*(Y \times B \UU)\arrow[u]
\end{tikzcd}
\end{equation*}
induced by pullback of the affine bundles are isomorphisms. The top arrow in the diagram is also an isomorphism since, as seen above, $B \mathbb{G}_m$ has the CKP. Thus the bottom arrow is an isomorphism as well.

We are left with the case $n=0$. The forgetful map 
\begin{equation}\label{eqn:univCurve0}
    \pi \colon \M_{0,1}^\textup{sm}\to \M_{0,0}^\textup{sm}
\end{equation}
gives the universal curve over $\M_{0,0}^\textup{sm}$. The map \eqref{eqn:univCurve0} can be thought of as the morphism between quotient stacks
\[\pi : [\PP^1 / \PGL_2] \to [\mathrm{Spec}\, k / \PGL_2]\]
induced by the $\PGL_2$-equivariant map $\PP^1 \to \mathrm{Spec}\, k$. By \cite[Remark B.20]{BS1}, the map $\pi$ is projective and the line bundle $\mathcal{O}_{\PP^1}(2)$ on $\PP^1$ descends to a $\pi$-relatively ample line bundle on $[\PP^1 / \PGL_2]$.

Now for any finite type stack $Y$ consider a commutative diagram 
\begin{equation*}
\begin{tikzcd}
\CH_*(Y) \otimes \CH_*(B \mathrm{PGL}_2) \arrow[r] & \CH_*(Y \times B \mathrm{PGL}_2)\\
\CH_*(Y) \otimes \CH_*(B \UU) \arrow[r, "\cong"] \arrow[u,"\textup{id}\otimes \pi_*"] & \CH_*(Y \times B \UU)\arrow[u,"(\textup{id}\times \pi)_*"]
\end{tikzcd}
\end{equation*}
induced by the projective pushforward $\pi_*$. 
Note that the map $(\textup{id}\times \pi)_*$ is surjective. Indeed, a small computation\footnote{See the proof of Proposition \ref{Pro:M0nsmCKP_higherChow} for a variant of this computation.} shows that for $\alpha \in \CH_*(Y \times B \mathbb{U})$ we have
\[(\mathrm{id} \times \pi)_* \left( \frac{1}{2} c_1(\mathcal{O}_{\PP^1}(2)) \cap (\mathrm{id} \times \pi)^* \alpha \right)= \alpha\,.\]
Then the surjectivity of the top arrow follows. 

To prove injectivity of the top arrow consider the diagram
\begin{equation*}
\begin{tikzcd}
\CH_*(Y) \otimes \CH_*(B \mathrm{PGL}_2) \arrow[r] \arrow[d,"\textup{id}\otimes \pi^*"] & \CH_*(Y \times B \mathrm{PGL}_2)\arrow[d,"(\textup{id}\times \pi)^*"]\\
\CH_*(Y) \otimes \CH_*(B \UU) \arrow[r, "\cong"]  & \CH_*(Y \times B \UU)
\end{tikzcd}
\end{equation*}
induced by the flat pullback $\pi^*$. Similar to above, we see that for $\alpha \in \CH_*(B \mathbb{U})$ we have
\[\pi_* \left( \frac{1}{2} c_1(\mathcal{O}_{\PP^1}(2)) \cap \pi^* \alpha \right)= \alpha\,.\]
Thus the map $\mathrm{id} \otimes \pi^*$ is injective and thus the top arrow must be injective as well, finishing the proof.
\end{proof}
For the next results, we say that an equidimensional, locally finite type stack $X$ has the \emph{Chow K\"unneth generation property up to codimension $d$} if (\ref{eqn:tensorsurject}) is surjective in all codimensions up to $d$.

\begin{lemma} \label{Lem:CKgPprod}
Let $X, X'$ be equidimensional algebraic stacks, locally finite type over $k$ and admitting good filtrations. Then for $X,X'$ having the CKgP (up to codimension $d$), also $X \times X'$ has the CKgP (up to codimension $d$).
\end{lemma}
\begin{proof}
This is immediate from the definition.
\end{proof}

\begin{lemma} \label{Lem:CKgPpropmap}
Let $Z$ be an algebraic stack,  locally finite type over $k$, stratified by quotient stacks and with a good filtration by finite-type substacks. Let $\widehat Z \to Z$ be a proper, surjective map representable by Deligne-Mumford stacks such that $\widehat Z$ has the CKgP. Then $Z$ has the CKgP for stacks $Y$ stratified by quotient stacks.

If both $Z$ and $\widehat Z$ are equidimensional with $\dim \widehat Z - \dim Z = e \geq 0$ then if $\widehat Z$ has the CKgP up to codimension $d$, $Z$ has the CKgP (for stacks $Y$ stratified by quotient stacks) up to codimension $d-e$.
\end{lemma}
\begin{proof}
Let $Y$ be an algebraic stack of finite type over $k$, stratified by quotient stacks. 
Then in the diagram
\begin{equation}
\begin{tikzcd}
\CH_*(\widehat Z \times Y) \arrow[r, two heads] & \CH_*(Z \times Y)\\
\CH_*(\widehat Z) \otimes \CH_*(Y) \arrow[r] \arrow[u, two heads] & \CH_*(Z) \otimes \CH_*(Y)\arrow[u]
\end{tikzcd}    
\end{equation}
the top arrow is surjective by \cite[Proposition B.19]{BS1} (and \cite[Remark B.21]{BS1}) applied to $\widehat Z \times Y \to Z \times Y$ and the left arrow is surjective since $\widehat Z$ has the CKgP. It follows that $\CH_*(Z) \otimes \CH_*(Y) \to \CH_*(Z \times Y)$ is surjective, so $Z$ has the CKgP for stacks $Y$ stratified by quotient stacks. The statement with bounds on codimensions follows by looking at the correct graded parts of the above diagram and noting that codimension $d'$ cycles on $\widehat Z$ push forward to codimension $d'-e$ cycles on $Z$.
\end{proof}

\begin{proposition} \label{Pro:CKgPexcision}
Let $X$ be an algebraic stack over $k$ with a good filtration by finite type substacks and let $U \subset X$ be an open substack with complement $Z=X \setminus U$ such that $U$ and $Z$ have the CKgP.
Then $X$ has the CKgP. 

If $X$ is equidimensional and $Z$ has pure codimension $e$, $U$ has the CKgP up to codimension $d$ and $Z$ has the CKgP up to codimension $d-e$, then $X$ has the CKgP up to codimension $d$.
\end{proposition}
\begin{proof}
For $Y$ a finite type stack, using excision exact sequences on $X$ and $X \times Y$ we obtain a commutative diagram
\begin{equation*}
\begin{tikzcd}[column sep = small]
 \CH_*(Z \times Y) \arrow[r] & \CH_*(X \times Y) \arrow[r] & \CH_*(U \times Y) \arrow[r] & 0\\
\CH_*(Z) \otimes \CH_*(Y) \arrow[r] \arrow[u, two heads] & \CH_*(X) \otimes \CH_*(Y) \arrow[r]\arrow[u] & \CH_*(U) \otimes \CH_*(Y) \arrow[r]\arrow[u, two heads] & 0
\end{tikzcd}    
\end{equation*}
with exact rows. The vertical arrows for $U, Z$ are surjective since $U,Z$ have the CKgP. By the four lemma, the middle arrow is surjective as well, so $X$ has the CKgP. Again, the variant with bounds on the codimension follows by looking at the correct graded parts of the above diagram, noting that codimension $d'$ cycles on $Z$ push forward to codimension $d'+e$ cycles on $X$.
\end{proof}
Combining these ingredients, we are now ready to prove Proposition \ref{pro:M0nCKgP}.
\begin{proof}[Proof of Proposition \ref{pro:M0nCKgP}]
We will show that for all $d \geq 0$, all spaces $\M_{0,n}$ have the CKgP up to codimension $d$ by induction on $d$. Every stack has the CKgP up to codimension $d=0$, so the induction start is fine. Let now $d\geq 1$, then we want to apply Proposition \ref{Pro:CKgPexcision} for $X = \M_{0,n}$ with $U=\M_{0,n}^{\textup{sm}}$. Then $U$ has the CKgP by Proposition \ref{Pro:M0nsmCKP}. Its complement $Z= \partial\, \M_{0,n}$ admits a proper, surjective, representable cover
\begin{equation} \label{eqn:bdryparametrization} 
\widehat Z = \coprod_{I \subset \{1, \ldots, n\}} \M_{0,I \cup \{p\}} \times \M_{0, I^c \cup \{p'\}} \to Z = \partial\, \M_{0,n} \subset \M_{0,n}
\end{equation}
by gluing maps. Note that $\widehat Z$ and $Z$ are both equidimensional of the same dimension. By induction the spaces $\M_{0,I \cup \{p\}}$ and $\M_{0, I^c \cup \{p'\}}$ have the CKgP up to codimension $d-1$ (note that they both have at least one marking). So by Lemma \ref{Lem:CKgPprod} their product has the CKgP up to codimension $d-1$. By Lemma \ref{Lem:CKgPpropmap} we have that $Z$ has the CKgP up to codimension $d-1$. This is sufficient to apply Proposition \ref{Pro:CKgPexcision} to conclude that $\M_{0,n}$ has the CKgP up to codimension $d$ as desired.
\end{proof}

\begin{proof}[Proof of Theorem \ref{Thm:M0ngenerators}]
We show $\CH^d(\M_{0,n})=\R^d(\M_{0,n})$ (for all $n\geq 0$) by induction on $d \geq 0$. The induction start $d=0$ is trivial. So let $d \geq 1$ and assume the statement holds in codimensions up to $d-1$. By excision we have an exact sequence
\[\CH^{d-1}(\partial\, \M_{0,n}) \to \CH^d(\M_{0,n}) \to \CH^d(\M_{0,n}^\textup{sm}) \to 0.\]
By Lemma \ref{Lem:Chowsmooth} all elements of $\CH^d(\M_{0,n}^\textup{sm})$ have tautological representatives, so it suffices to show that this is also true for elements coming from $\CH^{d-1}(\partial\, \M_{0,n})$. Using the parametrization (\ref{eqn:bdryparametrization}) it suffices to show that codimension $d-1$ classes on products $\M_{0,n_1} \times \M_{0,n_2}$ are tautological (where $n_1, n_2 \geq 1$). By Proposition \ref{pro:M0nCKgP} we have a surjection 
\[\CH^*(\M_{0,n_1}) \otimes \CH^*(\M_{0,n_2}) \to \CH^*(\M_{0,n_1} \times \M_{0,n_2})\]
and by the induction hypothesis, all classes on the left side are (tensor products of) tautological classes up to degree $d-1$. Since tensor products of tautological classes map to tautological classes under gluing maps, this finishes the proof.
\end{proof}

\subsection{Higher Chow-K\"unneth property} \label{Sect:HigherChowK}
The goal of this section is to give a background to compute the higher Chow group of $\M^\textup{sm}_\Gamma$ for prestable graphs $\Gamma$. Computing higher Chow groups of $\M_{0,n}^\textup{sm}$ has two different flavors. When $n=0, 1,2$ or $3$, we use the projective bundle formula and its consequences. When $n\geq 4$, $\M_{0,n}^\textup{sm}$ is a hyperplane complement inside affine space and we use the motivic decomposition from \cite{Chatz}.

Below we study the Chow-K\"unneth property for higher Chow groups. Unlike the Chow-K\"unneth property for Chow groups, formulating the Chow-K\"unneth property for higher Chow groups in general is rather complicated, see \cite[Theorem 7.2]{Totaro16_motive}. Below, we focus on the case of the first higher Chow group $\CH^*(X,1)$ defined in \cite{kreschartin}\footnote{In \cite{kreschartin}, this group is denoted by $\underline{A}_*(X)$.}.
\begin{definition}\label{def:higher_CKP}
A quotient stack $X$ over $k$ is said to have the \emph{higher Chow K\"unneth property} (hCKP) if for all algebraic stacks $Y$ of finite type over $k$ the natural morphism \begin{equation} \label{eqn:HigherCKP} \CH^*(X,\bullet) \otimes_{\CH^*(k,\bullet)} \CH^*(Y,\bullet) \to \CH^*(X \times Y, \bullet)\end{equation} is an isomorphism in degree $\bullet=1$. A quotient stack $X$ over $k$ is said to have the \emph{higher Chow K\"unneth generating property} (hCKgP) if the above morphism is surjective.
\end{definition}
Expanding this definition slightly, the degree $\bullet = 1$ part of the left hand side of \eqref{eqn:HigherCKP} is given by the quotient
\begin{equation} \label{eqn:HigherCKP_explicitLHS}
    \frac{\left(\CH^*(X,1) \otimes_{\mathbb{Q}} \CH^*(Y,0)\right) \oplus \left(\CH^*(X,0) \otimes_{\mathbb{Q}} \CH^*(Y,1)\right)}{  \CH^*(k,1) \otimes_{\mathbb{Q}} \CH^*(X,0) \otimes_{\mathbb{Q}} \CH^*(Y,0)},
\end{equation}
where
\[\alpha \otimes \beta_X \otimes \beta_Y \in \CH^*(k,1) \otimes_{\mathbb{Q}} \CH^*(X,0) \otimes_{\mathbb{Q}} \CH^*(Y,0)\]
maps to
\[\left((\alpha \cdot \beta_X)\otimes \beta_Y, - \beta_X \otimes( \alpha \cdot \beta_Y) \right)\]
in the numerator of \eqref{eqn:HigherCKP_explicitLHS}. The cokernel of the following map \[\CH^1(k,1)\otimes\CH^{*-1}(X)\to \CH^{*}(X,1)\] is called the {\em indecomposable part}  $\RCH^*(X,1)$ of $\CH^*(X,1)$. 
For example $\RCH^1(\operatorname{Spec} k,1)=0$. 

We summarize some properties for higher Chow groups of quotient stacks $X=[U/G]$. In this case, the definition of the first higher Chow group of $X$ from \cite{kreschartin} coincides with the definition using Bloch's cycle complex of the finite approximation of $U_G=U\times_G EG$ from \cite{EdidinGraham98}. For the properties of higher Chow groups presented below, many of the proofs follow from this presentation. 
\begin{lemma}\label{Lem:ProjectiveBundle}
Let $X$ be a quotient stack and $E\to X$ be a vector bundle of rank $r+1$, and let $\pi\colon \PP(E)\to X$ be the projectivization. Let $\mathcal{O}(1)$ be the canonical line bundle on $\PP(E)$. Then the map
\begin{equation*}
    \theta_E(\bullet) \colon \bigoplus_{i=0}^r\CH_{*+i}(X,1)\to \CH_{*+r}(\PP(E),1)
\end{equation*}
given by
\[(\alpha_0,\ldots,\alpha_r)\mapsto \sum_{i=0}^r c_1(\mathcal{O}(1))^i\cap \pi^*\alpha_i\]
is an isomorphism.
\end{lemma}
\begin{proof}
Let $X=[U/G]$ be a quotient stack. Choose a $G$-representation $V$ and an open subspace $W\subset V$ on which $G$ acts freely. We can take a representation $V$ so that the codimension of $V\setminus W$ in $V$ has arbitrary large codimension. By \cite[Section 2.7]{EdidinGraham98}, the group $\CH_*(X,1)$ is isomorphic to  $\CH_*(U\times W/G,1)$ and the similar formula holds for $\CH_*(\PP(E),1)$. Now the property follows from the projective bundle formula \cite[Theorem 7.1]{bloch_higherChow}.\footnote{See also \cite[Theorem 4.2.2]{JoshuaHigher2}.}
\end{proof}
An affine bundle of rank $r$ over $X$ is a morphism $B\to X$ such that locally (in the smooth topology) on $X$, $B$ is a trivial affine $r$ plane over $X$ \cite[Section 2.5]{kreschartin}. 
We assume that the structure group of an affine bundle of rank $r$ is the group of affine transformations $\textup{Aff}(r)$ in $\textup{GL}(r+1)$. Therefore there exists an associated vector bundle $E$ of rank $r+1$ and an exact sequence of vector bundles
\[0\to F\to E\to \CO_X \to 0\,.\]
The complement of $\PP(F) \hookrightarrow \PP(E)$ is the affine bundle $B$.

We have a homotopy invariance property of higher Chow groups for affine bundles (see also \cite[Proposition 2.3]{KrishnaHomotopy}).
\begin{corollary}\label{Cor:affineBundle}
Let $X$ be a quotient stack and $\varphi\colon B\to X$ be an affine bundle or rank $r$. Then 
\[\varphi^*\colon \CH_*(X,1)\to \CH_{*+r}(B,1)\]
is an isomorphism. 
\end{corollary}
\begin{proof}
Let $p$ and $q$ be projections from $\PP(E)$ and $\PP(F)$ to $X$ respectively. There exists an excision sequence
\begin{equation*}
    \CH_*(\PP(F),1)\xrightarrow{i_*} \CH_*(\PP(E),1)\xrightarrow{j^*}\CH_*(B,1)\xrightarrow{\partial} \CH_*(\PP(F))\xrightarrow{i_*}\CH_*(\PP(E)) 
\end{equation*}
because all stacks are quotient stacks \cite{EdidinGraham98}. 
Since $\PP(F)$ is the vanishing locus of the canonical section of $\CO_{\PP(E)}(1)$ we have 
\[i_*q^*\alpha = c_1(\CO_{\PP(E)}(1))\cap p^* \alpha \text{ for }\alpha \in \CH_*(X)\]
by \cite[Lemma 3.3]{Fulton1984Intersection-th}. 
As $\alpha$ runs through a basis of $\CH_*(X)$, the classes
\begin{equation*}
    c_1(\CO_{\PP(F)}(1))^\ell \cap q^* \alpha \text{ for } 0 \leq \ell \leq r-1
\end{equation*}
run through a basis of $\CH_*(\PP(F))$ by Lemma \ref{Lem:ProjectiveBundle}. Pushing them forward via $i$, the classes
\begin{equation}\label{eqn:projectivizationBundle}
    i_*\left( c_1(\CO_{\PP(F)}(1))^\ell \cap q^* \alpha\right) = c_1(\CO_{\PP(E)}(1))^{\ell+1} \cap p^* \alpha  \text{ for }0 \leq \ell \leq r-1, \alpha
\end{equation}
form part of a basis of $\CH_*(\PP(E))$. In particular, the map $i_*\colon \CH_*(\PP(F))\to \CH_*(\PP(E))$ is injective and furthermore, we see that 
\begin{equation} \label{eqn:pstarisom} p^* : \CH_*(X,1)\to\CH_*(\PP(E),1) / \CH_*(\PP(F),1)\end{equation}
 gives an isomorphism. 

The injectivity of $i_*$ implies (via the excision sequence above) that $j^*$ is surjective. Using Lemma \ref{lem:intersectDivisor}, the formula (\ref{eqn:projectivizationBundle}) holds verbatim for higher Chow classes $\alpha \in \CH^*(X,1)$ so $i\colon \CH_*(\PP(F),1)\to \CH_*(\PP(E),1)$ is injective. Thus the excision sequence implies that $j^*$ induces an isomorphism
\begin{equation} \label{eqn:jstarisom}
    j^* : \CH_*(\PP(E),1) / \CH_*(\PP(F),1) \to \CH_*(B,1).
\end{equation}
But since $\varphi^*=j^*p^*$, we know that $\varphi$ is an isomorphism as the composition of the two isomorphisms \eqref{eqn:jstarisom} and \eqref{eqn:pstarisom}.
\end{proof}
\begin{proposition}\label{Pro:M0nsmCKP_higherChow}
For $n=0, 1, 2$ or $3$, the stacks $X=\M_{0,n}^\textup{sm}$ satisfy the hCKP for quotient stacks $Y$. Moreover we have $\RCH^*(X,1)=0$ and the natural morphism
\begin{equation} \label{eqn:tensorisom1_higherChow} \CH_*(X) \otimes_\QQ \CH_*(Y,1) \to \CH_*(X \times Y,1)\end{equation}
is an isomorphism. In particular, setting $Y = \mathrm{Spec}\ k$ we find
\[\CH_*(\M_{0,n}^\textup{sm},1) \cong \CH_*(\M_{0,n}^\textup{sm}) \otimes_{\mathbb{Q}} \CH_*(k,1)\,.\]
\end{proposition}
\begin{proof} When $n=3$, $\M_{0,3}^\textup{sm} = \textup{Spec} \,k$, so there is nothing to prove. 

When $n=2$, we use finite dimensional approximation of $B\mathbb{G}_m$ via projective spaces $\PP^N$, similar to the proof of \cite[Lemma 2]{oesinghaus}. Indeed, for the vector bundle $[\mathbb{A}^{N+1}/\mathbb{G}_m] \to B\mathbb{G}_m$, pullback induces an isomorphism of Chow groups and $[\mathbb{A}^{N+1}/\mathbb{G}_m]$  is isomorphic to $\PP^N$ away from codimension $N+1$. This shows the known identity \[\CH^\ell(B\mathbb{G}_m) \cong \CH^\ell(\PP^N)\text{ for }\ell \leq N.\]
Similarly, 
for $Y$ a quotient stack, $[\mathbb{A}^{N+1} \times Y / \mathbb{G}_m]$ is a vector bundle over $B\mathbb{G}_m \times Y$. By \cite[Proposition 5]{EdidinGraham98}, the higher Chow group of $[\mathbb{A}^{N+1} \times Y / \mathbb{G}_m]$ and $\PP^N\times Y$ is isomorphic up to degree $\ell \leq N$. One can use the homotopy invariance of higher Chow groups proven in \cite[Proposition 4.3.1]{kreschartin} to show that we have 
\[\CH^\ell(B\mathbb{G}_m \times Y, 1) \cong \CH^\ell(\PP^N \times Y, 1)\text{ for }\ell \leq N.\] 
On the other hand, the natural morphism
$$\CH^*(\PP^N) \otimes \CH^*(Y,1) \to \CH^*(\PP^N \times Y ,1)$$ 
is an isomorphism by Lemma \ref{Lem:ProjectiveBundle}. Combining with the equalities above, this shows that the map
\begin{equation} \label{eqn:hKCPBGm} \CH^*(B\mathbb{G}_m) \otimes \CH^*(Y,1)\to \CH^*(B\mathbb{G}_m \times Y ,1)\end{equation}
is an isomorphism. This shows the hCKP of $B \mathbb{G}_m$. 

When $n=1$, we have $\mathfrak{M}_{0,1}^{\textup{sm}} \cong B \UU$ for $\UU=\mathbb{G}_a \rtimes \mathbb{G}_m$ by Lemma \ref{Lem:Chowsmooth}. We already saw that for any finite type stack $Y$ the map $B \mathbb{G}_m \times Y  \to B \UU \times Y $ is an affine bundle. By Corollary \ref{Cor:affineBundle} we have the homotopy invariance \[\CH^*(B \UU \times Y, 1) \cong \CH^*(B \mathbb{G}_m \times Y, 1)\]
for all quotient stacks $Y$. Then the hCKP and the vanishing $\RCH^*(B \UU,1)=0$ for $B \UU$ follow from the corresponding properties of $B \mathbb{G}_m$ proven above.

We are left with the case $n=0$. For any quotient stack $Y$ consider a commutative diagram 
\begin{equation*}
\begin{tikzcd}
\CH_*(Y,1) \otimes \CH_*(B \mathrm{PGL}_2) \arrow[r] & \CH_*(Y \times B \mathrm{PGL}_2,1)\\
\CH_*(Y,1) \otimes \CH_*(B \UU) \arrow[r, "\cong"] \arrow[u,"\textup{id}\otimes \pi_*"] & \CH_*(Y \times B \UU,1)\arrow[u,"(\textup{id}\times \pi)_*"]
\end{tikzcd}
\end{equation*}
induced by the projective pushforward $\pi_*$. We start by proving surjectivity of $(\textup{id}\times \pi)_*$. By \cite[Remark B.20]{BS1} the morphism $\text{id}\times\pi$ can be factorized as
\begin{equation*}
    \begin{tikzcd}
    Y\times B\UU \arrow[r, hook, "i"] \arrow[d,"\text{id}\times\pi"]& Y\times \PP(E) \arrow[dl, "p"]\\
    Y\times B\mathrm{PGL}_2 &
    \end{tikzcd}
\end{equation*}
where $E$ is a rank $3$ vector bundle $E$ on $B\mathrm{PGL}_2$ associated to $H^0(\PP^1,\CO_{\PP^1}(2))$. Let $\xi=c_1(\CO_{\PP(E)}(1))$ be the relative hyperplane class. For any class $\alpha$ in $\CH(Y\times B\mathrm{PGL}_2,1)$, we have
\begin{align*}
    (\text{id}\times\pi)_*((i^* \xi) \cdot (\text{id}\times\pi)^*\alpha) &= p_*i_*((i^*\xi) \cdot i^*p^*\alpha)\\
    &= p_*\left(\xi \cdot i_*(i^*(p^*\alpha))\right)\\
    &= 2 p_*(\xi^2 \cdot p^*\alpha)\\
    &=2 \alpha
\end{align*}
where the first equality comes from the functoriality of pushforward and  Gysin pullback for higher Chow groups and the second equality is the projection formula \eqref{eqn:firstChernhigherprojectionformula} from Appendix \ref{Sect:Gysinpullbackhigher}. The third equality comes from Lemma \ref{lem:intersectDivisor} and the factor of two comes from the fact that the map $B \mathbb{U} \to \PP(E)$ is the second Veronese embedding of fiberwise degree two.  The fourth equality comes from \cite[Proposition 4.6]{Krishna13HigherChow}. Therefore $(\text{id}\times \pi)_*$ is surjective.

To prove injectivity of the top arrow consider the diagram
\begin{equation*}
\begin{tikzcd}
\CH_*(Y,1) \otimes \CH_*(B \mathrm{PGL}_2) \arrow[r] \arrow[d,"\textup{id}\otimes \pi^*"] & \CH_*(Y \times B \mathrm{PGL}_2,1)\arrow[d,"(\textup{id}\times \pi)^*"]\\
\CH_*(Y,1) \otimes \CH_*(B \UU) \arrow[r, "\cong"]  & \CH_*(Y \times B \UU,1)
\end{tikzcd}
\end{equation*}
induced by the flat pullback $\pi^*$. As seen in the proof of Proposition \ref{Pro:M0nsmCKP}, the map $\pi^* : \CH_*(B \PGL_2) \to \CH_*(B \mathbb{U})$ is injective and thus the left arrow of the above diagram is likewise injective. Hence the top arrow is an isomorphism, finishing the proof.
\end{proof}
The language of motives is a convenient way to state the higher Chow-K\"unneth property for $\mathcal{M}_{0,n}$ in the case $n\geq 4$.  For simplicity, let $k$ be a perfect field\footnote{This assumption can be removed by the work of Cisinski and D\'eglise, see \cite[Theorem 5.1]{Totaro16_motive}}. Let $\DM(k; \QQ)$ be the Voevodsky's triangulated category of motives over $k$ with $\QQ$-coefficients. Let $\mathrm{Sch}/k$ be the category of separated schemes of finite type over $k$. Then there exists a functor
\[\Ms\colon \mathrm{Sch}/k \to \DM(k; \QQ)\]
which sends a scheme to its motive.
The category $\DM(k; \QQ)$ is a tensor triangulated category, with a symmetric monoidal product $\otimes$ and $\Ms$ preserves the monoidal structure, namely $\Ms(X\times_k Y) = \Ms(X)\otimes \Ms(Y)$. See \cite{Voevodsky_book06} for the basic theory of motives. 

There is an invertible object, called the {\em Tate motive} $$\QQ(1)[2] \in \DM(k;\QQ)\,,$$ and by taking its shifting and tensor product we have $\QQ(a)[n]$ for any integers $a$ and $n$. Define the motivic cohomology of a scheme $X$ (in $\QQ$-coefficient) as 
\[H^{i}(X,\QQ(j))= \Hom_{\DM(k; \QQ)}(\Ms(X),\QQ(j)[i])\,.\]
The motivic cohomology is a bi-graded module over the motivic cohomology of the base field $k$. The motivic cohomology of $k$ is related to Milnor's K-theory of fields.

Voevodsky proved that the higher Chow group and the motivic cohomology have the following comparison isomorphism
\begin{equation}\label{eq:comparisonIso}
    H^{i}(X,\QQ(j)) \cong \CH^j(X,2j-i)_{\QQ}
\end{equation}
where the right hand side is Bloch's higher Chow group introduced in \cite{bloch_higherChow}. Bloch's definition of higher Chow groups will be used to compute the connecting homomorphism of the localization sequence. When $X$ is a smooth scheme over $k$, the higher Chow group and the motivic cohomology have product structure \[\CH^a(X,p)\otimes\CH^b(X,q)\to \CH^{a+b}(X,p+q)\] and the comparison isomorphism (\ref{eq:comparisonIso}) is a ring isomorphism (\cite{KondoYasuda11}).

Now we summarize results from \cite{Chatz}. For a hyperplane complement $U\subset \mathbb{A}^N$, there is a finite index set $I$ and $n_i\geq 0$ such that 
\[\Ms(U) \cong \bigoplus_{i\in I}\QQ(n_i)[n_i]\,.\]
As a corollary, $\CH^*(U,\bullet)$ is a finitely generated free module over $\CH^*(k,\bullet)$ and 
\begin{equation}\label{eq:firstHigherChow}
\CH^\ell(U,1)_{\mathbb{Z}} =
    \begin{cases}
    H^0(U,\mathcal{O}_U^\times) & \text{ , if $\ell=1$} \\
    0 & \text{ , otherwise.}
    \end{cases}
\end{equation}
There exists an isomorphism \[\CH^1(U,1)\cong \overline{\CH}^1(U,1)\oplus \CH^1(k,1)\,.\]
\begin{example}
Let $U= \textup{Spec}\,k[x,x^{-1}]$ be the complement of the origin in the affine line. Then \[\CH^1(U,1)_{\mathbb{Z}}\cong \bigsqcup_{a\in \mathbb{Z}} k^\times\langle x^a \rangle \cong \mathbb{Z} \oplus k^\times \]
and the element $m\in \CH^0(k)_{\mathbb{Z}}=\mathbb{Z}$ acts by $x^a\mapsto x^{ma}$ and $\lambda\in \CH^1(k,1)_{\mathbb{Z}}=k^\times$ acts by $x^a\mapsto \lambda x^a$. In fact $\CH^*(U,\bullet)$ is generated by the fundamental class and $\langle x\rangle$ over $\CH^*(k,\bullet)$.
\end{example}
\begin{proposition}\label{Pro:M0nsmCKP_stable_higherChow}
Let $U\subset \mathbb{A}^N$ be a hyperplane complement as above. Then the hCKP holds for quotient stacks. 
\end{proposition}
\begin{proof}
Let $Y$ be a quotient stack and hence $Y\times U$ is also a quotient stack. $Y$ admits a vector bundle $E$ such that the vector bundle is represented by a scheme off a locus of arbitrarily high codimension. Since $U$ is a scheme, the pullback of $E$ to $Y\times U$ also satisfies the same property. For higher Chow groups of quotient stacks, the homotopy invariance for vector bundle and the extended localization sequence is proven in \cite{Krishna13HigherChow}. Therefore we may assume that $Y$ is a scheme. When $Y$ is a scheme, there exists isomorphisms
\begin{align*}
    \CH^l(Y\times U,1) &= \Hom(\Ms(Y\times U), \QQ(l)[2l-1])\\
    &=\Hom(\Ms(Y)\otimes\Ms(U),\QQ(l)[2l-1])\\
    &= \Hom(\bigoplus_{i\in I}\Ms(Y)(n_i)[n_i],\QQ(l)[2l-1])\\
    &=\bigoplus_{i\in I}\Hom(\Ms(Y)(n_i)[n_i],\QQ(l)[2l-1])\\
    &= \bigoplus_{i\in I} \Hom(\Ms(Y), \QQ(l-n_i)[2l-n_i-1])\\
    &= \bigoplus_{i\in I} \CH^{l-n_i}(Y,1-n_i)\\
    &= \bigoplus_{n_i\leq 1} \CH^{l-n_i}(Y,1-n_i)
\end{align*}
where the fifth equality comes from the cancellation theorem. In the proof of \cite[Proposition 1.1]{Chatz}, the index $n_i=0$ corresponds to $\CH^0(U,0)$ and the indices $n_i=1$ corresponds to generators of $\RCH^1(U,1)$ over $\QQ$. Therefore we get the isomorphism.
\end{proof}

After identifying
\begin{equation*}
    \mathcal{M}_{0, n} = \{ (x_1,\ldots, x_n)\in \mathbb{A}^{n-3} \colon x_i\neq x_j, \textup{for} \,i\neq j,\, x_i\neq 0, x_j\neq 1 \}\subset \mathbb{A}^{n-3}\,,
\end{equation*} Proposition \ref{Pro:M0nsmCKP_higherChow} and  \ref{Pro:M0nsmCKP_stable_higherChow} compute the higher Chow group of $\prod_{v\in V(\Gamma)}\M_{0,n(v)}^\textup{sm}$ for any prestable graph $\Gamma$.
\medskip

Now we revisit the CKP for the stack $\M_{0,n}$. We recall the definition of Bloch's higher Chow groups \cite{bloch_higherChow}.
Let 
\[\Delta^m=\text{Spec} \left(k[t_0, \ldots, t_m]/(t_0+\ldots + t_m - 1)\right)\]
be the algebraic $m$ simplex. For $0\leq i_1<\ldots<i_a \leq m$, the equation $t_{i_1}= \ldots = t_{i_a}=0$ defines a face $\Delta^{m-a}\subset \Delta^n$. 
Let $X$ be an equidimensional quasi-projective scheme over $k$. Let $z^i(X,m)$ be the free abelian group generated by all codimension $i$ subvarieties of $X\times \Delta^m$ which intersect all faces $X\times \Delta^l$ properly for all $l<m$. Taking the alternating sum of restriction maps to $i+1$ faces of $X\times\Delta^i$, we get a chain complex $(z^*(X,m),\delta)$. The higher Chow group $\CH^i(X,m)$ is the $i$-th cohomology of the complex $z^*(X,m)$.

When $m=1$, the proper intersection is equivalent to saying that cycles are not contained in any of the (strict) faces. Let $R=\Delta^1\setminus \{[0], [1]\}$. Then  the group $z^*(X,1)$ is equal to $z^*\left(X\times R\right)$ and the differential 
\[\ldots \xrightarrow{} z^*(X\times R) \xrightarrow{\delta} z^*(X)\xrightarrow{} 0\]
is given by specialization maps. If $\sum a_i W_i$ is a cycle in $X\times R$,
\begin{equation}\label{eqn:specialization}
    \delta\left(\sum a_i W_i\right) = \sum a_i\overline{W_i}\cap X\times [0] - \sum a_i\overline{W_i}\cap X\times [1]
\end{equation}
where $\overline{W_i}$ is the closure of $W_i$ in $X\times \Delta^1$. 
\begin{lemma} \label{Lem:HigherChowProductcompat}
Let $X_1, X_2$ be algebraic stacks stratified by quotient stacks and $Z_1 \subset X_1$ and $Z_2\subset X_2$ be closed substacks with complements \[i_1 \colon U_1 = X_1 \setminus Z_1 \hookrightarrow X_1, \hspace{3mm} i_2 : U_2 = X_2 \setminus Z_2 \hookrightarrow X_2\] where $U_1, U_2$ are quotient stacks. Let $Z_{12}=X_1\times X_2 \setminus U_1\times U_2$. Denote by
\begin{align*}
    \partial_1 &\colon \CH_*(U_1,1) \to \CH_*(Z_1)\,,\\
    \partial_2 &\colon \CH_*(U_2,1) \to \CH_*(Z_2)\,,\\
    \partial &\colon \CH_*(U_1 \times U_2,1) \to \CH_*(Z_{12})
\end{align*}
the boundary maps for the inclusions $U_1 \subset X_1$, $U_2 \subset X_2$, $U_1 \times U_2 \subset X_1 \times X_2$.
\begin{enumerate}[label=(\alph*)] 
    \item For $\alpha\in \CH_*(U_1,1)$ and $\beta \in \CH_*(U_2)$, \[\partial(\alpha\otimes \beta)=\partial_1(\alpha)\otimes \overline{\beta}\, \text{ in }\CH_*(Z_{12})\] where $\overline{\beta}\in \CH_*(X_2)$ is any extension of $\beta$.
    \item The following diagram commutes
    \begin{equation}\label{eqn:chainrule}
\begin{tikzcd}[column sep = huge]
 \begin{array}{r} \CH_*(U_1,1) \otimes \CH_*(X_2) \\ \oplus \CH_*(X_1) \otimes \CH_*(U_2,1)\end{array} \arrow[r,"(\partial_1 \otimes \mathrm{id}) \oplus (\mathrm{id} \otimes \partial_2)"] \arrow[d,"(\mathrm{id} \otimes i_2^*) \oplus (i_1^* \otimes \mathrm{id})"]& \begin{array}{r} \CH_*(Z_1) \otimes \CH_*(X_2) \\ \oplus \CH_*(X_1) \otimes \CH_*(Z_2)\end{array}\arrow[d]\\
 \CH_*(U_1 \times U_2,1) \arrow[r,"\partial"] & \CH_*(Z_{12})\,,
\end{tikzcd}
\end{equation}
where the arrow on the right is induced by the natural map
\[Z_1 \times X_2 \sqcup X_1 \times Z_2 \to Z_{12}.\]
\end{enumerate}
\end{lemma}
\begin{proof} (a) We first prove that the right hand side is well-defined. For a different choice of extension $\overline{\beta}'$ of $\beta$, the difference $\overline{\beta}-\overline{\beta}'$ is a class supported on $Z_2$. Therefore, the class $\partial_1(\alpha)\otimes (\overline{\beta}-\overline{\beta}')$ on $Z_{12}$ is supported on $Z_1 \times Z_2$. In particular, it a class pushed forward from $\CH_*(X_1\times Z_2)$. This class vanishes because $\partial_1 \alpha$ vanishes as a class in $\CH_*(X_1)$. 

We first prove the equality when $X_1,X_2$ are schemes. The proof follows from diagram chasing. Recall that the connecting homomorphism $\partial \colon \CH_*(U_1 \times U_2,1) \to \CH_*(Z_{12})$ is defined using the following diagram
\begin{equation*}
    \begin{tikzcd}
     z^*(Z_{12},1) \arrow[r] \arrow[d] & z^*(Z_{12})\arrow[d,"j_*"] \arrow[r] &0\\
     z^*(X_1\times X_2,1) \arrow[r,"\delta"] \arrow[d,"i^*"] &z^*(X_1\times X_2)\arrow[d]\arrow[r] &0\\
     z^*(U_1\times U_2,1) \arrow[r,"\delta"]  &z^*(U_1\times U_2)\arrow[r] &0\,.
    \end{tikzcd}
\end{equation*}
For each class in $\CH_*(U_1\times U_2,1)$ take a representative in $z^*(U_1\times U_2,1)$. By taking a preimage under $i^*$, applying the map $\delta$ and taking a preimage under $j_*$, we get a class in $\CH_*(Z_{12})$ which corresponds to the image of $\partial$. Fix a representative of $\alpha$ in $z^*(U_1\times R)$ and $\beta$ in $z^*(U_2)$. Let $\overline{\alpha}$ be the closure of $\alpha$ in $X_1\times R$ and $\overline{
\beta}$ be the closure of $\beta$ in $X_2$. Let $\widetilde{\alpha}$ be the closure of $\overline{\alpha}$ in $X_1\times\AA^1$. Then to compute $\partial(\alpha \times \beta)$ we observe that $\alpha \times \beta = i^*(\overline{\alpha}\times \overline{\beta})$. Applying $\delta$, we have
\begin{align*}
    \delta(\overline{\alpha}\times \overline{\beta}) &= \widetilde{\alpha}\times \overline{\beta} \cap X_1\times  [0]\times X_2 - \widetilde{\alpha}\times \overline{\beta} \cap X_1\times [1] \times X_2\\
    &= j_*\left(\widetilde{\alpha}\cap X_1\times [0]-\widetilde{\alpha}\cap X_1\times [1] \right)\times \overline{\beta}\\
    &=j_*(\partial(\alpha)\times \overline{\beta})
\end{align*}
and this proves the equality.

In general, let $U_1$ be a quotient stack by assumption. For a projective morphism $S_1\to U_1$ from a reduced stack $S_1$, there exists a projective morphism $T_1\to X_1$ such that $S_1\cong U_1\times_{X_1} T_1$ (\cite[Corollary 2.3.2]{kreschartin}). Let $E_1$ be a vector bundle on $S_1$. By \cite[Proposition 2.3.3]{kreschartin}, there exists a projective modification $T_1'\to T_1$ and a vector bundle $E_1'$ which restricts to $E_1$. We perform a similar construction for the quotient stack $U_2$. The image of $\CH_*(U_1,1)\otimes \CH_*(U_2)$ under the boundary map 
\[\partial\colon \CH_*(U_1\times U_2,1)\to \CH_*(Z_{12})\]
is defined by the limit of boundary maps for naive higher Chow groups of $E_1\boxtimes E_2\subset E_1'\boxtimes E_2'$. The corresponding computation is precisely equal to the case above. Therefore the same formula holds for stacks $X_1$ and $X_2$.


(b) Let $\alpha\otimes \beta\in \CH_*(U_1,1)\otimes \CH_*(X_2)$. We take a natural extension $\beta$ of $i^*_2\beta$. Then by (a), we have
\begin{align*}
    \partial\circ(\mathrm{id}\otimes i^*_2)\big(\alpha\otimes\beta\big)&= \partial\big(\alpha\otimes i^*_2\beta\big)\\ &=\partial_1(\alpha)\otimes \beta\\ &=\partial_1\otimes\mathrm{id}\big(\alpha\otimes \beta\big)\,.
\end{align*}
The same computation holds for $\CH_*(X_1)\otimes \CH_*(U_2,1)$ and we get the commutativity of (\ref{eqn:chainrule}).
\end{proof}
\begin{remark} \label{Rmk:partialvanishonCHk1}
Applying Lemma \ref{Lem:HigherChowProductcompat} to $X_1 = U_1 = \mathrm{Spec}\, k$ (so that $Z_1=\emptyset$) and $Z = Z_2 \subseteq X = X_2$ with $U = X \setminus Z$, we find that the composition 
\[\CH_*(k,1) \otimes \CH_*(U) \to  \CH_*(U,1) \xrightarrow{\partial} \CH_*(Z)\]
vanishes since for $\alpha \in \CH_*(k,1)$ and $\beta \in \CH_*(U)$ we have $\partial(\alpha \otimes \beta) = \partial_1(\alpha) \otimes \overline{\beta} = 0$ as $\partial_1(\alpha)$ lives in $\CH_*(Z_1) = \CH_*(\emptyset) = 0$. This implies that $\partial$ factors through the indecomposable part $\RCH_*(U,1)$ of $\CH_*(U,1)$.
\end{remark}

To prove the CKP for $\M_{0,n}$, we want to use that via the boundary gluing morphisms, the space $\M_{0,n}$ is stratified by (finite quotients of) products of spaces $\M_{0,n_i}^\textup{sm}$, for which we know the CKP. The following proposition tells us that indeed the CKP for such a stratified space can be checked on the individual strata.
\begin{proposition}\label{Pro:CKPstratifiedspace}
Let $X$ be an algebraic stack, locally of finite type over $k$ with a good filtration and stratified by quotient stacks $X=\cup X_i$. Suppose each stratum $X_i$ has the CKP and the hCKgP for quotient stacks. Then $X$ has the CKP for quotient stacks.
\end{proposition}
\begin{proof}
Since $X$ has a good filtration, the Chow groups of $X$ and $X \times Y$ of a fixed degree can be computed on a sufficiently large finite-type open substack. This allows us to reduce to the case where $X$ has finite type. 

Now by assumption, there exists a nonempty open substack $U\subset X$ which is a quotient stack and has the CKP. Let $Z=X\setminus U$ be the complement. For a quotient stack $Y$ consider a commutative diagram
\begin{equation}\label{eqn:exsequencestratification}
\begin{tikzcd}[column sep = tiny]
\CH_*(U,1)\otimes\CH_*(Y) \arrow[r] \arrow[d,"\gamma_1"]  & \CH_*(Z)\otimes\CH_*(Y) \arrow[r] \arrow[d,"\gamma_2"] & \CH_*(X)\otimes\CH_*(Y) \arrow[r] \arrow[d,"\gamma_3"] & \CH_*(U)\otimes\CH_*(Y)\arrow[r]\arrow[d,"\gamma_4"] & 0\\
 \CH_*(U\times Y,1) \arrow[r] & \CH_*(Z\times Y) \arrow[r] & \CH_*(X\times Y) \arrow[r] & \CH_*(U\times Y) \arrow[r]&0
\end{tikzcd}
\end{equation}
where the rows are exact by the excision sequence. Since $U$ has the CKP, the arrow $\gamma_4$ is an isomorphism and by Noetherian induction, the same is true for $\gamma_2$. We extend the domain of the map $\gamma_1$ by inserting an extra component $\CH_*(U)\otimes\CH_*(Y,1)$. Then the following diagram
\begin{equation*} 
\begin{tikzcd}[column sep = huge]
 \begin{array}{r} \CH_*(U,1)\otimes\CH_*(Y) \\ \oplus \CH_*(U)\otimes\CH_*(Y,1)\end{array} \arrow[r] \arrow[d,"\gamma'_1"]& \CH_*(Z)\otimes\CH_*(Y)\arrow[d]\\
 \CH_*(U\times Y,1) \arrow[r] & \CH_*(Z\times Y)
\end{tikzcd}
\end{equation*}
commutes by applying Lemma~\ref{Lem:HigherChowProductcompat} to $U \times Y \subset X \times Y$. Note that the new factor $\CH_*(U)\otimes\CH_*(Y,1)$ maps to $\CH_*(X)\otimes\CH_*(Y \setminus Y)=0$ under the top arrow, and so in particular the top row of \eqref{eqn:exsequencestratification} remains exact after the modification. Furthermore, the modified map $\gamma'_1$ is surjective because $U$ has the hCKgP for quotient stacks. Therefore $\gamma_3$ is an isomorphism by applying the five lemma to the modified version of \eqref{eqn:exsequencestratification}.
\end{proof}
To apply this to the stratification of $\M_{0,n}$ by prestable graph, we need a small further technical lemma, due to the fact that the strata of $\M_{0,n}$ are quotients of products of $\M_{0,n_i}$ by finite groups.
\begin{lemma} \label{Lem:Chowfinitequotient}
Let $\M$ be an algebraic stack of finite type over $k$ and stratified by quotient stacks with an action of a finite group $G$. Then the quotient map $\pi : \M \to \M/G$ induces an isomorphism
\begin{equation}
    \pi^* : \CH_*(\M / G) \to \CH_*(\M)^G
\end{equation}
from the Chow group of the quotient $\M/G$ to the $G$-invariant part of the Chow group of $\M$. On the other hand, the map
\begin{equation}
    \pi_* : \CH_*(\M ) \to \CH_*(\M/ G)
\end{equation}
is a surjection.
\end{lemma}
\begin{proof}
The map $\pi$ is representable and a principal $G$-bundle, hence in particular it is finite and \'etale. Thus we can both pull back cycles and push forward cycles under $\pi$. 
For $g \in G$ let $\sigma_g : \M \to \M$ be the action of $g$ on $\M$. Then the relation $\pi \circ \sigma_g \cong \pi$ shows that $\sigma_g^*$ acts as the identity on the image of $\pi^*$ and thus $\pi^*$ has image in the $G$-invariant part of $\CH_*(\M)$. The equality 
\[\pi_* \circ \pi^* = |G| \cdot \mathrm{id} : \CH_*(\M / G) \to \CH_*(\M / G)\]
shows that $\pi^*$ is injective and that $\pi_*$ is surjective (since we work with $\mathbb{Q}$-coefficients). On the other hand, we have
\[\pi^* \circ \pi_* = \sum_{g \in G} \sigma_g^* : \CH_*(\M) \to \CH_*(\M)\]
thus restricted on the $G$-invariant part, we again have
\[\pi^* \circ \pi_*|_{\CH_*(\M)^G} = |G| \cdot \mathrm{id} : \CH_*(\M)^G \to \CH_*(\M)^G\,,\]
showing $\pi^*$ is surjective.
\end{proof}
\begin{remark} \label{Rmk:higherChowfinitequotient}
The above lemma is also true for the first higher Chow groups with $\QQ$-coefficients.
\end{remark}

\begin{corollary}\label{Cor:CKPM0n}
For all $n\geq 0$, the stacks $\M_{0,n}$ have the CKP for  quotient stacks.
\end{corollary}
\begin{proof}
Recall that for a prestable graph $\Gamma$ of genus $0$ with $n$ markings, there exists the locally closed substack $\M^\Gamma \subset \M_{0,n}$ of curves with dual graph exactly $\Gamma$. By Proposition~\ref{Pro:CKPstratifiedspace}, it suffices to show that the stacks $\M^\Gamma$ have the CKP and the hCKgP for quotient stacks.
Now from \cite[Proposition 2.4]{BS1} we know that the restriction of the gluing map $\xi_\Gamma$ induces an isomorphism
\[\left(\prod_{v \in V(\Gamma)} \M_{0,n(v)}^\textup{sm} \right)/ \mathrm{Aut}(\Gamma) \xrightarrow{\xi_\Gamma} \M^\Gamma\,.\] 
The product of spaces $\M_{0,n(v)}^\textup{sm}$ has the CKP by Proposition \ref{Pro:M0nsmCKP}
and the hCKgP for quotient stacks by Proposition \ref{Pro:M0nsmCKP_higherChow} and  Proposition~\ref{Pro:M0nsmCKP_stable_higherChow}. From Lemma \ref{Lem:Chowfinitequotient} and Remark~\ref{Rmk:higherChowfinitequotient} it follows that the quotient of a space with the CKP (or hCKgP) under a finite group action still has the CKP (or hCKgP), so by the above isomorphism all $\M^\Gamma$ have the CKP and hCKgP for quotient stacks. This finishes the proof.
\end{proof}
We proved the Chow-K\"unneth property of $\M_{0,n}$ with respect to quotient stacks. This assumption comes from technical assumptions in \cite{kreschartin}. For example, the extended excision sequence is only proven when the open substack is a quotient stack. Such assumptions are not necessary for a different cycle theory of algebraic stacks constructed in \cite{khan2019virtual}. Therefore, the following remark could remove the technical assumptions in the above Chow-K\"unneth property.
\begin{remark}
Let $X$ be an algebraic stack, locally of finite type over $k$. Let $\mathrm{H}_{*}^{\textup{BM}}$ be the rational motivic Borel--Moore homology theory defined in \cite{khan2019virtual}. There exists a cycle class map
\begin{equation*}
    \mathrm{cl} \colon \CH_*(X)_\QQ \to \mathrm{H}_{*}^{\textup{BM}}(X)
\end{equation*}
which is compatible with projective pushforward, Chern classes and lci pullbacks. In \cite{BaePark}, we will show that the cycle class map $\mathrm{cl}$ is an isomorphism when $X$ is stratified by quotient stacks.
\end{remark}

\subsection{Tautological relations}\label{sec:tautrelations}
In this section, we formulate and prove a precise form of Theorem \ref{Thm:wdvv}, see Theorem \ref{Thm:fullgenus0rels}.
Recall that for a prestable graph $\Gamma$, a decoration $\alpha$ is an element of $\CH^*(\M_\Gamma)$ given as a product $\alpha = \prod_v \alpha_v$ where $\alpha_v \in \CH^*(\M_{0,n(v)})$ are monomials in $\kappa$ and $\psi$-classes  on the factors $\M_{0,n(v)}$ of $\M_\Gamma$. 

\begin{definition} \label{Def:normalform}
Define the \emph{strata space} $\StratAlg_{g,n}$ to be the free $\mathbb{Q}$-vector space with basis given by isomorphism classes of decorated prestable graphs $[\Gamma, \alpha]$. 
\end{definition}
\noindent
By definition, the image of the map
\[\StratAlg_{g,n} \to \CH^*(\M_{g,n})\,, \,\,\, [\Gamma,\alpha]\mapsto \xi_{\Gamma*}\alpha\] 
is the tautological ring $\R^*(\M_{g,n})$\footnote{From \cite[Corollary 3.7]{BS1} we see that there is a $\mathbb{Q}$-algebra structure on $\StratAlg_{g,n}$ which makes this map into a $\mathbb{Q}$-algebra homomorphism, since products of decorated strata classes are given by explicit combinations of further decorated strata classes.}.

For the proofs below, it is convenient to allow decorations $\alpha_v$ at vertices of $\Gamma$ which are combinations of monomials in $\kappa$ and $\psi$-classes as follows.
\begin{definition}
Given a prestable graph $\Gamma$ in genus $0$ with $n$ markings, an element
\[
\alpha = \prod_{v\in V(\Gamma)} \alpha_v \in \prod_{v\in V(\Gamma)} \CH^*(\M_{0,n(v)})
\]
is said to be in \emph{normal form} if 
\begin{enumerate}[label=\alph*)]
    \item for vertices $v \in V(\Gamma)$ with $n(v)=0$, we have $\alpha_v = \kappa_2^a$ for some $a \geq 0$, 
    \item for vertices $v \in V(\Gamma)$ with $n(v) = 1$, we have $\alpha_v = \psi_h^b$, where $h$ is the unique half-edge at $v$ and $b \geq 0$,
    \item for vertices $v \in V(\Gamma)$ with $n(v) =2$, we have $\alpha_v = \psi_h^c + (-\psi_{h'})^{c}$, where $h,h'$ are two half-edges at $v$ and $c \geq 0$,
    \item for vertices $v \in V(\Gamma)$ with $n(v) \geq 3$, we have that $\alpha_v = 1$ is trivial.
\end{enumerate}
\end{definition}
Note that because of the terms $\psi_h^c + (-\psi_{h'})^{c}$ in case c) above, the element $\alpha$ is not strictly speaking a decoration, since the $\alpha_v$ are not monomials. However, given $\Gamma, \alpha$ as in Definition \ref{Def:normalform}, we write $[\Gamma, \alpha]$ for the element in $\StratAlg_{0,n}$ obtained by expanding $\alpha$ in terms of monomial decorations.
\begin{definition}
For $g=0$ let $\StratAlg_{0,n}^\textup{nf} \subset \StratAlg_{0,n}$ be the subspace additively generated by $[\Gamma ,\alpha]$ for $\alpha$ in normal form\footnote{
Note that  at vertices $v \in V(\Gamma)$ with $n(v)=2$ we have a choice of ordering of the two half-edges $h,h'$, and the possible decorations $\alpha_v = \psi_h^c + (-\psi_{h'})^{c}$  differ by a sign for $c$ odd. Still they generate the same subspace of $\StratAlg_{0,n}$ and the independence of this subspace from the choice of ordering of half-edges will be important in a proof below.}.
\end{definition}

\begin{definition} \label{Def:relsgeneratedby}
Let $R_0 \in \StratAlg_{g_0, n_0}$ be a tautological relation. Given $g,n$, we say that \emph{the set of relations in $\StratAlg_{g,n}$ generated by $R_0$} is the subspace of the $\QQ$-vector space $\StratAlg_{g,n}$ generated by elements of $\StratAlg_{g,n}$ obtained by
\begin{itemize}
    \item choosing a prestable graph $\Gamma$ in genus $g$ with $n$ markings and a vertex $v \in V(\Gamma)$ with $g(v)=g_0, n(v)=n_0$,
    \item choosing an identification of the $n_0$ half-edges incident to $v$ with the markings $1, \ldots, n_0$ for $\StratAlg_{g_0, n_0}$,
    \item choosing decorations $\alpha_w \in \CH^*(\M_{g(w), n(w)})$ for all vertices $w \in V(\Gamma) \setminus \{v\}$,
    \item gluing the relation $R_0$ into the vertex $v_0$ of $\Gamma$, putting decorations $\alpha_w$ in the other vertices and expanding as an element of $\StratAlg_{g,n}$.
\end{itemize}
More generally, given any family $(R_0^i \in \StratAlg_{g_0^i, n_0^i})_{i \in I}$ of tautological relations, we define the relations in $\StratAlg_{g,n}$ generated by this family to be the sum of the spaces of relations generated by the $R_0^i$.
\end{definition}

\begin{example} \label{exa:gluingrelationinvertex}
Let $n_0=4$ with markings labelled $\{3,4,5,h\}$ and let $R_0\in \StratAlg_{0,n_0}$ be the WDVV relation
\begin{equation*}
    \begin{tikzpicture}[baseline=-3pt,label distance=0.5cm,thick,
    virtnode/.style={circle,draw,scale=0.5}, 
    nonvirt node/.style={circle,draw,fill=black,scale=0.5} ]
    \node [nonvirt node] (A) {};
    \node at (1,0) [nonvirt node] (B) {};
    \draw [-] (A) to (B);
    \node at (-.7,.5) (n1) {$3$};
    \node at (-.7,-.5) (n2) {$h$};
    \draw [-] (A) to (n1);
    \draw [-] (A) to (n2);
    
    \node at (1.7,.5) (m1) {$4$};
    \node at (1.7,-.5) (m2) {$5$};
    \draw [-] (B) to (m1);
    \draw [-] (B) to (m2);    
    \end{tikzpicture}
=
    \begin{tikzpicture}[baseline=-3pt,label distance=0.5cm,thick,
    virtnode/.style={circle,draw,scale=0.5}, 
    nonvirt node/.style={circle,draw,fill=black,scale=0.5} ]
    \node at (0,0) [nonvirt node] (A) {};
    \node at (1,0) [nonvirt node] (B) {};
    \draw [-] (A) to (B);
    \node at (-.7,.5) (n1) {$4$};
    \node at (-.7,-.5) (n2) {$h$};
    \draw [-] (A) to (n1);
    \draw [-] (A) to (n2);
    
    \node at (1.7,.5) (m1) {$3$};
    \node at (1.7,-.5) (m2) {$5$};
    \draw [-] (B) to (m1);
    \draw [-] (B) to (m2);    
    \end{tikzpicture} \,.
\end{equation*}
For a prestable graph 
\begin{equation*}
    \Gamma=
    \begin{tikzpicture}[baseline=-3pt,label distance=0.5cm,thick,
    virtnode/.style={circle,draw,scale=0.5}, 
    nonvirt node/.style={circle,draw,fill=black,scale=0.5} ]
    \node [nonvirt node] (A) {};
    \node at (1,0) [nonvirt node] (B) {};
    \draw [-] (A) to (B);
    \node at (-.7,.5) (n1) {$1$};
    \node at (-.7,-.5) (n2) {$2$};
    \node at (0,-.3) (n3) {$w$};
    \draw [-] (A) to (n1);
    \draw [-] (A) to (n2);
    
    \node at (.7,.3) (x1) {$h$};
    \node at (1,-.3) (x2) {$v$};
    \node at (1.7,.5) (m1) {$3$};
    \node at (1.7,0) (m3) {$4$};
    \node at (1.7,-.5) (m2) {$5$};
    \draw [-] (B) to (m1);
    \draw [-] (B) to (m2);    
    \draw [-] (B) to (m3);
    \end{tikzpicture}
    \hspace{5mm} \text{ in $\StratAlg_{0,5}$}
\end{equation*}
and a decoration $\alpha_w=\kappa_3$, the corresponding relation is 
\begin{equation*}
    \begin{tikzpicture}[baseline=-3pt,label distance=0.5cm,thick,
    virtnode/.style={circle,draw,scale=0.5}, 
    nonvirt node/.style={circle,draw,fill=black,scale=0.5} ]
    \node [nonvirt node] (A) {};
    \node at (1,0) [nonvirt node] (B) {};
    \node at (2,0) [nonvirt node] (C) {};
    \draw [-] (A) to (B);
    \node at (-.7,.5) (n1) {$1$};
    \node at (-.7,-.5) (n2) {$2$};
    \node at (1,.7) (n3) {$3$};
    \node at (0,.3) (x) {$\kappa_3$};
    \draw [-] (A) to (n1);
    \draw [-] (A) to (n2);
    \draw [-] (B) to (n3);
    \draw [-] (B) to (C);
    
    \node at (.7,.3) (x1) {$h$};
    \node at (2.7,.5) (m1) {$4$};
    \node at (2.7,-.5) (m2) {$5$};
    \draw [-] (C) to (m1);
    \draw [-] (C) to (m2);    
    \end{tikzpicture}
=
    \begin{tikzpicture}[baseline=-3pt,label distance=0.5cm,thick,
    virtnode/.style={circle,draw,scale=0.5}, 
    nonvirt node/.style={circle,draw,fill=black,scale=0.5} ]
    \node [nonvirt node] (A) {};
    \node at (1,0) [nonvirt node] (B) {};
    \node at (2,0) [nonvirt node] (C) {};
    \draw [-] (A) to (B);
    \node at (-.7,.5) (n1) {$1$};
    \node at (-.7,-.5) (n2) {$2$};
    \node at (1,.7) (n3) {$4$};
    \node at (0,.3) (x) {$\kappa_3$};
    \draw [-] (A) to (n1);
    \draw [-] (A) to (n2);
    \draw [-] (B) to (n3);
    \draw [-] (B) to (C);
    
    \node at (.7,.3) (x1) {$h$};
    \node at (2.7,.5) (m1) {$3$};
    \node at (2.7,-.5) (m2) {$5$};
    \draw [-] (C) to (m1);
    \draw [-] (C) to (m2);    
    \end{tikzpicture}\,.
\end{equation*}
\end{example}

\begin{definition}
Consider the family $\Rels_{\kappa, \psi}^0$ of relations obtained by multiplying  the relations of $\kappa$ and $\psi$-classes from Lemmas \ref{Lem:psi12}, \ref{lem:removePsi} and \ref{Lem:remove_kappa} with an arbitrary monomial in $\kappa$ and $\psi$-classes.
Define the space $\Rels_{\kappa, \psi} \subset \StratAlg_{0,n}$ as the space of relations generated by $\Rels_{\kappa, \psi}^0$.

Define $\Rels_{\WDVV} \subset \StratAlg_{0,n}^\textup{nf}$ as the space of relations   obtained by gluing some WDVV relation  into a decorated prestable graph $[\Gamma, \alpha]$ in normal form at a vertex $v$ with $n(v) \geq 4$. In other words, it is the space of relations generated by the WDVV relation as in Definition \ref{Def:relsgeneratedby} where we restrict to $\Gamma, (\alpha_w)_{w \neq v}$ such that $[\Gamma, \alpha]$ (with $\alpha_v = 1$) is in normal form\footnote{Again we relax the condition of the $\alpha_w$ being monomials and allow $\alpha_w$ of the form $\psi_h^c + (-\psi_{h'})^{c}$ at vertices of valence $2$.}. 
\end{definition}
\begin{remark} \label{Rmk:howtouserelations}
Let us comment on the role of the sets of relations appearing above. The relations $\Rels_{\kappa, \psi}^0$ allow to write any monomial $\alpha$ in $\kappa$ and $\psi$-classes on $\M_{0,n}$ as a sum $\alpha = \alpha_0 + \beta$ of
\begin{itemize}
    \item a (possibly zero) monomial term $\alpha_0$ in $\kappa$ and $\psi$-classes such that the trivial prestable graph with decoration $\alpha_0$ is in normal form (implying that $\alpha_0$ restricts to a basis element of $\CH^*(\M_{0,n}^\textup{sm})$ as computed in Lemma \ref{Lem:Chowsmooth}),
    \item a sum $\beta$ of generators $[\Gamma_i, \alpha_i]$ supported in the boundary (i.e. with $\Gamma_i$ nontrivial).
\end{itemize}
The relations $\Rels_{\kappa, \psi}$ allow to do the above at each of the vertices of a decorated stratum class $[\Gamma, \alpha]$ and by a recursive procedure allow to write $[\Gamma,\alpha]$ as a sum of decorated strata classes in normal form. The relations in $\Rels_{\WDVV}$ then encode the remaining freedom to express relations among these classes in normal form generated by the $\WDVV$ relation. The following theorem and the course of its proof make precise the statement that these processes describe tautological relations on $\M_{0,n}$.
\end{remark}
\begin{theorem} \label{Thm:fullgenus0rels}
The kernel of the surjection $\StratAlg_{0,n} \to \CH^*(\M_{0,n})$ is given by $\Rels_{\kappa, \psi} + \Rels_{\WDVV}$. In particular, we have
\[\CH^*(\M_{0,n}) = \StratAlg_{0,n} / (\Rels_{\kappa, \psi} + \Rels_{\WDVV})\,.\]
\end{theorem}
We split the proof of the above theorem into two parts.
\begin{proposition} \label{Pro:kappapsirels}
The map $\StratAlg_{0,n}^\textup{nf} \hookrightarrow \StratAlg_{0,n} \to \StratAlg_{0,n} / \Rels_{\kappa, \psi}$ is surjective.
\end{proposition}
\begin{proof}
The statement says that we can use $\kappa, \psi$ relations on each vertex to express any decorated stratum class as a linear combination of stratum classes in normal form. This follows from Lemma \ref{Lem:psi12}, \ref{lem:removePsi} and \ref{Lem:remove_kappa} as described in Remark \ref{Rmk:howtouserelations}.
\end{proof}

\begin{theorem} \label{Thm:WDVVnfrels}
The kernel of the surjection $\StratAlg_{0,n}^\textup{nf} \to \CH^*(\M_{0,n})$ is given by $\Rels_{\WDVV}$.
\end{theorem}
\noindent
The proof is separated into several steps. The overall strategy is to stratify $\M_{0,n}$ by the number of edges of the prestable graph $\Gamma$ and use an excision sequence argument. For $p \geq 0$ we denote by $\M_{0,n}^{\geq p}$ the closed substack of $\M_{0,n}$ of curves with at least $p$ nodes. Similarly, we denote by $\M_{0,n}^{=p}$ the open substack of $\M_{0,n}^{\geq p}$ of curves with exactly $p$ nodes. It is clear that 
\[\M_{0,n}^{\geq p} \setminus \M_{0,n}^{= p} = \M_{0,n}^{\geq p+1}\]
and also 
\[\M_{0,n}^{=p} = \coprod_{\Gamma \in \mathcal{G}_p} \M^\Gamma\,,\]
where $\mathcal{G}_p$ is the set of prestable graphs of genus $0$ with $n$ markings having exactly $p$ edges.
For the strata space $\StratAlg_{0,n}$, consider the decomposition
\[\StratAlg_{0,n} = \bigoplus_{p \geq 0} \StratAlg_{0,n}^p\]
according to the number $p$ of edges of graph $\Gamma$.\footnote{This decomposition is not equal to the standard decomposition of $\StratAlg_{0,n}$ via degree of a class.} This descends to decompositions
\[\StratAlg_{0,n}^\textup{nf} = \bigoplus_{p \geq 0} \StratAlg_{0,n}^\textup{nf,$p$},\ \Rels_\WDVV = \bigoplus_{p \geq 0} \Rels_\WDVV^p\]
for $\StratAlg_{0,n}^\textup{nf}$ and $\Rels_\WDVV$. We note that $\Rels_\WDVV^p$ is exactly the space of relations obtained by taking a prestable graph $\Gamma$ with $p-1$ edges, a decoration $\alpha$ on $\Gamma$ in normal form and inserting a $\WDVV$ relation at a vertex $v_0 \in V(\Gamma)$ with $n(v_0) \geq 4$.

From Proposition \cite[Proposition 2.4]{BS1} and Lemma \ref{Lem:Chowfinitequotient} it follows that
 \begin{align} \label{eqn:M0nexcision2}
  \CH^*(\M_{0,n}^{=p}) &= \bigoplus_{\Gamma \in \mathcal{G}_p} \CH^*(\M^\Gamma) = \bigoplus_{\Gamma \in \mathcal{G}_p} \CH^*(\M_\Gamma^\textup{sm})^{\aut(\Gamma)}\,,\\
  \CH^*(\M_{0,n}^{=p},1) &= \bigoplus_{\Gamma \in \mathcal{G}_p} \CH^*(\M^\Gamma,1)=\bigoplus_{\Gamma \in \mathcal{G}_p} \CH^*(\M_\Gamma^\textup{sm},1)^{\aut(\Gamma)}\,.
\end{align}
Note that we have a natural map $\StratAlg_{0,n}^\textup{nf,$p$} \to \CH^*(\M_{0,n}^{\geq p})$.
\begin{lemma}\label{Lem:normalform_isom}
The composition $\StratAlg_{0,n}^\textup{nf,$p$} \to \CH^*(\M_{0,n}^{\geq p}) \to \CH^*(\M_{0,n}^{= p})$ is an isomorphism. 
\end{lemma}
\begin{proof} 
First we note that $\StratAlg_{0,n}^\textup{nf,$p$}$ decomposes into a direct sum of subspaces $\StratAlg_{0,n}^\textup{nf,$\Gamma$}$ indexed by prestable graphs $\Gamma \in \mathcal{G}_p$, according to the underlying prestable graph of the generators. The analogous decomposition of $\CH^*(\M_{0,n}^{= p})$ is given by \eqref{eqn:M0nexcision2}. Now for two non-isomorphic prestable graphs $\Gamma$ and $\Gamma'$ with the same number $p$ of edges, the induced map $\StratAlg_{0,n}^\textup{nf,$\Gamma$} \to \CH^*(\M^{\Gamma'})$ vanishes. Indeed, the locally closed substack $\M^{\Gamma'}$ is disjoint from the image of the gluing map $\xi_{\Gamma}$ and all generators of $\StratAlg_{0,n}^\textup{nf,$\Gamma$} $ are pushforwards under $\xi_\Gamma$. Thus we are reduced to showing that $\StratAlg_{0,n}^\textup{nf,$\Gamma$} \to \CH^*(\M^{\Gamma})$ is an isomorphism. The image of a generator $[\Gamma, \alpha]$ under this map is obtained by pushing forward $\alpha \in \CH^*(\M_\Gamma)$ to $\M_{0,n}^{=p}$ under $\xi_\Gamma$ and restricting to the open subset $\M^\Gamma$. From the cartesian diagram
\begin{equation*}
\begin{tikzcd}
 \M_\Gamma \arrow[r,"\xi_\Gamma"]& \M_{0,n}^{\geq p}\\
 \M_\Gamma^\textup{sm} \arrow[u,hook] \arrow[r,"\xi_\Gamma^\textup{sm}"] & \M^\Gamma = \M^\textup{sm}_\Gamma / \aut(\Gamma) \arrow[u,hook]
\end{tikzcd}
\end{equation*}
in which the vertical arrows are open embeddings, it follows that this is equivalent to first restricting $\alpha$ to $\M_\Gamma^\textup{sm}$ and then pushing forward to $\M^\Gamma$. As we saw in \eqref{eqn:M0nexcision2}, we can identify $\CH^*(\M^\Gamma)$ with the $\aut(\Gamma)$-invariant part of $\CH^*(\M_\Gamma^\textup{sm})$ via pullback under $\xi_\Gamma^\textup{sm}$. But clearly
\begin{equation} \label{eqn:averaging}
  (\xi_\Gamma^\textup{sm})^* [\Gamma, \alpha] =  (\xi_\Gamma^\textup{sm})^* (\xi_\Gamma^\textup{sm})_* \alpha = \sum_{\sigma \in \aut(\Gamma)} \sigma^* \alpha\,,
\end{equation}
where the automorphisms $\sigma$ act on $\M_\Gamma^\textup{sm}$ by permuting the factors. 

Now by Proposition~\ref{Pro:M0nsmCKP} we have
\[\CH^*(\M_\Gamma^\textup{sm}) = \bigotimes_{v \in V(\Gamma)} \CH^*(\M_{g(v),n(v)}^\textup{sm})\,.\] 
Thus it follows from Lemma \ref{Lem:Chowsmooth} that the set of all possible $\alpha$ such that $[\Gamma, \alpha]$ is in normal form is a basis of $\CH^*(\M_\Gamma^\textup{sm})$ and the action of automorphisms $\sigma$ of $\Gamma$ acts by permuting these basis elements. Hence a basis of the $\aut(\Gamma)$-invariant part of $\CH^*(\M_\Gamma^\textup{sm})$ is given by the sums of orbits of these basis elements (with the dimension of $\CH^*(\M_\Gamma^\textup{sm})^{\aut(\Gamma)}$ being the number of such orbits). 
Then the fact that we chose the basis of $\StratAlg_{0,n}^\textup{nf,$\Gamma$}$ to be the set of $[\Gamma, \alpha]$ \emph{up to isomorphism} implies via \eqref{eqn:averaging} that they exactly map to this basis of $\CH^*(\M_\Gamma^\textup{sm})^{\aut(\Gamma)}$.
\end{proof}
Next we realize the WDVV relation as the image of the connecting homomorphism $\partial$ of the excision sequence
 \begin{equation} \label{eqn:M0nkexcision}
  \CH^*(\M_{0,n}^{=p},1) \xrightarrow{\partial} \CH^{*-1}(\M_{0,n}^{\geq p+1}) \to \CH^*(\M_{0,n}^{\geq p}) \to \CH^*(\M_{0,n}^{=p}) \to 0\,.
\end{equation}
By\cite[Proposition 2.4]{BS1}, the stack $\M^{=p}_{0,n}$ is a  quotient stack and hence (\ref{eqn:M0nkexcision}) is exact by \cite[Proposition 4.2.1]{kreschartin}. 

Before we study the map $\partial$ in the sequence (\ref{eqn:M0nkexcision}), we consider an easier situation: we show that in the setting of the moduli spaces $\Mbar_{0,n}$ of stable curves, we can explicitly compute the connecting homomorphism $\partial$, see Proposition \ref{Pro:WDVV_Mbar} below. In the proof, we will need the following technical lemma about the connecting homomorphisms of excision sequences.

\begin{lemma} \label{Lem:connectinghomrestriction} 
Let $X$ be an equidimensional scheme and let
\[Z' \xrightarrow{j'} Z \xrightarrow{j} X\]
be two closed immersions. Consider the open embedding
\[U = X \setminus Z \xrightarrow{i} U' = X \setminus Z'.\]
Then we have a commutative diagram
\begin{equation}
\begin{tikzcd}
 \CH_n(U', 1) \arrow[d,"i^*"] \arrow[r,"\partial'"] & \CH_n(Z') \arrow[d,"j'_*"]\\
 \CH_n(U, 1) \arrow[r,"\partial"] & \CH_n(Z)
\end{tikzcd},
\end{equation}
where $\partial$ and $\partial'$ are the connecting homomorphisms for the inclusions of $U$ and $U'$ in $X$.
\end{lemma}
\begin{proof}
Elements of $\CH_n(U', 1)$ are represented by cycles $\sum a_i W_i$ on $U' \times \Delta^1$ with the $W_i$ of dimension $n+1$ intersecting the faces of $U' \times \partial\Delta^1$ properly. 
On the one hand, to evaluate the connecting homomorphism $\partial'$, we form the closures $\overline{W_i}$ in $X \times \Delta^1$ and take alternating intersections with faces. This is a sum of cycles of dimension $n$ supported on $Z' \times \partial\Delta^1$ and via $j'_*$ we regard it as a sum of cycles on $Z \times\partial\Delta^1$.

On the other hand, to evaluate $\partial \circ i^*$ we first restrict all $W_i$ to $U \times \Delta^1$, take the closure $\overline{W_i \cap U \times \Delta^1}$ and take alternating intersection with faces. But the only way that this closure can be different from $\overline{W_i}$ is when $W_i$ has generic point in $Z \times \Delta^1$. But then it defines an element of $z_n(Z, 1)$ and thus it maps to zero in $\CH_*(Z)$.
\end{proof}

\begin{proposition} \label{Pro:WDVV_Mbar}
For $n\geq 4$, the image of the connecting homomorphism $\partial$ of
\begin{equation*}
  \CH^1(\mathcal{M}_{0,n},1) \xrightarrow{\partial} \CH^0(\Mbar_{0,n}^{\geq 1}) \xrightarrow{\iota_*} \CH^1(\Mbar_{0,n}) \to 0 = \CH^1(\mathcal{M}_{0,n}) \,.    
\end{equation*}
is the set of WDVV relations, where we identify $\CH^0(\Mbar_{0,n}^{\geq 1})$ as the $\QQ$-vector space with basis given by boundary divisors of $\Mbar_{0,n}$.
\end{proposition}
\begin{proof}
First, we prove this proposition when $n=4$.  Identify $\Mbar_{0,4}\cong \PP^1$ and $\mathcal{M}_{0,4}\cong \mathbb{A}^1-\{0,1\}$. Then $\CH^*(\mathcal{M}_{0,4},\bullet)$ is a $\CH^*(k,\bullet)$-algebra generated by two elements $f_0$ and $f_1$ corresponding to two points in $\mathbb{A}^1$ (\cite{Chatz}). Fix a (non-canonical) isomorphism $\Delta^1\cong \mathbb{A}^1$ and set the two faces as $0$ and $1$. Consider a line $L_0$ through $(0,0)$ and $(1,1)$ in $\PP^1\times \Delta^1$ restricted to $(\PP^1-\{0,1,\infty\})\times \Delta^1$ as illustrated in Figure \ref{fig:L0picture}.

\begin{figure}[ht]
    \centering
    \begin{tikzpicture}
        \draw (0,0) rectangle (4,4);
        \draw[dashed] (0,0.7) -- (4,0.7);
        \draw[dashed] (0,1.7) -- (4,1.7) node[right]{$(\PP^1-\{0,1,\infty\})\times \Delta^1$};
        \draw[dashed] (0,3.4) -- (4,3.4);
        \draw[dashed] (3.4,0) -- (3.4,4);
        \draw (0,0) -- (4,4);
        \draw (2,2.8) node {$L_0$};
        \draw[fill, white] (0.7,0.7) circle (2 pt); \draw[black] (0.7,0.7) circle (2 pt);
        \draw[fill, white] (1.7,1.7) circle (2 pt); \draw[black] (1.7,1.7) circle (2 pt);
        \draw[fill, white] (3.4,3.4) circle (2 pt); \draw[black] (3.4,3.4) circle (2 pt);
        \draw (0,-1) -- (4,-1) node[right] {$\Delta^1$};
        \draw[fill, white] (3.4,-1) circle (2 pt); \draw[black] (3.4,-1) circle (2 pt);
        \draw (0.7,-0.9) -- (0.7, -1.1) node[below]{$0$};
        \draw (1.7,-0.9) -- (1.7, -1.1) node[below]{$1$};
    \end{tikzpicture}
    \caption{The line $L_0$ in $(\PP^1-\{0,1,\infty\})\times \Delta^1$}
    \label{fig:L0picture}
\end{figure}
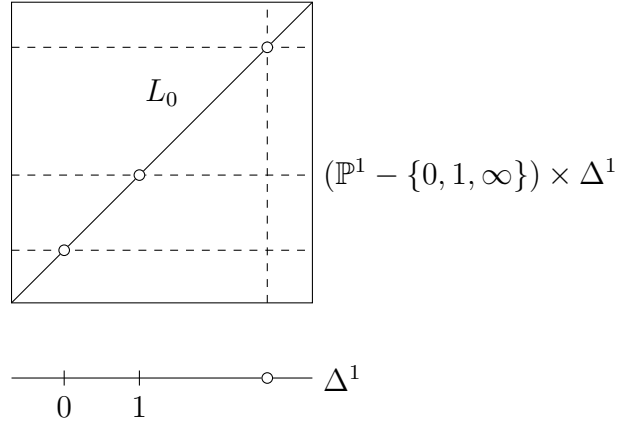\noindent
Then $f_0=[L_{0}]$ and $$\partial(L_0)= [0]-[1]\in \CH^0(\Mbar_{0,4}\setminus\mathcal{M}_{0,4})\,.$$
This is one of the WDVV relations on $\Mbar_{0,4}$ after identifying $[0], [1]$ and $[\infty]$ with three boundary strata in (\ref{eqn:wdvveq}). The second one is obtained from the generator $f_1$ in an analogous way, finishing the proof for $n=4$.

For the case of general $n \geq 4$, the space $\mathcal{M}_{0,n}$ is a hyperplane complement with hyperplanes associated to pairs of points that collide and there is a correspondence between generators of $\CH^1(\mathcal{M}_{0,n},1)$ and hyperplanes. On the one hand, the action of the symmetric group $S_n$ on $\mathcal{M}_{0,n}$ is transitive on the hyperplanes (and thus on the generators). On the other hand, we can obtain one of the hyperplanes as the pullback of a boundary point in $\mathcal{M}_{0.4}$ under the forgetful morphism $\pi: \Mbar_{0,n} \to \Mbar_{0,4}$ and thus, via the action of $S_n$, any hyperplane can be obtained under a suitable forgetful morphism to $\Mbar_{0,4}$ (varying the subset of four points to remember).

Note that the morphism $\pi$ is flat and that we have an open embedding $i: \mathcal{M}_{0,n} \to \pi^{-1}(\mathcal{M}_{0,4})$ and a closed embedding $j' : \pi^{-1}(\partial \Mbar_{0,4}) \to \partial \Mbar_{0,n}$. Combining the compatibility of the connecting homomorphism $\partial$ with flat pullback and Lemma \ref{Lem:connectinghomrestriction} above, we obtain a commutative diagram
\begin{equation*}
\begin{tikzcd}
 \CH^1(\mathcal{M}_{0,n},1) \arrow[r,"\partial"] & \CH^0(\partial \Mbar_{0,n})\\
 \CH^1(\pi^{-1}(\mathcal{M}_{0,4}),1) \arrow[r,"\partial"]  \arrow[u,"i^*"]& \CH^0(\pi^{-1}(\partial \Mbar_{0,4}))\arrow[u,"j'_*"]\\
 \CH^1(\mathcal{M}_{0,4},1) \arrow[r,"\partial"] \arrow[u,"\pi^*"]& \CH^0(\partial \Mbar_{0,4}) \arrow[u,"\pi^*"]
\end{tikzcd}
\end{equation*}
Therefore, since (under a suitable permutation of the markings) every generator of $\CH^1(\mathcal{M}_{0,n},1)$ can be obtained as the image of one of the generators of $\CH^1(\mathcal{M}_{0,4},1)$, the image of 
\[\CH^1(\mathcal{M}_{0,n},1) \to \CH^0(\partial \Mbar_{0,n})\]
is generated by WDVV relations on $\CH^0(\partial \Mbar_{0,4})$ pulled back via $\pi$.
\end{proof}
We extend the above computation to $\M_{0,n}$.
\begin{corollary}\label{Cor:WDVV_connecting_homomorphism}
For $n\geq 4$, the image of the connecting homomorphism $\partial$ of
\begin{equation*}
  \CH^{\ell+1}(\mathcal{M}_{0,n},1) \xrightarrow{\partial} \CH^\ell(\M_{0,n}^{\geq 1}) \xrightarrow{\iota_*} \CH^{\ell+1}(\M_{0,n}) \to 0 \,    
\end{equation*}
is the set of WDVV relations for $\ell=0$ and is zero for $\ell>0$.
\end{corollary}
\begin{proof} By (\ref{eq:firstHigherChow}), we have $\CH^{\ell+1}(\mathcal{M}_{0,n},1)=0$ when $\ell>0$ and hence $\partial$ is trivial in this range.

For the statement in degree $\ell=0$ we in fact ignore the definition of $\partial$ and the machinery of higher Chow groups and simply use that here the image of $\partial$ is given by the kernel of the map
\begin{equation}
    \CH^0(\M_{0,n}^{\geq 1}) \xrightarrow{\iota_*} \CH^{1}(\M_{0,n}),
\end{equation}
in other words by linear combinations of boundary divisors adding to zero in $\CH^{1}(\M_{0,n})$. Given such a relation, restricting to the open substack $\Mbar_{0,n}$ simply kills all unstable boundary divisors and by Proposition \ref{Pro:WDVV_Mbar} (or classical theory) the result is a combination of WDVV relations. After subtracting those from the original relation, we obtain a combination of unstable boundary divisors forming a relation. The proof is finished if we can show that this must be the trivial linear combination, i.e. that the unstable boundary divisors are linearly independent.

There are exactly $n+1$ strictly prestable graphs with one edge. Let $\Gamma_0$ be the prestable graph with a vertex of valence $1$ and $\Gamma_i$ be the semistable graph with the $i$-th leg on the semistable vertex. Suppose there is a linear relation
\[R= a_0 [\Gamma_0] + a_1 [\Gamma_1]+\ldots + a_n [\Gamma_n] = 0, \,\, a_i\in \mathbb{Q}\]
in $\CH^1(\M_{0,n})$, then we want to show that all $a_i=0$.

To see this we can simply construct test curves $\sigma_i : \mathbb{P}^1 \to \M_{0,n}$ intersecting precisely the divisor $[\Gamma_i]$ and none of the others. To obtain $\sigma_i$, start with the trivial family $\PP^1 \times \PP^1 \to \PP^1$ and a tuple of $n$ disjoint constant sections $p_1, \ldots, p_n$. Let $\sigma_0$ be defined by the family of prestable curves obtained by blowing up a point on $\PP^1 \times \PP^1$ away from any of the sections. For $1 \leq i \leq n$ we similarly obtain $\sigma_i$ by blowing up a point on the image of $p_i$ and taking the strict transform of the old section $p_i$. Then we have $0 = \sigma_i^* R = a_i$ for all $i$, finishing the proof.
\end{proof}

\begin{remark} \label{Rmk:ChowisomboundaryMbar0}
As a consequence of the above result, the map
\begin{equation*}
   \iota_*:  \CH^\ell(\M_{0,n}^{\geq 1}) \to \CH^{\ell+1}(\M_{0,n})
\end{equation*}
is an isomorphism in degree $\ell \geq 1$. Restricting to the locus of stable curves, the same proof implies that
\begin{equation*}
   \iota_*:  \CH^\ell(\partial \Mbar_{0,n}) \to \CH^{\ell+1}(\Mbar_{0,n})
\end{equation*}
is an isomorphism for $\ell \geq 1$. 

The surjectivity of $\iota_*$ comes from the excision sequence that we discussed. The injectivity can be explained from the results of Kontsevich and Manin (\cite{KontsevichManinGW}). Indeed, by Proposition \cite[Proposition B.19]{BS1}, the vector space $\CH^\ell(\partial \Mbar_{0,n})$ is generated by boundary strata of $\Mbar_{0,n}$ with at least one edge. To show injectivity of $\iota_*$, it is enough to show that any relation among boundary strata in $\CH^{\ell+1}(\Mbar_{0,n})$ is a pushforward of a relation holding already in the Chow group of $\partial\Mbar_{0,n}$. By \cite[Theorem 7.3]{KontsevichManinGW} the set of relations between boundary strata in $\CH^{\ell+1}(\Mbar_{0,n})$ is spanned by the relations obtained from gluing the WDVV relation into a vertex $v_0$ of a stable graph $\Gamma$ with at least $\ell$ edges.
When $\ell\geq 1$, this relation is a pushforward of a class 
\[\prod_{v\neq v_0}\left[\Mbar_{0,n(v)}\right] \times \textup{WDVV} \in \CH^1(\Mbar_{\Gamma})\]
where $\textup{WDVV}\in\CH^1(\Mbar_{0,n(v_0)})$ is the WDVV relation corresponding to the choice of four half edges at $v_0$. Under the gluing map, this class is a relation on $\partial\Mbar_{0,n}$. Therefore we get the injectivity of $\iota_*$.
\end{remark}



This corollary is enough to compute the connecting homomorphism in arbitrary degree.
\begin{proposition}\label{Pro:connecting_higherChow}
The image of $\partial\colon \CH^*(\M_{0,n}^{=p},1) \to \CH^{*-1}(\M_{0,n}^{\geq p+1})$ in \eqref{eqn:M0nkexcision} is equal to the image of the composition \[\Rels_\WDVV^{p+1} \to \StratAlg_{0,n}^\textup{nf,$p+1$} \to \CH^*(\M_{0,n}^{\geq p+1})\,.\]
Thus we can write
    \begin{equation} \label{eqn:Chow1quotient}
        \CH^*(\M_{0,n}^{\geq p+1})/\CH^*(\M_{0,n}^{=p},1) = \CH^*(\M_{0,n}^{\geq p+1})/\Rels_\WDVV^{p+1}\,.
    \end{equation}
\end{proposition}
\begin{proof} 
From the functoriality of higher Chow groups, it follows that we have a commutative diagram
\begin{equation}
\begin{tikzcd}
\bigoplus_\Gamma \CH^*(\M_{\Gamma}^\textup{sm},1) \arrow[r,"\bigoplus (\xi_\Gamma^\textup{sm})_*"] \arrow[d,"\bigoplus \partial_\Gamma"] & \bigoplus_\Gamma \CH^*(\M^{\Gamma},1) \arrow[r, equal] & \CH^*(\M_{0,n}^{=p},1) \arrow[d,"\partial"] \\
\bigoplus_\Gamma  \CH^{*-1}(\M_{\Gamma} \setminus\M_{\Gamma}^\textup{sm} ) \arrow[rr,"\sum (\xi_\Gamma)_*"] &   & \CH^{*-1}(\M_{0,n}^{\geq p+1})
\end{tikzcd}
\end{equation}
where the sums run over prestable graphs with exactly $p$ edges. By Remark \ref{Rmk:higherChowfinitequotient}, the maps $(\xi_\Gamma^\textup{sm})_*$ are surjective. Thus the image of $\partial$ is given by the sum of the images of the maps $(\xi_\Gamma)_* \circ \partial_\Gamma$.

From Remark \ref{Rmk:partialvanishonCHk1} we know that $\partial_\Gamma$ vanishes on the image of the map
\begin{equation} \label{eqn:trivparthigherchowMgammasmooth}
\CH_*(k,1) \otimes \CH_*(\M_{\Gamma}^\textup{sm}) \to \CH_*(\M_{\Gamma}^\textup{sm},1),
\end{equation}
and thus factors through its cokernel. 
On the other hand, it follows from Propositions \ref{Pro:M0nsmCKP_higherChow} and \ref{Pro:M0nsmCKP_stable_higherChow} that the cokernel of \eqref{eqn:trivparthigherchowMgammasmooth} is generated 
by classes coming from the direct sum
\begin{equation*}
   \bigoplus_{v\in V(\Gamma), n(v)\geq 4}\Big(\CH^1(\M^\textup{sm}_{0,n(v)},1)\otimes \bigotimes_{v'\in V(\Gamma),v\neq v'} \CH^*(\M^\textup{sm}_{0,n(v')})\Big)\,.
\end{equation*}
For an element $\alpha_v \otimes \bigotimes_{v\neq v'}\alpha_{v'}$ in $\CH^*(\M_{\Gamma}^\textup{sm},1)$, choose an extension $\overline{\alpha}_{v'}$  of each $\alpha_{v'}$ from $\M_{0,n(v')}^\textup{sm}$ to $\M_{0,n(v')}$.  By Lemma~\ref{Lem:HigherChowProductcompat} (a), the boundary map $\partial_\Gamma$ has the form
\[\partial_\Gamma\Big(\alpha_v \otimes \bigotimes_{v\neq v'}\alpha_{v'}\Big) = \partial(\alpha_v)\otimes \bigotimes_{v\neq v'}\overline{\alpha}_{v'}\,.\]
By Corollary \ref{Cor:WDVV_connecting_homomorphism}, the elements $\partial(\alpha_v)$ at vertices $v$ with $n(v)\geq 4$ are precisely the $\WDVV$ relations on $\M_{0,n(v)}$, whereas the classes $\overline{\alpha}_{v'}$ at other vertices $v'$ are exactly the types of decorations allowed in decorated strata classes in normal form. After pushing forward via $\xi_\Gamma$ this is precisely our definition of the relations $\Rels_\WDVV^{p+1}$.
\end{proof}

\begin{proof}[Proof of Theorem \ref{Thm:WDVVnfrels}]
Recall that by Lemma \ref{Lem:normalform_isom}, the composition $$\StratAlg_{0,n}^\textup{nf,$p$} \to \CH^*(\M_{0,n}^{\geq p}) \to \CH^*(\M_{0,n}^{= p})$$ is an isomorphism. Thus in the diagram
\begin{equation*}
\begin{tikzcd}[column sep = small]
 & & & \StratAlg_{0,n}^\textup{nf,$p$} \arrow[d,"\cong"] \arrow[dl] &\\
   \CH^*(\M_{0,n}^{=p},1) \arrow[r,"\partial"] & \CH^{*-1}(\M_{0,n}^{\geq p+1}) \arrow[r] & \CH^*(\M_{0,n}^{\geq p}) \arrow[r] & \CH^*(\M_{0,n}^{=p}) \arrow[r] & 0
\end{tikzcd}
\end{equation*}
we obtain a canonical splitting of the excision exact sequence \eqref{eqn:M0nkexcision} and thus we have
\begin{equation} \label{eqn:excisionsplitting}
        \CH^*(\M_{0,n}^{\geq p}) = \StratAlg_{0,n}^\textup{nf,$p$} \oplus \CH^{*-1}(\M_{0,n}^{\geq p+1})/\CH^*(\M_{0,n}^{=p},1)\,.
\end{equation}
Combining \eqref{eqn:excisionsplitting} and equation \eqref{eqn:Chow1quotient} from Proposition \ref{Pro:connecting_higherChow}, we see
\begin{equation} \label{eqn:excisionmaster}
    \CH^*(\M_{0,n}^{\geq p}) = \StratAlg_{0,n}^\textup{nf,$p$} \oplus \CH^{*-1}(\M_{0,n}^{\geq p+1})/\Rels_\WDVV^{p+1}\,.
\end{equation}
Applying \eqref{eqn:excisionmaster} for $p=0,1,2, \ldots$ we obtain
\begin{align*}
\CH^*(\M_{0,n}) &=\CH^*(\M_{0,n}^{\geq 0})\\
&= \StratAlg_{0,n}^\textup{nf,$0$} \oplus \CH^{*-1}(\M_{0,n}^{\geq 1})/\Rels_\WDVV^{1}\\
&= \StratAlg_{0,n}^\textup{nf,$0$} \oplus \left(\StratAlg_{0,n}^\textup{nf,$1$}/\Rels_\WDVV^{1} \right) \oplus \CH^{*-2}(\M_{0,n}^{\geq 2})/\Rels_\WDVV^{2} \\ &= \ldots \\
&=\bigoplus_{p \geq 0}\StratAlg_{0,n}^\textup{nf,$p$}/\Rels_\WDVV^{p}\\
&= \StratAlg_{0,n}^\textup{nf} / \Rels_\WDVV\,,
\end{align*}
finishing the proof.    
\end{proof}
Finally we end the proof of the main theorem.
\begin{proof}[Proof of Theorem \ref{Thm:fullgenus0rels}]
We know that the kernel of $\StratAlg_{0,n}\to \CH^*(\M_{0,n})$ contains $\Rels_{\kappa, \psi}$. We define \[\Rels_{\text{res}} = \ker\big( \StratAlg_{0,n}/\Rels_{\kappa, \psi} \to \CH^*(\M_{0,n}) \big)\,.\]
Likewise, by Theorem \ref{Thm:WDVVnfrels} we know that the kernel of $\StratAlg_{0,n}^\textup{nf} \to \CH^*(\M_{0,n})$ is equal to $\Rels_{\WDVV}$. We then obtain a diagram of morphisms with exact rows
\begin{equation}
\begin{tikzcd}
 0 \arrow[r] & \Rels_{\WDVV} \arrow[r] \arrow[d,dashed] & \StratAlg_{0,n}^\textup{nf} \arrow[r] \arrow[d,two heads] & \CH^*(\M_{0,n}) \arrow[d, equal] \arrow[r] & 0\\
 0 \arrow[r]& \Rels_{\text{res}}\arrow[r] & \StratAlg_{0,n}/\Rels_{\kappa, \psi}\arrow[r] & \CH^*(\M_{0,n}) \arrow[r] & 0
\end{tikzcd}
\end{equation}
where the arrow $\StratAlg_{0,n}^\textup{nf} \to \StratAlg_{0,n}/\Rels_{\kappa, \psi}$ is surjective by Proposition \ref{Pro:kappapsirels}.
A short diagram chase shows that $\Rels_{\WDVV}$ factors through $\Rels_{\text{res}}$ via the dashed arrow. By the four-lemma, the map $\Rels_{\WDVV} \to \Rels_{\text{res}}$ is surjective. This clearly implies the statement of the theorem.
\end{proof}
\subsection{Relation to previous works} \label{Sect:Relpreviouswork}
Let us start by pointing out several results in Gromov--Witten theory, studying intersection numbers on moduli spaces of stable maps, which can be seen as coming from results about the tautological ring of $\M_{g,n}$.

\begin{example} \label{Exa:LeePandharipande}
In $\cite{LeePandharipande04}$, degree one relations on the {\em moduli space of stable maps to a projective space} $\Mbar_{0,n}(\PP^N,d)$ are used to reduce two pointed genus $0$ potentials to one pointed genus $0$ potentials. Theorem 1.(2) from \cite{LeePandharipande04} can be obtained by the pullback of the relation in Lemma \ref{Lem:psi12} along the forgetful morphism \[\Mbar_{0,n}(\PP^N,d)\to \M_{0,2}\,.\] 
Similarly, Theorem 1.(1) of \cite{LeePandharipande04} can be obtained from Lemma \ref{lem:removePsi} on $\M_{0,3}$. The relevant computations are given explicitly in  \cite{bae2019tautological}.
\end{example}
From Theorem~\ref{Thm:fullgenus0rels} we see that any universal relation in  the Gromov--Witten theory of genus $0$ obtained from tautological relations on $\M_{0,n}$ must follow either from the WDVV relation (\ref{eqn:wdvveq}) or the relation (\ref{eqn:psi12}) between $\psi$ and boundary classes on $\M_{0,2}$. 

\vspace{5mm}

Apart from applications to Gromov--Witten theory, there are several results in the literature which compute the Chow groups of some strict open subloci of $\M_{0,n}$. 
\begin{example}
Restricting to the locus $\Mbar_{0,n}\subset\M_{0,n}$ of stable curves,  Theorem \ref{Thm:fullgenus0rels} specializes to the classical result in \cite{keel,KontsevichManinGW} that all relations between undecorated strata of $\Mbar_{0,n}$ are additively generated by the WDVV relations. 
\end{example}


\begin{example} \label{Exa:Oesinghaus}
In \cite{oesinghaus}, Oesinghaus computes the Chow ring (with integer coefficients) of the open substack $\TT$ of $\M_{0,3}$ of curves with prestable graph of the form
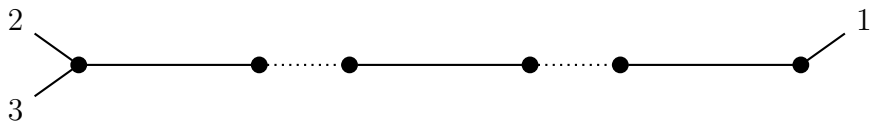
\begin{equation*}
\begin{tikzpicture}[scale=1.2, baseline=-3pt,label distance=0.3cm,thick,
    virtnode/.style={circle,draw,scale=0.5}, 
    nonvirt node/.style={circle,draw,fill=black,scale=0.5} ]
    \node at (-6,0) [nonvirt node] (F) {};
    \node at (-4,0) [nonvirt node] (E) {};
    \node at (-3,0) [nonvirt node] (D) {};
    \node at (-1,0) [nonvirt node] (C) {};
    \node [nonvirt node] (A) {};
    \node at (2,0) [nonvirt node] (B) {};
    \draw [-] (A) to (B);
    \draw [-, dotted] (C) to (A);
    \draw [-] (D) to (C);
    \draw [-, dotted] (E) to (D);
    \draw [-] (F) to (E);
    \node at (2.7,.5) (m1) {$1$};
    \draw [-] (B) to (m1);
    \node at (-6.7,.5) (n1) {$2$};
    \draw [-] (F) to (n1);
    \node at (-6.7,-.5) (n2) {$3$};
    \draw [-] (F) to (n2);    
    \end{tikzpicture}
\end{equation*}
where we denote by $\Gamma_k$ the graph of the shape above with $k$ edges (for $k \geq 0$). 
Oesinghaus shows that the Chow ring $\CH^*(\TT)$ is given by the ring QSym of {\em quasi-symmetric functions} on the index set $\mathbb{Z}_{>0}$. QSym can be seen as the subring of $\mathbb{Q}[\alpha_1, \alpha_2, \ldots]$ with additive basis given by
\begin{equation}
    M_{J} = \sum_{i_1 < \ldots < i_k} \alpha_{i_1}^{j_1} \cdots \alpha_{i_k}^{j_k} \text{ for }k \geq 1, J = (j_1, \ldots, j_k) \in \mathbb{Z}_{\geq 1}^{k}.
\end{equation}
Under the isomorphism $\CH^*(\TT) \cong \text{QSym}$, the element $M_J$ is a basis element of degree $\sum_\ell j_\ell$ in the Chow group of $\TT$. 
As we explain in \cite[Example 4.3]{BS1}, the cycle $M_J$
corresponds to the tautological class supported on the stratum $\M^{\Gamma_k}$ given by
\begin{equation} \label{eqn:oesinghauscomparisonpicture}
\begin{tikzpicture}[scale=1.2, baseline=-3pt,label distance=0.3cm,thick,
    virtnode/.style={circle,draw,scale=0.5}, 
    nonvirt node/.style={circle,draw,fill=black,scale=0.5} ]
    \node at (-6,0) [nonvirt node] (F) {};
    \node at (-4,0) [nonvirt node] (E) {};
    \node at (-3,0) [nonvirt node] (D) {};
    \node at (-1,0) [nonvirt node] (C) {};
    \node [nonvirt node] (A) {};
    \node at (2,0) [nonvirt node] (B) {};
    \draw [-] (A) to (B);
    \draw [-, dotted] (C) to (A);
    \draw [-] (D) to (C);
    \draw [-, dotted] (E) to (D);
    \draw [-] (F) to (E);
    \node at (-5,.3) {$(-\psi - \psi')^{j_1-1}$};
    \node at (-2,.3) {$(-\psi - \psi')^{j_\ell-1}$};
    \node at (1,.3) {$(-\psi - \psi')^{j_k-1}$};
    \node at (2.7,.5) (m1) {$1$};
    \draw [-] (B) to (m1);
    \node at (-6.7,.5) (n1) {$2$};
    \draw [-] (F) to (n1);
    \node at (-6.7,-.5) (n2) {$3$};
    \draw [-] (F) to (n2);    
    \end{tikzpicture}
\end{equation}
Using the correspondence, we can verify several of the results of our paper in this particular example. Indeed, one can use Theorem \ref{Thm:WDVVnfrels} to verify that the classes \eqref{eqn:oesinghauscomparisonpicture} form a basis of $\CH^*(\TT)$. For this, one observes that decorated strata in normal form generically supported on $\TT$ must have underlying graph $\Gamma_k$ for some $k$, with trivial decoration on the valence $3$ vertex and decorations $(-\psi_h - \psi_{h'})^{c_\ell}$ on the valence $2$ vertices. Since every term appearing in a WDVV relation has at least two vertices of valence at least $3$, all these relations restrict to zero on $\TT$ and thus the above generators form a basis by Theorem \ref{Thm:WDVVnfrels}. Note that the form of these generators in normal form is not quite the same as the one shown in \eqref{eqn:oesinghauscomparisonpicture}, but a small combinatorial argument shows that the two bases can be converted to each other by using the relation \eqref{eqn:psi12} between $\psi$-classes and the boundary divisor on $\M_{0,2}$.

Note that  \cite{oesinghaus} also computes the Chow group of the semistable loci $\M_{0,2}^{ss}$ and $\M_{0,3}^{ss}$. By a straightforward generalization of the discussion above, a correspondence of the generators in \cite{oesinghaus} to the tautological generators on these spaces, as well as a comparison of relations can be established.

\end{example}

\begin{example} \label{Exa:fulghesu}
In a series of paper \cite{fulghesurational, fulghesutaut, fulghesu3nodes}, Fulghesu presented a computation of the Chow ring of the open substack $\M_{0}^{\leq 3} \subset \M_{0}$ of rational curves with at most $3$ nodes, as an explicit algebra with $10$ generators and $11$ relations. Some of the generators are given by $\kappa$-classes, some are classes of strata and others are decorated classes supported on strata. 

Establishing a precise correspondence to the generators and relations discussed in our paper is challenging due to the complexity of the involved combinatorics. However, as a nontrivial check of our results we can compare the dimensions $\dim \CH^d(\M_{0}^{\leq 3})$ of the graded pieces of the Chow ring. Given any open substack $U \subseteq \M_{0}$, we package the ranks of the Chow groups of $U$ in the generating function 
\[H_U = \sum_{d \geq 0} \dim \CH^d(U) t^d, \]
which is the \emph{Hilbert series} of the graded ring $\CH^*(U)$.

In \cite{fulghesu3nodes}, Fulghesu computes the Chow rings of the open substacks $U=\M_0^{\leq e}$ for $e=0,1,2,3$ in terms of generators and relations. 
    Using the software Macaulay2 \cite{M2} we can compute\footnote{The output of the relevant computation can be found \href{https://cocalc.com/share/f765c8c72a7372905a4d4d2d0c8606ad2864fecd/FulghesuComputation.txt?viewer=share}{here}.} the Hilbert functions $H_U^F$ of the graded algebras given by Fulghesu. We list them in Figure \ref{fig:fulghesunumbers}.
\begin{figure}[htb]
    \begin{center}
    \begin{tabular}{c|cl}
       $U$  & $H_U^F$ \\
       \hline$\M_0^{\leq 0}$  & $\frac{1}{1-t^2}$ &$= 1 + t^2 + t^4 + \ldots$\\
       $\M_0^{\leq 1}$  & $\frac{1}{(1-t^2)(1-t)}$ &$= 1 + t + 2 t^2 + 2 t^3  + 3 t^4 + \ldots$\\
       $\M_0^{\leq 2}$  & $\frac{t^4+1}{(1-t^2)^2(1-t)}$ &$= 1 + t + 3 t^2 + 3 t^3  + 7 t^4 + 7 t^5 + 13 t^6  + 13 t^7 + 21 t^8 + \ldots$\\
       $\M_0^{\leq 3}$ & $H'$ &$= 1 + t + 3 t^2 + 5 t^3  + 10 t^4 + 15 t^5 + 26 t^6  + 36 t^7 + \mathbf{55} t^8 + \ldots$
    \end{tabular}
    \end{center}
    \caption{The Hilbert series of the Chow rings of open substacks $U$ of $\M_0$, as computed by Fulghesu; for space reasons we don't write the full formula for the rational function $H'$, only giving the expansion}
    \label{fig:fulghesunumbers}
\end{figure}    

On the other hand, since for stable graphs with at most three edges no WDVV relations can appear, Theorem \ref{Thm:WDVVnfrels} implies that the Chow group $\CH^*(\M_0^{\leq e})$ is equal to the subspace of the strata algebra $\StratAlg_{0,0}$ spanned by decorated strata in normal form with at most $e$ nodes (for $e \leq 3$).
By some small combinatorial arguments, this allows us to compute the Hilbert functions $H_U$ of the spaces $U=\M_0^{\leq e}$:

$\mathbf{e=0}$
For $U=\M_0^{\leq 0}$ the only generators in normal form are the classes $\kappa_2^a$, existing in every even degree $d=2a$, so that the generating function is given by
\[H_{\M_0^{\leq 0}} = 1+t^2 + t^4 + \ldots = \frac{1}{1-t^2},\]
recovering the formula from Figure \ref{fig:fulghesunumbers}.

$\mathbf{e=1}$
On $U=\M_0^{\leq 0}$ we get additional generators 
\[[\Gamma_1, \psi_{h_1}^a \psi_{h_2}^b] \text{ for }\Gamma_1 = 
\begin{tikzpicture}[scale=0.7, baseline=-3pt,label distance=0.3cm,thick,
    virtnode/.style={circle,draw,scale=0.5}, 
    nonvirt node/.style={circle,draw,fill=black,scale=0.5} ]
    \node at (0,0) [nonvirt node] (A) {};
    \node at (2,0) [nonvirt node] (B) {};
    \draw [-] (A) to (B);
    \node at (0.5,0.4) {$h_1$};
    \node at (1.5,0.4) {$h_2$};
    \end{tikzpicture}\,.\]
Since the automorphism group of $\Gamma_1$ exchanges $h_1, h_2$, the numbers $a,b$ above are only unique up to ordering. We get a canonical representative by requiring $a \leq b$. Overall, we obtain the generating series
\begin{align*}
H_{\M_0^{\leq 1}} = H_{\M_0^{\leq 0}} + \sum_{0 \leq a \leq b} t^{a+b+1} &= \frac{1}{1-t^2} + t \sum_{a \geq 0} \sum_{c \geq 0} t^{a+(a+c)}\\
&= \frac{1}{1-t^2} + t \left(\sum_{a \geq 0} t^{2a} \right) \left(\sum_{c \geq 0} t^{c} \right)\\
&= \frac{1}{1-t^2} + t \frac{1}{1-t^2} \frac{1}{1-t} = \frac{1}{(1-t^2)(1-t)},
\end{align*}
where we used the substitution $b=a+c$. Again we recover the formula from Figure \ref{fig:fulghesunumbers}. 

$\mathbf{e=2}$
The additional generators for $U = \M_0^{\leq 2}$ are given by
\[[\Gamma_2, \psi_{h_1}^a (\psi_{h_2}^b + (-\psi_{h_3})^b) \psi_{h_4}^c] \text{ for }\Gamma_2 = 
\begin{tikzpicture}[scale=0.7, baseline=-3pt,label distance=0.3cm,thick,
    virtnode/.style={circle,draw,scale=0.5}, 
    nonvirt node/.style={circle,draw,fill=black,scale=0.5} ]
    \node at (0,0) [nonvirt node] (A) {};
    \node at (2,0) [nonvirt node] (B) {};
    \node at (4,0) [nonvirt node] (C) {};
    \draw [-] (A) to (B);
    \draw [-] (B) to (C);
    \node at (0.5,0.4) {$h_1$};
    \node at (1.5,0.4) {$h_2$};
    \node at (2.5,0.4) {$h_3$};
    \node at (3.5,0.4) {$h_4$};    
    \end{tikzpicture}\,.
\]
The automorphism group of $\Gamma_2$ exchanges $h_1, h_4$ and $h_2, h_3$, so $a,c$ are only well-defined up to ordering. Moreover, for $a=c$ and $b=2\ell+1$ odd, this symmetry implies
\[[\Gamma_2, \psi_{h_1}^a (\psi_{h_2}^b + (-\psi_{h_3})^b) \psi_{h_4}^a] = - [\Gamma_2, \psi_{h_1}^a (\psi_{h_2}^b + (-\psi_{h_3})^b) \psi_{h_4}^a],\]
so the corresponding generator vanishes. Overall, the numbers of basis elements supported on $\Gamma_2$ have generating series
\[t^2 \cdot \big(\underbrace{\sum_{0 \leq a \leq c} \sum_{b \geq 0} t^{a+b+c}}_{=\frac{1}{(1-t^2) (1-t)^2}} - \underbrace{\sum_{a \geq 0} \sum_{\ell \geq 0} t^{2a+2\ell+1}}_{=\frac{t}{(1-t^2)^2}} \big) = \frac{t^2(t^2+1)}{(1-t)(1-t^2)^2}\,. \]
Adding this to the generating series for $\M_0^{\leq 1}$ we obtain the formula
\[
H_{\M_0^{\leq 2}} = H_{\M_0^{\leq 1}} + \frac{t^2(t^2+1)}{(1-t)(1-t^2)^2} = \frac{t^4+1}{(1-t^2)^2(1-t)},
\]
again obtaining the same formula as in Figure \ref{fig:fulghesunumbers}.

$\mathbf{e=3}$
For the full locus $U = \M_0^{\leq 3}$ a discrepancy between our results and Fulghesu's computations appears. 
There are two new types of generators appearing: firstly we have
\[[\Gamma_3', \psi_{h_1}^a \psi_{h_2}^b \psi_{h_3}^c ] \text{ for }\Gamma_3' = 
\begin{tikzpicture}[scale=0.7, baseline=-3pt,label distance=0.3cm,thick,
    virtnode/.style={circle,draw,scale=0.5}, 
    nonvirt node/.style={circle,draw,fill=black,scale=0.5} ]
    \node at (0,0) [nonvirt node] (A) {};
    \node at (2,0) [nonvirt node] (B) {};
    \node at (-1,{sqrt(3)}) [nonvirt node] (C) {};
    \node at (-1,{-sqrt(3)}) [nonvirt node] (D) {};
    \draw [-] (A) to (B);
    \draw [-] (A) to (C);
    \draw [-] (A) to (D);
    \node at (1.5,0.4) {$h_1$};
    \node at (-0.3,1.4) {$h_2$};
    \node at (-0.3,-1.4) {$h_3$};
    \end{tikzpicture}
\]
giving a contribution of
\[
t^3\sum_{0 \leq a \leq b \leq c} t^{a+b+c} = \frac{t^3}{(1-t^3)(1-t^2)(1-t)}
\]
to the generating series. The second type of generator is
\[[\Gamma_3'', \psi_{h_1}^a (\psi_{h_2}^b+ (-\psi_{h_3})^b) ( (-\psi_{h_4})^c + \psi_{h_5}^c) \psi_{h_6}^d ] \text{ for }\Gamma_3'' = 
\begin{tikzpicture}[scale=0.7, baseline=-3pt,label distance=0.3cm,thick,
    virtnode/.style={circle,draw,scale=0.5}, 
    nonvirt node/.style={circle,draw,fill=black,scale=0.5} ]
    \node at (0,0) [nonvirt node] (A) {};
    \node at (2,0) [nonvirt node] (B) {};
    \node at (4,0) [nonvirt node] (C) {};
    \node at (6,0) [nonvirt node] (D) {};
    \draw [-] (A) to (B);
    \draw [-] (B) to (C);
    \draw [-] (C) to (D);
    \node at (0.5,0.4) {$h_1$};
    \node at (1.5,0.4) {$h_2$};
    \node at (2.5,0.4) {$h_3$};
    \node at (3.5,0.4) {$h_4$};      
    \node at (4.5,0.4) {$h_5$};
    \node at (5.5,0.4) {$h_6$};         
    \end{tikzpicture}
\]
Since $\Gamma_3''$ again has an automorphism of order $2$, we count such generators using a trick: if the vertices of $\Gamma_3''$ were ordered, the generating series would be
\[t^3 \sum_{a,b,c,d \geq 0} t^{a+b+c+d} = \frac{t^3}{(1-t)^4}.\]
Due to the automorphism, we counted almost all the generators twice, \emph{except} those fixed by the automorphism, for which $(a,b,c,d)=(a,b,b,a)$ and whose generating series is
\[t^3 \sum_{a,b \geq 0} t^{2a + 2b} = \frac{t^3}{(1-t^2)^2}.\]
Adding these two series, we count every generator twice, so we obtain the correct count after dividing by two. Overall we get
\begin{align*}
H_{\M_0^{\leq 3}} &= H_{\M_0^{\leq 2}} + \frac{t^3}{(1-t^3)(1-t^2)(1-t)} + \frac{1}{2} \left(\frac{t^3}{(1-t)^4} + \frac{t^3}{(1-t^2)^2} \right) \\  &= \frac{t^6 + t^5 + 2 t^4 + t^3 +1}{(1-t^2)^2(1-t)(1-t^3)}.
\end{align*}
However, expanding this series we obtain
\begin{equation}
\frac{t^6 + t^5 + 2 t^4 + t^3 +1}{(1-t^2)^2(1-t)(1-t^3)} = 1 + t + 3 t^2 + 5 t^3  + 10 t^4 + 15 t^5 + 26 t^6  + 36 t^7 + \mathbf{54} t^8 + \ldots
\end{equation}
Comparing with the expansion of the corresponding function $H'$ in Figure \ref{fig:fulghesunumbers} we see that the coefficient of $t^8$ is $55$ for Fulghesu and $54$ for us. We used a  modified version of the software package admcycles \cite{admcycles} for the open-source software SageMath \cite{sage} to verify the number $54$ above.

After revisiting Fulghesu's proof, we think we can explain this discrepancy from a relation that was missed in \cite{fulghesu3nodes}. In the notation of this paper, we claim that there is a relation
\begin{equation} \label{eqn:fulghesumissingrel}
r \cdot q - \gamma_3'' \cdot s + 2 u \cdot \gamma_2 - \gamma_3'' \cdot q \cdot \kappa_2 - s \cdot \gamma_2 \cdot \kappa_1 = 0 \in \CH^8(\M_0^{\leq 3}).
\end{equation}
Here the classes $r,s,u,\gamma_3''$ are supported on the closed stratum $\M^{\Gamma_3''} \subset \M_0^{\leq 3}$.
The relation \eqref{eqn:fulghesumissingrel} follows from the description of the Chow ring $\CH^*(\M^{\Gamma_3''})$ and the formulas for restrictions of classes $q,\gamma_3'',\gamma_2, \kappa_2, \kappa_1$ to $\M^{\Gamma_3''}$ computed in \cite[Section 6.2]{fulghesu3nodes}. On the other hand, using Macaulay2 we verified that the relation \eqref{eqn:fulghesumissingrel} is \emph{not} contained in the ideal of relations given in \cite[Theorem 6.3]{fulghesu3nodes}. Adding this missing relation, we obtain the correct rank $54$ for $\CH^8(\M_0^{\leq 3})$. 

Our numerical experiments indicated that there are further relations missing in degrees $d>9$. So while the general proof strategy of \cite{fulghesu3nodes} seems sound, more care needed is needed in the final step of the computation. 

\subsection{Chow rings of open substacks of \texorpdfstring{$\M_{0,n}$}{M0n} - finite generation and Hilbert series} \label{Sect:opensubstack}
In the previous section, we saw some explicit computations for Chow groups $\CH^*(U)$ of of open substacks $U \subset \M_{0,n}$ and their Hilbert series
\[H_U = \sum_{d \geq 0} \dim_{\mathbb{Q}} \CH^d(U) t^d.\]
For $U= \M_0^{\leq e}$ and $e=0,1,2$ we have that $\CH^*(U)$ is a finitely generated graded algebra by the results of \cite{fulghesu3nodes}. But recall that any such algebra, having generators in degrees $d_1, \ldots, d_r$ has a Hilbert series which is the expansion (at $t=0$) of a rational function $H(t)$ of the form
\[H(t) = \frac{Q(t)}{\prod_{i=1}^r (1-t^{d_i} )} \text{ for some }Q(t) \in \mathbb{Z}[t]\]
(see \cite[Theorem 13.2]{matsumura}).
This explains the shape of the Hilbert functions of $\M_0^{\leq e}$ from Figure \ref{fig:fulghesunumbers}. We remark here, that all functions $H(t)$ of the above form have poles only at roots of unity.

On the other hand, for the open substack $\mathcal T \subset \M_{0,3}$ studied by Oesinghaus, we saw 
$\CH^*(\TT) \cong \text{QSym}$, where the algebra $\text{QSym}$ had an additive basis element $M_J$ in degree $d$ for each composition $J$ of $d$. Since for $d \geq 1$ the number of compositions of $d$ is $2^{d-1}$, the Hilbert series of the Chow ring of $\TT$ is given by
\[H_{\mathcal T} = 1 + \sum_{d \geq 0} 2^{d-1} t^d  = 1+ \frac{t}{1-2t} = \frac{1-t}{1-2t}.\]
From this we can see two things:
\begin{itemize}
    \item The Chow ring $\CH^*(\TT)$ cannot be a finitely generated algebra, since the function $H_{\TT}$ has a pole at $1/2$, which is not a root of unity.
    \item On the other hand, we still have that $H_\TT$ is the expansion of a rational function, even though $\TT$ is not even of finite type.
\end{itemize}
The above observations lead to the following two questions.
\begin{question} \label{quest:finitelygeneratedAlgebra}
Is it true that for $U \subset \M_{0,n}$ an open substack of finite type, the Chow ring $\CH^*(U)$ is a finitely generated algebra?
\end{question}
\begin{question} \label{quest:rationalHilbertSeries}
Is it true that for $U \subset \M_{0,n}$ any open substack which is a union of strata $\M^\Gamma$, the Hilbert series $H_U$ is the expansion of a rational function at $t=0$?
\end{question}
For the first question, we note that by Theorem \ref{Thm:M0ngenerators} we know that $\CH^*(U)$ is additively generated by possibly infinitely many decorated strata $[\Gamma, \alpha]$, supported on \emph{finitely many} prestable graphs $\Gamma$. It is far from obvious whether we can obtain all of them multiplicatively from a finite collection of $[\Gamma_i, \alpha_i]$.

For the second question, we observe that it would be implied for all finite type open substacks $U$ of $\M_{0,n}$ assuming a positive answer to the first question. Further evidence is provided by the results from \cite{oesinghaus}: as we saw above, the open substack $\TT \subset \M_{0,3}$ has a rational generating series $H_\TT$. In fact, as mentioned above Oesinghaus computes the Chow ring for the the entire semistable locus in $\M_{0,2}$ and $\M_{0,3}$ (see \cite[Corollary 2,3]{oesinghaus}) and obtains
\[\CH^*(\M_{0,2}^\textup{ss}) = \mathrm{QSym} \otimes_\mathbb{Q} \mathbb{Q}[\beta],\  \CH^*(\M_{0,3}^\textup{ss}) = \mathrm{QSym} \otimes_\mathbb{Q}\mathrm{QSym} \otimes_\mathbb{Q}\mathrm{QSym}.\]
Since we know that the Hilbert series of $\mathrm{QSym}$ is $(1-t)/(1-2t)$ and the Hilbert series of $\mathbb{Q}[t]$ is $1/(1-t)$ and that Hilbert series are multiplicative under tensor products, we easily see that 
\[H_{\M_{0,2}^\textup{ss}} = \frac{1}{1-2t},\  H_{\M_{0,3}^\textup{ss}} = \frac{(1-t)^3}{(1-2t)^3}.\]
So Question \ref{quest:rationalHilbertSeries} has a positive answer for the non-finite type substacks of semistable points in $\M_{0,2}$ and $\M_{0,3}$.

To finish this section, we want to record some numerical data about the Chow groups of the full stacks $\M_{0,n}$. Using Theorems \ref{Thm:M0ngenerators} and \ref{Thm:fullgenus0rels} these groups have a completely combinatorial description. This has been implemented in a modified version of the software package \texttt{admcycles} \cite{admcycles}, which can enumerate prestable graphs, decorated strata in normal form and the relations $\Rels_{\kappa, \psi},  \Rels_{\WDVV}$ between them. Thus, from linear algebra we can compute the ranks of Chow groups of $\M_{0,n}$ in many cases. We record the results in Figure \ref{fig:Chowranks}.

\begin{figure}[htb]
    \centering
\begin{tabular}{c|ccccccccc}
d & n=0 & n=1 & n=2 & n=3 & n=4 & n=5 & n=6 & n=7 & n=8 \\
\hline 0 & 1 & 1 & 1 & 1 & 1 & 1 & 1 & 1 & 1\\
1 & 1 & 2 & 3 & 4 & 6 & 11 & 23 & 50 & 108 \\
2 & 3 & 5 & 9 & 16 & 33 & 80 & 215 & 621 & 1900\\
3 & 5 & 12 & 27 & 62 & 162 & 481 & 1572 & &\\
4 & 13 & 32 & 84 & 235 & 739 & 2594 &  & &\\
5 & 27 & 84 & 263 & 875 & 3219 &  &  & &\\
6 & 70 & 234 & 837 & 3219 &  &  &  & &\\
7 & 166 & 656 & 2683 &  &  &  &  & &\\
8 & 438 & 1892 &  &  &  &  &  & &\\
9 & 1135 &  &  &  &  &  &  & & \\
10 & 3081 &  &  &  &  &  &  & &
\end{tabular}
    \caption{The rank of the Chow groups $\CH^d(\M_{0,n})$}
    \label{fig:Chowranks}
\end{figure}


\end{example}

\section{Comparison with the tautological ring of the moduli of stable curves}\label{sec:comparison}
\subsection{Injectivity of pullback by forgetful charts} \label{Sec:forgetfulpullback}
Assume we are in the stable range $2g-2+n>0$ so that the moduli space $\Mbar_{g,n}$ is nonempty. 
Since $\Mbar_{g,n} \subset \M_{g,n}$ is an open substack, the Chow groups of $\M_{g,n}$ determine those of $\Mbar_{g,n}$: we have that $\CH^*(\Mbar_{g,n})$ is the quotient of $\CH^*(\M_{g,n})$ by the span of classes supported on the strictly unstable locus.

Restricting to the subrings of tautological classes, we note that the tautological ring $\R^*(\Mbar_{g,n})$ of $\Mbar_{g,n}$ is the subring of $\CH^*(\Mbar_{g,n})$ given by the restriction of $\R^*(\M_{g,n})\subseteq \CH^*(\M_{g,n})$ under the open embedding $i\colon \Mbar_{g,n}\hookrightarrow \M_{g,n}$. 
Thus the tautological ring $\R^*(\M_{g,n})$ determines $\R^*(\Mbar_{g,n})$  since the composition
\[\R^*(\Mbar_{g,n}) \xrightarrow{\st^*} \R^*(\M_{g,n}) \xrightarrow{i^*} \R^*(\Mbar_{g,n})\] 
is the identity and thus $\R^*(\Mbar_{g,n}) \xrightarrow{\st^*} \R^*(\M_{g,n})$ is injective.

It is an interesting question whether the converse is true: do the Chow (or tautological) rings of the moduli spaces of stable curves determine the Chow (or tautological) ring of $\M_{g,n}$? The following conjecture gives a precise way in which this could be true. 
\begin{conjecture} \label{conj:pullbackinjective}
Let $(g,n) \neq (1,0)$, then for a fixed $d \geq 0$ there exists $m_0 \geq 0$ such that for any $m \geq m_0$, the forgetful morphism
\[F_m : \Mbar_{g,n+m} \to \M_{g,n}\]
satisfies that the pullback
\[F_m^* : \CH^d(\M_{g,n}) \to \CH^d(\Mbar_{g,n+m})\]
is injective.
\end{conjecture}
We have seen in \cite[Lemma 2.1]{BS1} that (for $m$ sufficiently large) the image of $F_m$ is open with complement of codimension $\lfloor \frac{m}{2} \rfloor +1$. So for $m \geq 2d$, we have $\CH^d(F_m(\Mbar_{g,n+m})) \cong \CH^d(\M_{g,n})$, so certainly the image of $F_m$ is sufficiently large to capture the Chow group of codimension $d$ cycles. Still, it is not true that a surjective, smooth morphism has injective pullback in Chow (if the fibres are not proper, as is the case for $F_m$), so this does not suffice to prove the conjecture. 

One aspect of the conjecture we can prove so far is the statement that if $F_{m_0}^*$ is injective, it is true that for any $m \geq m_0$ the map $F_m^*$ remains injective.
\begin{proposition} \label{pro:pullbackkerneldecreasing}
For $(g,n) \neq (1,0)$ and $0 \leq m \leq m'$ with $2g-2+n+m > 0$ we have $\ker F_{m'}^* \subseteq \ker F_m^*$. In other words, the subspaces $(\ker F_\ell^*)_{\ell}$ form a \emph{non-increasing} sequence of subspaces of $\CH^d(\M_{g,n})$.
\end{proposition}
\begin{proof}
It suffices to show the statement for $m' = m + 1$. Consider the following \emph{non-commutative} diagram.
\begin{equation} \label{eqn:noncommdiagram}
\begin{tikzcd}
 \Mbar_{g,n+m+1} \arrow[d,"\pi"] \arrow[dr,"F_{m+1}"] & \\
 \Mbar_{g,n+m} \arrow[r,"F_m"] & \M_{g,n}
\end{tikzcd}
\end{equation}
Here $\pi$ is the usual map forgetting the marking $p_{n+m+1}$ and stabilizing the curve. For this reason, the diagram is only commutative on the complement of the locus 
$$Z = \left\{(C,p_1, \ldots, p_{n+m+1}) : \begin{array}{c} p_{n+m+1} \text{ contained in}\\\text{rational component of $C$}\\\text{with $3$ special points}\end{array}\right\}\subseteq \Mbar_{g,n+m+1}\,.$$
Let $i : Z \to \Mbar_{g,n+m+1}$ be the inclusion of $Z$ and let $\alpha \in \CH^*(\M_{g,n})$ be any class. By the commutativity of the diagram \eqref{eqn:noncommdiagram} away from $Z$, we know that the class $F_{m+1}^* \alpha - \pi^* F_m^* \alpha$ restricts to zero on the complement of $Z$ and thus, by the usual excision sequence, there exists a class $\beta \in \CH^*(Z)$ such that
\[F_{m+1}^* \alpha - \pi^* F_m^* \alpha = i_* \beta\,.\]
We want to transport this to an equality of classes on $\Mbar_{g,n+m}$ by intersecting with $\psi_{n+m+1}$ and pushing forward via $\pi$. But notice that $\psi_{n+m+1}|_{Z} = 0$ since on $Z$ the component of $C$ containing $p_{n+m+1}$ is parametrized by $\Mbar_{0,3}$ and thus the psi-class of $p_{n+m+1}$ vanishes here. Thus 
\begin{equation}
\psi_{n+m+1} \cdot F_{m+1}^* \alpha - \psi_{n+m+1} \cdot \pi^* F_m^* \alpha = (i_* \beta) \psi_{n+m+1} = i_*( i^* \psi_{n+m+1} \cdot \beta) = 0\,.
\end{equation}
Pushing forward by $\pi$ and using $\pi_* \psi_{n+m+1} = (2g-2+n+m) \cdot [\Mbar_{g,n+m}]$ we obtain
\[
\pi_* \left(\psi_{n+m+1} \cdot F_{m+1}^* \alpha \right) = (2g-2+n+m) \cdot  F_m^* \alpha\,.
\]
Thus, since $2g-2+n+m \geq 2g-2+n>0$, any class $\alpha$ with $F_{m+1}^* \alpha = 0$ also satisfies $F_{m}^* \alpha=0$, finishing the proof.
\end{proof}

Again, for $g=0$ we can give some numerical evidence for the above conjecture. In Figure \ref{fig:Chowpullbackranks} we compare ranks of the Chow groups $\CH^d(\M_{0,n})$ to (lower bounds on) the ranks of $F_m^*(\CH^d(\M_{0,n}))$. We see that the bounds for $F_m^*(\CH^d(\M_{0,n}))$ increase monotonically in $m$ (as predicted by Proposition \ref{pro:pullbackkerneldecreasing}) and, in all cases which we could handle computationally, stabilize at the rank of $\CH^d(\M_{0,n})$, implying that the corresponding pullbacks $F_m^*$ are indeed injective. 
\begin{figure}[htb]
    \centering
\begin{tabular}{c|c||rccccccccc}
$(n,d)$ & $\CH^d(\M_{0,n})$ & \multicolumn{10}{c}{$F_m^*(\CH^d(\M_{0,n}))$} \\
 &  & m=0 & 1 & 2 & 3 & 4 & 5 & 6 & 7 & 8 & 9 \\
\hline 
(0,0) & 1 & &  &  & 1 & 1 & 1 & 1 & 1 & 1 & 1\\ 
(0,1) & 1 & &  &  &  & 1 & 1 & 1 & 1 & 1 & 1\\ 
(0,2) & 3 & &  &  &  &  & 1 & 2 & 3 & 3 & 3 \\ 
(0,3) & 5 & &  &  &  &  &  & 1 & 2 & {4} & 5\\ 
(0,4) & 13 & &  &  &  &  &  &  & {1} & {2} & {7}\\ 
(1,0) & 1 & &  & 1 & 1 & 1 & 1 & 1 & 1 & 1 & 1\\ 
(1,1) & 2 & &  &  & 1 & 2 & 2 & 2 & 2 & 2 & 2\\ 
(1,2) & 5 & &  &  &  & 1 & 3 & 5 & 5 & 5 & 5 \\ 
(1,3) & 12 & &  &  &  &  & 1 & 4 & 7 & 12 & 12\\ 
(2,0) & 1 & & 1 & 1 & 1 & 1 & 1 & 1 & 1 &  & \\ 
(2,1) & 3 & &  & 1 & 3 & 3 & 3 & 3 & 3 &  & \\ 
(2,2) & 9 & &  &  & 1 & 5 & 9 & 9 & 9 &  & \\ 
(2,3) & 27 &  &  &  &  & 1 & 7 & 11 & 27 &  & \\ 
(3,0) & 1 & 1 & 1 & 1 & 1 & 1 & 1 &1  &  &  & \\ 
(3,1) & 4 & & 1 & 4 & 4 & 4 & 4 &  4  &  &  & \\ 
(3,2) & 16 & &  & 1 & 5 & 15 & 16 & 16 &  &  & \\ 
(3,3) & 62 & &  &  & 1 & 5 & 16 & 62 &  &  & \\ 
\end{tabular}
    \caption{The ranks of the Chow groups $\CH^d(\M_{0,n})$ compared to (lower bounds on) the ranks of $F_m^*(\CH^d(\M_{0,n}))$; in many cases it was not feasible to obtain  the precise rank of $F_m^*(\CH^d(\M_{0,n}))$, but a lower bound could be achieved by computing the rank of the intersection pairing of $F_m^*(\CH^d(\M_{0,n}))$ with a selection of tautological classes on $\Mbar_{0,n+m}$}
    \label{fig:Chowpullbackranks}
\end{figure}

\subsection{The divisor group of \texorpdfstring{$\M_{g,n}$}{Mgn}} \label{Sect:divisors}
As for the moduli space of stable curves, the group of divisor classes on $\M_{g,n}$ can be fully understood in terms of tautological classes and relations. For $g=0$ we already saw that all divisor classes are tautological and we explicitly described the relations, so below we can restrict to $g \geq 1$. As before, we also want to exclude the case $g=1, n=0$ since $\M_{1,0}$ does not have a stratification by quotient stacks.

Thus we can restrict to the range $2g-2+n>0$, where the space $\Mbar_{g,n}$ is nonempty. Then we have an exact sequence
\begin{equation}\label{eqn:unstable}
\begin{tikzcd}
\CH_*(\Mbar_{g,n},1) \arrow[r] &\CH_*(\M_{g,n}^\textup{us}) \arrow[r] & \CH_*(\M_{g,n})\arrow[r] & \CH_*(\Mbar_{g,n})\arrow[r] \arrow[l, bend right, dotted, "\st^*", swap] & 0
\end{tikzcd}
\end{equation}
where $\M_{g,n}^\textup{us}$ is the \emph{unstable locus} of $\M_{g,n}$, i.e. the complement of the open substack $\Mbar_{g,n} \subset \M_{g,n}$.
Using this sequence, we can completely understand $\CH^1(\M_{g,n})$ from the explicit description of $\CH^1(\Mbar_{g,n})$ in  \cite[Theorem 2.2]{arbarellocornalba}. 
\begin{proposition} 
For $(g,n) \neq (1,0)$ we have $\R^1(\M_{g,n})=\CH^1(\M_{g,n})$. Furthermore, for $2g-2+n>0$, all tautological relations in $\R^1(\M_{g,n})$ are pulled backed from relations in $\R^1(\Mbar_{g,n})$ via the stabilization morphism.
\end{proposition}
\begin{proof}
As discussed before, for the statement that all divisor classes are tautological we can restrict to the stable range $2g-2+n>0$, since the case of $g=0$ was treated before. For the moduli spaces of stable curves it holds that $\R^1(\Mbar_{g,n})$ and $\CH^1(\Mbar_{g,n})$ coincide by \cite{arbarellocornalba}. 
Since the locus $\M^\textup{us}_{g,n}$ is a union of boundary divisors, whose fundamental classes are pushforwards of appropriate gluing maps, the image of the pushforward map $\CH^0(\M^\textup{us}_{g,n})\to \CH^1(\M_{g,n})$ is contained in  $\R^1(\M_{g,n})$. Therefore the excision sequence (\ref{eqn:unstable}) gives the conclusion.

To compute the set of relations, we observe that
$$\CH^1(\Mbar_{g,n},1)\cong H^0(\Mbar_{g,n}, \mathcal{O}^\times_{\Mbar_{g,n}})=k^\times$$
because $\Mbar_{g,n}$ is smooth and projective over $k$. Therefore, in the degree $1$ part of the sequence \eqref{eqn:unstable}, the image of the connecting homomorphism is trivial. Thus, since the pullback $\st^*$ by the stabilization morphism defines a splitting of \eqref{eqn:unstable} on the right, we have that 
\[\CH^1(\M_{g,n}) = \bigoplus_{\substack{\Gamma\text{ unstable}\\|E(\Gamma)|=1}} \mathbb{Q} \cdot [\Gamma] \oplus \CH^1(\Mbar_{g,n}).\]
Thus all relations between decorated strata classes in codimension $1$ are pulled back from $\Mbar_{g,n}$.
\end{proof}
\subsection{Zero cycles on \texorpdfstring{$\M_{g,n}$}{Mgn}} \label{Sect:zerocycles}
After treating the case of  codimension $1$ cycles in the previous section, we want to make some remarks about cycles of \emph{dimension} $0$. For the moduli spaces of stable curves, these exhibit many interesting properties:
\begin{itemize}
    \item In \cite{grabervakiltaut}, Graber and Vakil showed  that the group $\R_0(\Mbar_{g,n})$ of tautological zero cycles on $\Mbar_{g,n}$ is always isomorphic to $\mathbb{Q}$, even though the full Chow group $\CH_0(\Mbar_{g,n})$ can be infinite-dimensional (e.g. for $(g,n)=(1,11)$, see \cite[Remark 1.1]{grabervakiltaut}).
    \item In \cite{zerocycles}, Pandharipande and the second author presented geometric conditions on stable curves $(C,p_1, \ldots, p_n)$ ensuring that the zero cycle $[(C,p_1, \ldots, p_n)]$ in $\Mbar_{g,n}$ is tautological.
\end{itemize}
We want to note here, that for the moduli stacks of prestable curves, the behaviour of tautological zero cycles becomes more complicated:
\begin{itemize}
    \item For $\M_{0,n}$ with $n=0,1,2$, we have $\R_0(\M_{0,n}) = \CH_0(\M_{0,n})=0$ for dimension reasons.
    \item As visible from Figure \ref{fig:Chowranks}, the group $\R_0(\M_{0,n})$ is no longer one-dimensional for $n \geq 4$. Indeed, looking at the example of $n=4$ we note that the boundary divisor of curves with one component having no marked points is a nonvanishing zero cycle (since it pulls back to an effective boundary divisor under the forgetful map $F_2 : \Mbar_{0,6} \to \M_{0,4}$), but it restricts to zero on $\Mbar_{0,4} \subset \M_{0,4}$ and is thus linearly independent of the generator of $\R_0(\Mbar_{0,4})$.
\end{itemize}
This indicates that for the moduli stacks of prestable curves, the group of zero cycles plays less of a special role than for the moduli spaces of stable curves.

\newpage
\appendix

\section{Gysin pullback for higher Chow groups} \label{Sect:Gysinpullbackhigher}
In \cite{DegliseJinKhan2018}, D\'eglise, Jin and Khan generalized Gysin pullback along a regular imbedding to motivic homotopy theories. We summerize the construction in the language of higher Chow groups. For a moment, let $X$ be a quasi-projective scheme over $k$ and we consider higher Chow groups with $\ZZ$-coefficients. For simplicity, we write $\GG_m X = X\times \GG_m$. Let $[t]$ be a generator of
\[\CH_0(\GG_m,1)\cong \left(k[t,t^{-1}]\right)^\times\]
and let 
\[\gamma_t \colon \CH_*(X,m)\to \CH_*(\GG_m X,m+1)\,, \; \alpha \mapsto \alpha\times [t]\]
be the morphism defined by the exterior product. Let $i\colon Z\to X$ be a regular imbedding of codimension $r$ and let $q : N_Z X \to Z$ be the normal bundle. Let $D_Z X$ be the Fulton--MacPherson's deformation space defined by
\[D_Z X = \textup{Bl}_{Z\times 0}(X\times \AA^1) - \textup{Bl}_{Z\times 0}(X\times 0)\]
which fits into the cartesian diagram
\begin{equation*}
    \begin{tikzcd}
     N_Z X \arrow[r]\arrow[d] & D_Z X \arrow[d] & \GG_m  X\arrow[l]\arrow[d]\\
     \{0\} \arrow[r] & \AA^1 &\arrow[l] \GG_m\,.
    \end{tikzcd}
\end{equation*}
By \cite{bloch_higherChow, BlochMoving} we have the localization sequence
\begin{equation}\label{eqn:localizationHigher}
    \ldots \to \CH_d(\GG_m X,m+1)\xrightarrow{\partial} \CH_d(N_Z X,m)\to \CH_d(D_Z X,m) \to \ldots\,.
\end{equation}
\begin{definition}
For a regular imbedding $i\colon Z\to X$, we define 
\[i^*\colon \CH_d(X,m)\to \CH_{d-r}(Z,m)\]
as a composition of following morphisms
\begin{equation}
    \CH_d(X,m)\xrightarrow{\gamma_t} \CH_d(\GG_m X, m+1)\xrightarrow{\partial} \CH_d(N_Z X,m) \xrightarrow[\cong]{(q^*)^{-1}} \CH_{d-r}(Z,m)
\end{equation}
where $\partial$ is the boundary map in \eqref{eqn:localizationHigher} and the flat pullback $q^*$ is an isomorphism by \cite{bloch_higherChow}.
\end{definition}
This definition extends the Gysin pullback for Chow groups in the sense that it coincides with the Gysin pullback defined in \cite{Fulton1984Intersection-th} when $m=0$. This construction extends to all lci morphism and satisfies functoriality, transverse base change and excess intersection formula, see \cite{DegliseJinKhan2018}.

Given a line bundle $q: L \to X$ with the zero section $0: X \to L$, the action of the first Chern class on higher Chow groups can be defined by
\[c_1(L)\cap \colon \CH_d(X,m)\xrightarrow{0_*} \CH_d(L,m)\xrightarrow[\cong]{(q^*)^{-1}} \CH_{d-1}(X,m)\,.\]
We want to note two basic compatibilities of this operation: firstly, given a proper morphism $f : X' \to X$ and the line bundle $L \to X$, a short computation shows the projection formula
\begin{equation} \label{eqn:firstChernhigherprojectionformula}
f_* \left(c_1(f^*L) \cap \alpha \right) =   c_1(L) \cap f_* \alpha \text{ for }\alpha \in \CH_*(X',m)\,.
\end{equation}
Secondly, intersecting higher Chow cycles with a Cartier divisor has the same formula as for ordinary Chow groups.
\begin{lemma}\label{lem:intersectDivisor}
Let $i\colon D\to X$ be an effective divisor and let $q\colon \CO(D)\to X$ be the associated line bundle. Then 
\begin{equation}\label{eqn:intersectDivisor}
    i_*i^*\alpha = c_1(\CO(D))\cap \alpha\,,  \,\alpha\in \CH_d(X,m)\,.
\end{equation}
\end{lemma}
\begin{proof}
Consider the following cartesian diagram
\begin{equation*}
\begin{tikzcd}
D\arrow[r,"i"]\arrow[d,"i"] & X\arrow[d,"s"]\\
X\arrow[r,"0"] & \CO(D)
\end{tikzcd}
\end{equation*}
where $s\colon X\to \CO(D)$ be the regular section defining $D$ and $0$ is the zero section. Recall that the action of the first Chern class can be defined by
\[c_1(\CO(D))\cap - \colon \CH_d(X,m)\xrightarrow{0_*} \CH_d(\CO(D),m)\xrightarrow[\cong]{(q^*)^{-1}} \CH_{d-1}(X,m)\,.\]
By the transverse base change formula \cite[Proposition 2.4.2]{DegliseJinKhan2018},
\[i_*i^*\alpha = s^*0_*\alpha\,.\]
Now we can conclude the result because $s^*$ is an inverse of $q^*$.
\end{proof}
The above construction can be extended to global quotient stacks without any difficulty. Let $X$ be an equidimensional quasi-projective scheme with a linearized $G$-action. Applying the Borel construction developed in \cite{EdidinGraham98} yields the definition of higher Chow groups for $[X/G]$, see \cite{Krishna13HigherChow}. For an arbitrary algebraic stack, the authors do not know whether a direct generalization of \cite{Fulton1984Intersection-th} is possible. Relying on the recent development of motivic homotopy theories, Khan used the six-operator formalism (\cite{KhanThesis}) to construct motivic Borel--Moore homology for derived algebraic stacks in  \cite{khan2019virtual}.

\bibliographystyle{alpha}
\bibliography{main}
\end{document}